\documentclass[leqno,12pt]{amsart} 
\usepackage{amssymb, mathrsfs, amsmath, tabu, amsthm}
\usepackage{hyperref}
\usepackage{tikz-cd, graphicx} 
\usepackage{pdfpages}
\usepackage{enumitem}
\theoremstyle{plain} 
\newtheorem{theorem}{\indent\sc Theorem}[section]   
\newtheorem{lemma}[theorem]{\indent\sc Lemma}

\newtheorem{proposition}[theorem]{\indent\sc Proposition}

\theoremstyle{definition} 

\newtheorem{remark}[theorem]{\indent\sc Remark}

\usepackage[active]{srcltx}
\usepackage{pdfsync}
\usepackage{tabularx} 
\usepackage{arydshln}

\newcommand{\mcL}{\mathcal{L}}

\newcommand{\FF}{\mathbb{F}}
\newcommand{\ZZ}{\mathbb{Z}}

\def\dim{\mathop{\hbox{\rm dim}}}
\def\tr{\mathop{\rm tr}}
\newcommand{\sg}{\mathrm{sgn}}
  
\newcommand{\spf}{\mathfrak{sp}}

\newcommand{\sof}{\mathfrak{so}}
\newcommand{\slf}{\mathfrak{sl}}

\newcommand{\id}{\mathrm{id}}
\newcommand{\mm}{\mathfrak{m}}
\newcommand{\hh}{\mathfrak{h}}
\newcommand{\g}{\mathfrak{g}}
\newcommand{\e}{\mathfrak{e}}
\def\ll#1{\llcorner\widetilde{#1}}

\title[  Models of $\e_8$]{  Linear models \\of the exceptional Lie algebra $\e_8$}

\author[Y.~Cabrera, C.~Draper and A.~Garv\'\i n]{  
Yolanda Cabrera${}^\star$, Cristina Draper${}^*$  and Antonio Garv\'\i n${}^\star$ 
}  

\thanks{Supported       by the Spanish Ministerio de Ciencia e Innovaci\'on   through projects  PID2019-104236GB-I00/AEI/10.13039/501100011033 (first and second authors)  and PID2020-118452GB-I00 (second and third authors),   with FEDER funds; and by
Junta de Andaluc\'{\i}a  through projects  FQM-336, UMA18-FEDERJA-119 (first and second authors) and FQM-213 and PROYEXCEL-00827 (third author). 
}

\address{%
${}^*$  Departamento de \'Algebra, Geometr\'\i a y Topolog\'\i a,   \endgraf
${}^\star$ Departamento de Matem\'atica Aplicada, \endgraf
Universidad de M\'{a}laga, 
 29071 M\'{a}laga,  
Spain
}
\email{yolandacc@uma.es, cdf@uma.es, garvin@uma.es.}
\begin{document}


\maketitle

\begin{abstract}
This work provides five   explicit constructions of the exceptional Lie algebra $\e_8$, 
based on its semisimple subalgebras of maximal rank. 
Each of these models is graded 
by an abelian group, 
namely, 
$\mathbb{Z}_4$, $\mathbb{Z}_5$,  $\mathbb{Z}_6$, $\mathbb{Z}_3^2$ and $\mathbb{Z}_2\times\mathbb{Z}_4$;
the neutral component is direct sum of special linear algebras and the remaining homogeneous components are irreducible modules 
for it.
\end{abstract}


\section{Introduction }

The second author gave in \cite{ModelosF4}   a description of the exceptional Lie algebra $\mathfrak{f}_4$ for each of its semisimple subalgebras of maximal rank, joint with a constructive procedure which can be applied to any simple Lie algebra and any reductive subalgebra of maximal rank, with the help of its
 irreducible modules. The main idea in it was that there always exists a grading by a (finitely generated) abelian group of the algebra such that the neutral homogeneous component is the given reductive subalgebra and the remaining homogeneous components are irreducible modules for it. Furthermore, the products between two components are essentially known up to some scalars, which can be determined 
 by imposing the constructed graded algebra to satisfy the Jacobi identity. The idea was not new, but it generalized a method of  getting models of simple Lie algebras based on gradings on cyclic groups, which was quite well known. Of course this method has been  used mostly for constructing exceptional Lie algebras, whose relevance (and amount of appearances) is  beyond question but they are usually difficult to handle. 
 On one hand,  their root decompositions play a key role in the theory but 
  sometimes they are not very practical when you really want to use them:
 the product of two (not opposite) root spaces $[\mcL_{\alpha},\mcL_{\beta}]$ fills the one-dimensional space $\mcL_{\alpha+\beta}$ but, at the end, 
 one has to know which is the scalar that gives such bracket, and the process to have explicitly all the scalars is very long and far from trivial.  
  On the other hand, there is a wonderful    unified construction of all the exceptional Lie algebras (different from $G_2$)  due to Tits \cite{Tits}, which does not use gradings but two ingredients;   the
exceptional simple Jordan algebra, the Albert algebra, jointly with a composition algebra.  
But Jordan algebras are not particularly easy,  and quite often one is interested in a particular symmetry of the object under study, that is, in a 
 concrete grading on the algebra. This is part of the reason why the models of the exceptional Lie algebras based on gradings 
 and on multilinear algebra have a long tradition, their high symmetry, 
 in addition to the fact  that they provide  an easier approach to these algebras.

 A classical reference for constructions of exceptional Lie algebras is  the encyclopedia \cite{EMSIII}: \cite[Chapter~5, \S1]{EMSIII} is devoted to models coming from the octonion algebra   (for instance, those coming from the Albert algebra), but \cite[Chapter~5, \S2]{EMSIII} develops models associated with gradings, although, as mentioned, it only deals with gradings over cyclic groups and not necessarily finite.  
 Another reference which has been very inspiring through this paper is \cite{Adams}, a posthumous  book about Adams' lecture notes taken on his lectures on exceptional Lie groups given at Cambridge University. All the complex exceptional Lie algebras are constructed  in it, as well as some of their real forms. Finally, a third source, probably the most accesible, is the section of algebraic constructions of the exceptional Lie algebra in  \cite[\S22.4]{FultonHarris}. The \emph{linear models} of the Lie algebra $\e_8$ developed in the above references are:
 \begin{itemize}
\item The $\ZZ_2$-grading on $\e_8$ with even part isomorphic to the orthogonal algebra $\sof(16)$ in \cite[Theorem~6.3]{Adams}. Here the odd homogeneous component is the half-spin representation for the action of $\mcL_{\bar0}$. 
\item The $\ZZ_3$-grading on $\e_8$ with neutral homogeneous component isomorphic to the special linear algebra $\slf(9)$. This   construction goes to Freudenthal \cite{Freudenthal_laZ3} in the fifties, but, due to its relevance, it can also be found in   
\cite[p.~181]{EMSIII} or in \cite[p.~360-361]{FultonHarris}. 
Thus, starting with $V$ a 9-dimensional vector space, $\e_8$ is taken to be the sum of the space of traceless endomorphisms  of $V$, the space of trivectors of $V$ and the space of cotrivectors.
In Chevalley words,  \lq\lq the bracket operation is given in a
very condensed form which exhibits a large degree of symmetry\rq\rq.
\end{itemize}  
In fact,  \cite[p.~180-181]{EMSIII} also gives: 
a $\ZZ_2$-grading on $E_7$ with fixed part of type $\mathfrak{sl}(8)$, 
a $\ZZ_2$-grading on $E_6$ with fixed part of type $\mathfrak{sp}(8)$, some $\ZZ$-gradings on $E_6$, $E_7$ and $E_8$ with fixed parts the general linear algebras
$\mathfrak{gl}(6)$, $\mathfrak{gl}(7)$ and $\mathfrak{gl}(8)$ respectively. The scalars are shown without proofs, computed by Katanova,
 appearing scalars as 40320 or 1080.   
(See Remark~\ref{re_Z3} to contrast how  suitable descriptions of the invariant products can make the scalars to be $\pm1$.) 
 
 Thus, the procedure is in some way folklore: the $\ZZ_m$-gradings  are produced by order $m$ automorphisms of the (simple) Lie algebra classified in  the book \cite[Chapter~8]{Kac}, which also describes some key properties of the homogeneous components of the grading as modules for the neutral component (see Section~\ref{se_gradings}). 
 The situation in \cite{ModelosF4} is, as mentioned, slightly more general: it can be applied for reductive subalgebras of maximal rank and the groups which appear are not necessarily cyclic. The main result there is inspired from \cite{Bourbaki}:

   \begin{theorem}\label{teo_main} (\cite{ModelosF4})   
Let $\mcL$ be a (finite-dimensional) simple Lie algebra over an algebraically closed  field $\FF$ of zero characteristic, $\Phi$ a root system of $\mcL$
relative to a Cartan subalgebra $H$ of $\mcL$,
and $\Phi'$ a subset of $ \Phi$ verifying $-\Phi'=\Phi'$ and $(\Phi'+\Phi')\cap\Phi=\Phi'$.
Let $G$ be the abelian group $\ZZ \Phi/\ZZ\Phi'$. Then:
\begin{description}
\item[a)] $\Phi\cap\ZZ\Phi'=\Phi'$.
\item[b)] $\mcL=\oplus_{g\in G} \mcL_{g}$ is $G$-graduated, being
$\mcL_{e}=\mathfrak h=H\oplus\sum_{\alpha\in\Phi'}  L_{\alpha}$ a reductive subalgebra
and, for any $e\ne g\in G$, either $\mcL_{g}=0$ or $\mcL_{g}$ is an $\mathfrak h$-irreducible module.
\item[c)] $G$ is the smallest group verifying the property b).
\item[d)] If $g_1,g_2\in G\setminus\{0\},\,g_1+g_2\ne0$, $\mathrm{Hom}_{\mathfrak h}(\mcL_{g_1}\otimes\mcL_{g_2},\mcL_{g_1+g_2})=\FF[\ ,\ ]\mid_{\mcL_{g_1}\otimes\mcL_{g_2}}$.
\item[e)] If $g_1,g_2\in G,\,g_1+g_2\ne0$, $[\mcL_{g_1},\mcL_{g_2} ]=\mcL_{g_1+g_2}$.
\item[f)] $\mathfrak h$ is semisimple if and only if $G$ is a finite group.
\item[g)] $\mcL_{g}\ne0$ for all $g\in G$ if and only if the bracket of any two irreducible components of $\mcL$ not contained in $\mathfrak h$ is not zero. In this case $\mathfrak h$ is semisimple.
\end{description}
\end{theorem}

We are interested in  getting some new constructions of the largest of the exceptional Lie algebras, $\e_8$, by applying just the above result.
A nice survey of classical and recent results on the exceptional
objects of type $E8$ including the Lie algebra is \cite{todoE8}. 
Section~4  in such work just recalls some   constructions of $\e_8$   via gradings, what Garibaldi calls \emph{a $\mcL_{\bar0}$-construction of $\mcL$}.
He also speaks  about  the most important examples (the $\ZZ_2$ and the $\ZZ_3$-gradings above), and  mentions   the $\ZZ_5$ and $\ZZ_2^2$-gradings too, although without showing complete constructions involving the scalars. What he encloses is a huge amount of references (170!) and many interesting connections and applications.

The list of semisimple subalgebras of $\mcL=\e_8$ to which Theorem~\ref{teo_main} can be applied    is   the following list, where we write $(G,\mathfrak h$) for $G$ the corresponding grading group and $\mathfrak h$ the semisimple subalgebra of $\mcL$:
$$
\begin{array}{c}
(\ZZ_3,A_8),\quad (\ZZ_4,A_7\oplus A_1),\quad (\ZZ_6,A_5\oplus A_2\oplus A_1),\quad  (\ZZ_5,2A_4),\vspace{3pt}\\
 (\ZZ_2,D_8),\quad (\ZZ_4,D_5\oplus A_3),\quad (\ZZ_3,E_6\oplus A_2),\quad (\ZZ_2,E_7\oplus A_1),
\end{array}
$$
if the group is cyclic, and 
$$
(\ZZ_2^2,D_6\oplus 2A_1),\  (\ZZ_2^2,2D_4),\  (\ZZ_4\times\ZZ_2,2A_3\oplus 2A_1), 
(\ZZ_3^2,4A_2),\ (\ZZ_2^3,D_4\oplus 4A_1),\ (\ZZ_2^4,8A_1),
$$
otherwise. 
For any of these pairs, $\e_8$ can be constructed as a $G$-graded algebra with $\mcL_e=\mathfrak h$, and surprisingly the   amount of pieces  breaks neither the symmetry nor the elegance of the constructions.
All of these constructions were presented in a very concise form in a conference in Spain in 2005, \cite{pos}. Now we think that it is convenient to present and complete these results, which could be  potentially useful to the mathematical community. We will focus on the cases in which all the  simple ideals of the neutral component  are of type $A$, because the actions on their irreducible modules are considerably easier to describe. 
This is the situation for 7 semisimple subalgebras of $\e_8$ of rank 8, according to the above list. 

Among them, Alberto Elduque provides the model of $\e_8$ related to the pair $(\ZZ_2^4,8A_1)$ in his paper about magic squares and symmetric composition algebras \cite{Alb_ylosA1}. In it, he applies the $\ZZ_2^2$ and $\ZZ_2^3$-grading on the para-Hurwitz algebras to construct all the
exceptional Lie algebras   by using only copies of  the symplectic algebra $\mathfrak{sp}(V)$ and   tensor products of copies of $V$, for $V$ a two dimensional vector space with an alternating form. In the case of $\e_8$, that construction  has neutral homogeneous component   $\mcL_e\cong 8 \,\mathfrak{sp}(V)$ and all the 14 remaining nonzero homogenous components   are $\mcL_e$-modules isomorphic to $V^{\otimes 4}$. The scalars which give the multiplication of any pair homogeneous components are computed in \cite[Table~5]{Alb_ylosA1}. Thus, we will concentrate our efforts on giving a complete description of the     models based on the remaining 5 subalgebras (apart from   $A_8$): those ones  of types $A_7\oplus A_1$  (\S\ref{se_A1A7}), $A_5\oplus A_2\oplus A_1$  (\S\ref{seZ6}), $2A_4$ (\S\ref{se_2A4}), $4A_2$  (\S\ref{se_4A2}) and $2A_3\oplus 2A_1$  (\S\ref{se_2A12A3}), each one with its own beauty. 
Although the procedure may not be novel, the results are. The main difficulty to be able to calculate all scalars in an explicit way has been to fix an appropriate notation in Section~\ref{se_invariantproducts}. 
 In fact, it is interesting to contrast how the scalars involved in our models are easier than those already mentioned in the literature, despite the fact that the groups are evidently more complicated. 

To emphasize the importance of having a determined model, and to have a panoramic view of the best known models, let us recall some of the appearances   in the literature of the remaining cases in the list.
 The (complex) model related to the pair $(\ZZ_2,E_7\oplus A_1)$ has been used in \cite{STS} to describe the real symplectic triple systems whose standard enveloping algebras have  type $\e_{8,8}$ or $\e_{8,-24}$. Examples of irreducible Lie-Yamaguti algebras are given in \cite{Fabi1} in connection to this model.  
 The beautiful construction of $\e_8$ based on $(\ZZ_2^2,2D_4)$ appears in \cite[Eq.~(25)]{Baez}. The scalars are not computed there because \cite{Baez} is a nice story connecting ideas and fields, but not including proofs, at most, hints of them. A complete construction, including the brackets, appears in \cite[\S3]{magic1}, where $\e_8$ is constructed from two copies of the para-octonion algebras, and note that the triality Lie algebra of the para-octonion algebra is just of type $D_4$. 
 Also, these constructions in \cite{magic1} provide other realizations in our list. To be precise, the construction of $\e_8$ based on the pair $(\ZZ_3,E_6\oplus A_2)$: Elduque  builds $\e_8$ from a para-octonion algebra and an Okubo algebra too (in fact, it can be constructed from any two symmetric composition algebras of dimension 8), and then the natural $\ZZ_3$-grading on the Okubo algebra is trivially extended to the whole $\e_8$. (Similarly, when he uses twice the Okubo algebra to construct $\e_8$, the extension to $\e_8$ of the $\ZZ_3$-grading on the Okubo algebra gives again the pair  $(\ZZ_3,  A_8)$.) 
 The whole work \cite{magic1} is a reinterpretation of the Freudenthal magic square by using two symmetric composition algebras for a unified construction of all the exceptional Lie algebras, instead of using the best known composition algebras. This other approach is taken by \cite{BartonSudbery}, where Barton and Sudbery provide  descriptions of the exceptional Lie superalgebras, by   using the   triality principle too. (A similar construction is given by Landsberg and Manivel in \cite{LM1}.)

Finally, let us mention one of the applications of our approach to construct an algebra via its gradings. The graded contractions are introduced in \cite{Contracciones} as an alternative to   In\"on\"u-Wigner contractions  in Physics. Instead of a limit of sequences of mutually isomorphic Lie algebras (\emph{continuos contractions}), the commutation relations among graded subspaces are multiplied by contraction parameters. This provides a system of quadratic equations resulting from the Jacobi identities. Of course, it is not difficult to read our results from this viewpoint. An example of work in this line is \cite{contracgraduadasdePauli}, which studies the algebras non isomorphic to $\slf(3,\mathbb C)$ which appear when deforming its fine $\ZZ_3^2$-grading given by the Gell-Mann matrices. This can be applied to find nilpotent and solvable Lie algebras with properties, as in \cite{contractionsg2}, dealing with graded contractions of the $\ZZ_2^3$-graded exceptional Lie algebra $\mathfrak{g}_2$.
Physical motivated cases appear in \cite{motivfisica}. \smallskip

 Through the paper, the considered ground field $\FF$ will be algebraically closed field of zero characteristic, although these models can be extended to   Lie algebras over other fields.


\section{Models based on subalgebras of type $A$}

\subsection{Invariant products}\label{se_invariantproducts}
The main reference used here for the results in multilinear algebra is  \cite[Appendix~B]{FultonHarris} (see also \cite[Chapter~3]{BourbakiAlg}).  
This material  is well known, but fixing the notation is crucial for our purpose of giving some unified models of $\e_8$.

If $V$ is a   vector space over $\FF$ of dimension $n$, and $\mcL=\slf(V)$ denotes the simple Lie algebra of traceless matrices, of rank $n-1$, recall that $\bigwedge^rV$ is an $\mcL$-irreducible module for any $r=0,\dots, n$,
where the action of $f\in\slf(V)$ is given by 
$
f\cdot v_1\wedge\dots\wedge v_r=\sum_{i=1}^r v_1\wedge\dots\wedge f(v_i)\wedge\dots \wedge v_r
$.  
More precisely, $\bigwedge^rV$ is the $\mcL$-module of fundamental weight $\varpi_r$  if $1\le r<n$ \cite{Humphreysalg}.   
(If $r> n$, what happens is $\bigwedge^rV=0$. And, if either $r=0$ or $r=n$, then $\bigwedge^0V\cong\FF$ and  $\bigwedge^rV\cong\FF$ are one-dimensional trivial modules.)
Also, the dual module $V^*$ has fundamental weight $\varpi_{n-1}$ and, more generally, $\bigwedge^rV^*$ has fundamental weight $\varpi_{n-r}$ if $1\le r<n$.
The adjoint module has as highest weight $\varpi_{1}+\varpi_{n-1} $. In fact, $V(\varpi_{1})\otimes V( \varpi_{n-1})\cong V(\varpi_{1}+\varpi_{n-1})\oplus V(0)$ because
  $\Psi\colon V^*\otimes V\to \mathfrak{gl}(V)$  given as
 $\Psi(\alpha\otimes v)(w)=\alpha(w)v$ for any $v,w\in V$ and any $\alpha\in V^*$,  is
 an $\mcL$-module isomorphism.

 Take some concrete $\mcL$-invariant maps  among these modules. 
 First, the wedge product
\begin{equation} \label{eq_wed}
\textstyle
\bigwedge^{r}V\times \bigwedge^{s}V\to \bigwedge^{r+s}V,\qquad
(v_1\wedge\dots\wedge v_r,v_{r+1}\wedge\dots\wedge v_{r+s})\mapsto v_{1}\wedge\dots\wedge v_{r+s}
\end{equation} 
is an   $\mcL$-invariant map (equal to zero if $r+s>n$).
Second, for any $r=0,\dots,n$, we have the non-degenerate bilinear form
\begin{equation*} 
\textstyle
\langle\ ,\ \rangle_r\colon \bigwedge^rV\times\bigwedge^{r}V^*\to \FF,
\end{equation*}
  given by
$$
\langle v_1\wedge\dots\wedge v_r,\varphi_1\wedge\dots\wedge
\varphi_r\rangle_r:=\det(\varphi_i(v_j))=\sum_{\sigma\in S_r}\textrm{sgn}(\sigma)\varphi_{\sigma(1)}(v_1)\dots \varphi_{\sigma(r)}(v_r),
$$
again an $\mcL$-invariant map. 
Here, $S_r$ denotes the usual symmetric group or permutation group.
Also, 
if $s,t\in\FF$, we are taking $\langle s,t \rangle_0=st$. We  simply use the notation  $\langle\ ,\ \rangle$ in case  there is no ambiguity. 
This pairing provides the identification 
\begin{equation}\label{eq_pairing}
\textstyle
\bigwedge^{r}V^*\to(\bigwedge^rV)^*,\qquad 
\alpha\in  \bigwedge^{r}V^* \mapsto \langle -,\alpha\rangle  \colon\bigwedge^{r}V\to\FF.
\end{equation}
Now
fix any  non-zero linear map 
\begin{equation}\label{eq_fi}
\textstyle
\phi\colon\bigwedge^nV\stackrel{\cong}\longrightarrow \FF,
\end{equation}
which clearly is an isomorphism between trivial $\mcL$-modules. 
It provides another $\mcL$-module isomorphism,
\begin{equation}\label{eq_isocondual}
\begin{array}{lcl}
\bigwedge^{n-r}V&\to&(\bigwedge^{r}V)^*\\
x&\mapsto&[y\in\bigwedge^{r}V\mapsto\phi(x\wedge y)\in
\FF].\end{array}
\end{equation}
Here the operation $\wedge$ refers to the map in Eq.~\eqref{eq_wed}. Be careful with this symbol, which can lead to confusion, since
$\wedge$ is not necessarily skew-symmetric, but
$x\wedge y=(-1)^{r(n-r)}y\wedge x $.
Now compose the map in  Eq.~\eqref{eq_isocondual} with the inverse of the map in Eq.~\eqref{eq_pairing}
to get the $\mcL$-module isomorphism
\begin{equation*}\label{eq_defdetilde}
\begin{array}{lcl}
\bigwedge^{n-r}V&\equiv&\bigwedge^{r}V^*\\
x&\mapsto&\tilde x\end{array}
\end{equation*}
given by
$\langle y,\tilde x\rangle=\phi(x\wedge y) $
for any $x\in \bigwedge^{n-r}V$ and $y\in \bigwedge^{r}V$.
\smallskip  

 There are related contraction maps, also called \emph{internal products}, denoted by $\lrcorner$ and $\llcorner$. 
 For $p$ and $q$ positive integers with $p+q\le n$, the contraction
$$
\begin{array}{rclcl}
\bigwedge^pV&\otimes&\bigwedge^{p+q}V^*&\to&\bigwedge^qV^*\\
x&\otimes&\alpha&\mapsto &x \,\lrcorner\,\alpha\end{array}$$
is determined by
$
\langle z,x \,\lrcorner\,\alpha \rangle=\langle z\wedge x,\alpha \rangle,
 $
for any $z\in\bigwedge^qV$. Similarly, the contraction
$$
\begin{array}{rclcl}
\bigwedge^{p+q}V&\otimes&\bigwedge^{p}V^*&\to&\bigwedge^qV\\
x&\otimes&\alpha&\mapsto &x \,\llcorner\,\alpha\end{array}
$$
is determined by
$$
\langle x \,\llcorner\,\alpha,\beta\rangle=\langle x,\alpha\wedge \beta\rangle,
$$
for any 
$\beta\in\bigwedge^qV^*$.  
We will make an extensive use of $\llcorner$, so that 
we should have some way in mind to handle it. 
For $v_i\in V$, $\varphi_i\in V^*$,  $i\le p+q$, we have
$$ 
(v_1\wedge\dots\wedge
v_{p+q})\,\llcorner\,(\varphi_1\wedge\dots\wedge\varphi_{p})=
\sum_{\sigma\in \hat S_{p+q}}\sg(\sigma)\varphi_1(v_{\sigma(1)})
\dots\varphi_p(v_{\sigma(p)})\,v_{\sigma(p+1)}\wedge\dots \wedge
v_{\sigma(q+p)}
$$
where $\hat S_{p+q}$ denotes the set of permutations of $\{1,\dots,p+q\}$ that preserve the order of $\{p+1,\dots,p+q\}$ (the so called \emph{shuffles}). 
  Observe that if $q=0$, then $x \,\llcorner\, \alpha=x \,\lrcorner\, \alpha=\langle x,\alpha \rangle $, for any  
$x\in \bigwedge^{p}V$ and $\alpha\in\bigwedge^{p}V^*$.

   With these tools, we introduce a product on the exterior algebra $\bigwedge V=\oplus_{i=0}^n \bigwedge^i V$
 which does not always coincide with the exterior product or wedge product.  
 We will use it throughout the manuscript to describe unifiedly different  models of $\e_8$.
Thus, define the $\mcL$-invariant bilinear  map
$$
 \begin{array}{cccl}
*\colon&\bigwedge V\times\bigwedge V&\to&\bigwedge V\\
&(x,y)&\mapsto &x*y,\end{array}
$$
where, if  $x\in\bigwedge^i V$ and $y\in\bigwedge^j V$,
\begin{equation}\label{eq_starproduct}
x*y:=\left\{\begin{array}{ll}
x\wedge y \quad&\text{ if }i+j< n,\\
x\,\llcorner\,\tilde y \quad&\text{ if }i+j\ge
n.\end{array}\right.  
\end{equation}

\noindent 
Note that $x*y\in\bigwedge^{[i+j]_n}V$, where we denote by  $[p]_n $  the element in
$\{0,1,\dots,n-1\}$ congruent with $p\in\mathbb N$ modulo $n$.

This operation $*$ gives  a remarkable $\mcL$-invariant map in case $i+j=n$, since 
  for $x\in \bigwedge^{i}V$, $y\in \bigwedge^{n-i}V$, we have
  $x*y=x \, \llcorner \, \widetilde{y} = \langle x, \widetilde{y} \rangle = \phi(y \wedge x)\in\FF$. 
   We denote the related bilinear map given by this restriction of $*$ by   
\begin{equation}\label{eq_laforma} 
\textstyle   
(\ ,\ )\colon\bigwedge^i V\times\bigwedge^{n-i} V\to \FF,\qquad (x,y):=x*y=\phi(y\wedge x).
\end{equation}
Note that $(x,y)=-(y,x)$ if both $i$ and $n-i$ are odd, while $(x,y)=(y,x)$ otherwise.
Next, in order to consider a second   $\mcL$-invariant map   
$
[\ ,\ ]\colon\bigwedge^i V\times\bigwedge^{n-i} V\to
\slf(V)
$, we need to
 recall some facts about the dualization of an action, essentially extracted from  
\cite{Adams}.

\begin{remark}\label{re_dualizar}
If $\mcL$ is a Lie algebra and $U$ and $W$ are   $\mcL$-modules,
denote by $U^\mcL:=\{u\in U: x\cdot u=0\ \forall x\in\mcL\}$, the trivial submodule of $U$. It is easy to check that 
$\mathrm{Hom}_\mcL(U,W)=\mathrm{Hom}(U,W)^\mcL$.
Now consider the usual $\mcL$-module isomorphism $  U^*\otimes W\to \mathrm{Hom}(U,W)$ given by  $ \alpha\otimes w\mapsto (u\mapsto \alpha(u)w)$ for any $u\in U$, $w\in W$ and $\alpha\in U^*$, whose restriction to the trivial submodule allows to identify $  (U^*\otimes W)^\mcL$ with $\mathrm{Hom}_\mcL(U,W)$. Take also in mind the    identifications
\begin{equation}\label{eq_dua}
\mathrm{Hom}(\mcL\otimes W,W)\cong (\mcL\otimes W)^*\otimes W\cong (W\otimes W^*)^*\otimes\mcL^{*}\cong 
\mathrm{Hom}(W\otimes W^*,\mcL^*).
\end{equation}
  Now the action of $\mcL$ on $W$ provides a map in $\mathrm{Hom}_\mcL(\mcL\otimes W,W)$, which in turn gives a map 
  in $\mathrm{Hom}_\mcL(W\otimes W^*,\mcL^*)$ by Eq.~\eqref{eq_dua} and the above considerations.
 The $\mcL$-invariant map  $W\times W^*\to\mcL^*$ obtained in this way is usually referred to as the  map \emph{dualizing the action} of $\mcL$ on $W$. If besides $\mcL$ is a simple Lie algebra, its Killing form provides a convenient isomorphism between the adjoint module $\mcL$ and its dual, and then we get an $\mcL$-invariant map $W\times W^*\to\mcL$. In the case $\mcL= \slf(V)$, we can replace the Killing form with the bilinear form on $\slf(V)$ given by the trace, which also provides a way of identifying $\mcL$ with $\mcL^*$ by means of $f\mapsto \tr(f\circ-)$.
\end{remark}

Particularizing  Remark~\ref{eq_dua} for $\mcL= \slf(V)$ and $W= \bigwedge^i V$, we get  the mentioned $\mcL$-invariant map   
$$
\textstyle
[\ ,\ ]\colon\bigwedge^i V\times\bigwedge^{n-i} V\to
\slf(V)
$$
 determined by  the condition 
\begin{equation}\label{eq_dualizando}
  \tr(f\circ[x,y])=(f\cdot x,y)
\end{equation}
for all $ f\in\slf(V)$, $x\in\bigwedge^{i} V$, $y\in\bigwedge^{n-i} V$.  
Again note that $[x,y]=[y,x]$ if both $i$ and $n-i$ are odd, while $[x,y]=-[y,x]$ otherwise.
(With a slight abuse of notation, $[\ ,\ ]$  is \emph{skew-symmetric} if and only if $(\ ,\ )$ is \emph{symmetric}.)
For later use, observe that if $i=0$, then $[x,y]=0$ for any $x\in \bigwedge^0 V=\FF$ and for any $y\in  \bigwedge^n V\cong \FF$, since $f\cdot x=0$ for all $ f\in\slf(V)$.

Recall from representation theory (see, for instance, \cite{Humphreysalg})
that   $\dim\mathrm{Hom}_{\slf(V)}( \bigwedge^{r}V\otimes\bigwedge^{n-r}V,\slf(V)) =1$ if $0<r<n$,
to conclude that any other 
$\mcL$-invariant map
$\bigwedge^r V\times\bigwedge^{n-r} V\to
\slf(V)
$ is necessarily a scalar multiple of  $[\ ,\ ]$.
Similarly, $\dim\mathrm{Hom}_{\slf(V)}( \bigwedge^{r}V\otimes\bigwedge^{s}V,\bigwedge^{r+s}V) =1$ if $0\le r+s\le n$, so that any   
$\mcL$-invariant map $ \bigwedge^{r}V\otimes\bigwedge^{s}V\to\bigwedge^{r+s}V$ is necessarily multiple of that one in Eq.~\eqref{eq_wed}.
Both facts  will be crucial for constructing $\e_8$ starting from its gradings, following the lines in Theorem~\ref{teo_main}. 
\medskip

Relative to this kind of constructions, an observation is in order.
    \begin{remark}\label{remark:slv} 
    In order to construct a Lie algebra  from a  \lq\lq simpler\rq\rq Lie algebra $\mathfrak{h}$  and an $\mathfrak{h}$-module $ \mathfrak{m}$,  define in $ \mathfrak{g}= \mathfrak{h}\oplus  \mathfrak{m}$ the bracket $[h+x,h'+x']:=[h,h']+h\cdot x'-h'\cdot x+\mu(x,x')$ for some fixed $\mathfrak{h}$-invariant skew-symmetric map $\mu\colon  \mathfrak{m}\times  \mathfrak{m}\to  \mathfrak{h}$, and any elements $h,h'\in\mathfrak h$, $x,x'\in\mathfrak m$, where the action of $\mathfrak{h}$ on $\mathfrak{m}$ is denoted with $\cdot$. We wonder when $(\mathfrak{g},[\ ,\ ])$ is a Lie algebra, that is, when $J (\g,\g,\g)=0$, for $J$ the Jacobian operator defined by  $J(x,y,z) := [[x,y],z]+[[y,z],x]+[[z,x],y]$. 
    Recall that $J(\hh,\hh,\hh)=0$ is a consequence of the fact that $\mathfrak{h}$ is a Lie subalgebra; and $J(\hh,\hh,\mm)=0$  is a consequence of being 
    $\mathfrak{m}$    an $\mathfrak{h}$-module. The condition $J (\hh,\mm,\mm)=0$ is equivalent to the $\mathfrak{h}$-invariance of $\mu$. All this means that $\mathfrak{g}$ is a Lie algebra if and only if $J (\mm,\mm,\mm)=0$. 
  \end{remark}

At some point, it will be necessary to make concrete computations to check when $J (\mm,\mm,\mm)=0$ to apply the above remark. To that aim, it is convenient to handle $*$, $\,\llcorner\,$ and $[\ ,\ ]$ for elements in a basis.

 Take $(e_1,\dots,e_n)$ a basis of $V$ and $(e^1,\dots,e^n)$ its dual basis of $V^*$. 
 That is, $e^i(e_j)=\delta_{ij}$ for $\delta_{ij}$ the Kronecker delta ($1$ if the variables are equal, and $0$ otherwise).
   Denote by 
   $$
   \textstyle
   e_{i_1\dots i_k}:=e_{i_1}\wedge\dots\wedge e_{i_k}\in\bigwedge^kV,\qquad e^{i_1\dots i_k}:=e^{i_1}\wedge\dots\wedge e^{i_k}\in\bigwedge^kV^*,
   $$ 
   for any $ i_1,\dots,i_k\in \{1,\dots,n\}$.  
 Thus   $\mathcal B_{k}:=\{e_{i_1\dots i_k}:1\le i_1<\dots<i_k\le n\}$ is  a basis of $ \bigwedge^kV$.
 Denote by $e_i^j\colon V\to V$ the linear map such that  $e_i^j(e_k):=\delta_{jk}e_i$.
 (With the notation at the beginning of the section,
 $e_i^j=\Psi(e^j\otimes e_i)$.) The set $\{e_i^j:1\le i,j\le n\}$ provides a basis of $\mathfrak{gl}(V)$ such that $\tr(e_i^j)=\delta_{ij}$.
 
 Let us fix $\phi(e_{1\dots n})=1$ the map in \eqref{eq_fi}.
  If  $\sigma=(\sigma(1),\dots,\sigma(n))$ denotes a permutation of $\{1,\dots,n\}$, we have $e_{\sigma(1)\dots\sigma(n)}=\sg(\sigma)\,e_{1\dots n}$, so that
  $\phi(e_{\sigma(1)\dots\sigma(n)})=\sg(\sigma)$.

  \begin{lemma}\label{le_notacionesyproductos}
  Fix $(e_1,\dots,e_n)$ a basis of $V$ with $\phi(e_{1\dots n})=1$, and take $\sigma=(i_1,\dots,i_n)$  and $(j_1,\dots,j_n)$   permutations of $\{1,\dots,n\}$. Then we have 
 \begin{itemize}
 \item[\rm (a)] $
 \widetilde{e_{i_1\dots i_k}}= \sg(\sigma) \,e^{i_{k+1}\dots i_n}\,;\vspace{3pt}    
 $
 \item[\rm (b)] 
 The map $[\ ,\ ]$ defined in Eq.~\eqref{eq_dualizando} is given, for elements in $\mathcal B_{k}$ and $\mathcal B_{n-k}$, by:\vspace{3pt}
 \begin{itemize}
 \item[$(1)$] 
 $[e_{i_1\dots i_k},e_{i_{k+1}\dots i_n}]=(-1)^{k(n-k)}\,\sg(\sigma)\,
 \big(\sum_{j=1}^ke_{i_j}^{i_j}-\frac kn\id_V\big )$;\vspace{3pt}
 \item [$(2)$]
 $[e_{i_1\dots i_k},e_{i_1i_{k+1}\dots i_{n-1}}]=(-1)^{k(n-k)}(-1)^{n+k}\,\sg(\sigma)\,
 e_{i_1}^{i_n}$;\vspace{3pt}
  \item [$(3)$]
 $[e_{i_1\dots i_k},e_{j_1\dots j_{n-k}}]=
0$ if $\{i_1,\dots, i_k\}\cap\{j_1,\dots,j_{n-k}\}$ has at least cardinal 2.\vspace{2pt}
  \end{itemize}\vspace{3pt}

    \item[\rm (c)]  
 $e_{i_1 \ldots i_{p+q}}\, \llcorner\, e^{j_1 \ldots j_{p}} =0$ if $J=\{j_1, \ldots, j_{p}\}\not\subset I= \{i_1,\ldots, i_{p+q} \}$. Otherwise, if 
 $k_1<\dots<k_p\in \{1,\dots,p+q\}$, then,
 $$
 e_{i_1 \ldots i_{p+q}}\, \llcorner\, e^{i_{k_1} \ldots i_{k_p}}= \sg\tiny
 \begin{pmatrix} i_1&\dots&i_p&i_{p+1}&& &\dots && &i_{p+q}\\i_{k_1}&\dots&i_{k_p}&i_1&\dots&\widehat{i_{k_1}}& \dots&\widehat{i_{k_p}}& \dots& {i_{p+q}}  \end{pmatrix} \,e_{i_1 \ldots \widehat{i_{k_1}} \ldots \widehat{i_{k_p}} \ldots i_{p+q}}.
 $$
 
\item[\rm (d)]  If $x\in\bigwedge^iV$, $y\in\bigwedge^jV$, then
$x*y=\begin{cases} (-1)^{ij}y*x \quad \textrm{ if $i+j\le n$,}
\\
 (-1)^{ij+(i+j-n)}y*x \quad \textrm{  if $i+j>n$.}\end{cases}$
  \end{itemize}
\end{lemma}

As usual, the hat on an index in item (c) indicates that the index is omitted.
 
 \begin{proof}
 For (a) we have to check that $\langle y,\sg(\sigma)\,e^{i_{k+1}\dots i_n}\rangle=\phi(e_{i_1\dots i_k}\wedge y) $ for any $y\in\bigwedge^{n-k}V$. If
 $y=e_{i_{k+1}\dots i_n}$, clearly $\langle y,e^{i_{k+1}\dots i_n} \rangle=\det I_{n-k}=1$, so that
 both sides of the equation equal $\sg(\sigma)$. If $y$ is any other basic  vector in $\mathcal B_k$, both sides of the equation annihilate.
 
   (b1)   Note that, for any $f\in\slf(V)$,
 \begin{equation}\label{eqtrazas}
 e_{i_1}\wedge \dots \wedge  f(e_{i_k}) \wedge  \dots  \wedge   e_{i_n}=\tr(f\circ e_{i_k}^{i_k})e_{i_1\dots i_n},
 \end{equation}
 since, if we write $ f(e_{j}) =\sum_{i=1}^n a_{ij}e_i$, then $\tr(f\circ e_{i}^{i})=a_{ii}$ and 
 $e_{i_1}\wedge \dots \wedge  f(e_{i_k}) \wedge  \dots  \wedge   e_{i_n}=\sum_{i=1}^n a_{ii_k}e_{i_1\dots i_{k-1}ii_{k+1}\dots i_n}
 =a_{{i_k}{i_k}}e_{i_1\dots i_n}$. Now, 
 $$
 \begin{array}{rl}
(-1)^{k(n-k)} \tr(f\circ [e_{i_1\dots i_k},e_{i_{k+1}\dots i_n}])&=\phi(f\cdot e_{i_1\dots i_k}\wedge e_{i_{k+1}\dots i_n} ) 
   \vspace{2pt}\\
      &
 =\phi(f(e_{i_1})\wedge e_{i_2\dots i_n}  +\dots +e_{i_1\dots  i_{k-1}}\wedge f(e_{i_k})\wedge e_{i_{k+1}\dots i_n} )
  \end{array}
 $$
 coincides, by 
 \eqref{eqtrazas},
 with
 $$
  \phi\Big( \sum_{j=1}^k\tr(f\circ e_{i_j}^{i_j})  e_{i_1\dots i_n}\Big)  
 = \sg(\sigma)\tr\big(f\circ \sum_{j=1}^k e_{i_j}^{i_j}\big)   
 =\sg(\sigma)\tr\big(f\circ (\sum_{j=1}^k e_{i_j}^{i_j}-\frac kn\id_V)\big), 
 $$
 
\noindent
 since $f$ has zero trace. As
  $ \left( \sum_{j=1}^k e_{i_j}^{i_j}-\frac kn\id_V\right) \in\slf(V)$
  and Eq.~\eqref{eq_dualizando} determines a unique element in $\slf(V)$,
   so  this traceless map coincides with $(-1)^{k(n-k)}\sg(\sigma)\,[e_{i_1\dots i_k},e_{i_{k+1}\dots i_n}]$.

  (b2) Observe first  that
\begin{equation}\label{hola}
  \begin{array}{rl}
   (-1)^{k(n-k)}\tr(f\circ
[e_{i_1\dots i_k},e_{i_1i_{k+1}\dots i_{n-1}}])&= \phi(f\cdot e_{i_1\dots i_k}\wedge e_{i_1i_{k+1}\dots i_{n-1}})\vspace{2pt}\\ 
 &= \phi(f(e_{i_1})\wedge e_{i_2\dots i_{k}i_1i_{k+1}\dots i_{n-1}}  ),
 \end{array}
\end{equation}
 since the index $i_1$ appears twice in each term $e_{i_1\dots i_{j-1}}\wedge f(e_{i_{j}})\wedge e_{i_{j+1}\dots i_{k}i_1i_{k+1}\dots i_{n-1}} $, $j\le k$. 
 Since $\tr(f\circ e_i^j)=a_{ji}$,  the expression in \eqref{hola} coincides with 
 $$
 \phi\big(\sum_{i=1}^n a_{ii_1}  e_{ii_2\dots i_{k}i_1i_{k+1}\dots i_{n-1}} \big ) 
  = \phi(e_{i_ni_2\dots i_{k}i_1i_{k+1}\dots i_{n-1}}  ) a_{i_ni_1}  
 =  (-1)^{n-1+k-1} \sg(\sigma)\tr(f\circ e_{i_1}^{i_n}),
 $$
 and hence Eq.~\eqref{eq_dualizando}  gives the required expression for $[e_{i_1\dots i_k},e_{i_1i_{k+1}\dots i_{n-1}}]$.
 
  (b3) For  any $f\in\slf(V)$, we have  \vspace{-1pt}
  $$
  (f\cdot e_{i_1\dots i_k})\wedge e_{j_1\dots j_{n-k}}=\sum_{j=1}^k e_{i_1\dots i_{j-1}}\wedge f(e_{i_{j}})\wedge e_{i_{j+1}\dots i_kj_1\dots j_{n-k}}=0.
  $$
  In fact, each   term in the 
 sum  annihilates since being    the wedge product of basic vectors with at least two of them repeated.
  
 Item (c) is clear. Moreover, $e_I\,\llcorner\, e^J=\sg\tiny\begin{pmatrix}I\\J\ I\setminus J\end{pmatrix}e_{I\setminus J}$ if $J\subset I$, even though $J$ does not appear preserving the order in $I$. Here $I$ and $J$ are ordered subsets of different indices and we have used the notation $e_I$ for $e_{i_1\dots i_k}$ if $I=(i_1,\dots,i_{k})$.\vspace{2pt}

   (d) If $i+j < n$, we have already mentioned that  $x \wedge y = (-1)^{ij} y \wedge x$. Also, if $i+j = n$, then $x * y = \phi (y \wedge x)= (-1)^{ij} \phi (x \wedge y)=(-1)^{ij} y*x.$ Finally consider the case $i+j>n$. There is no loss of generality in assuming that both $x$ and $y$ are wedge products of basic vectors in $\{e_1,\dots,e_n\}$. If the union of the indices involved in $x$ and $y$ is not the whole set $\{1,\dots,n\}$, then $x*y=x\,\llcorner\,\tilde y =0$ by (a) and (c), so there is nothing to prove, since similarly $y*x=0$. Otherwise, we can write $x=e_{i_1\dots i_sj_1\dots j_p}$ and $y=e_{i_1\dots i_sk_1\dots k_q}$ with $\{1,\dots,n\}=\{i_1,\dots ,i_s,j_1,\dots, j_p,k_1,\dots ,k_q\}$, $n=s+p+q$, with the indices possibly not ordered. We compute
  $$
  \begin{array}{l}
  x*y=x\,\llcorner\,\tilde y \stackrel{(a)}=x\,\llcorner\, \sg( \sigma_1)e^{j_1\dots j_p}\stackrel{(c)}=\sg( \sigma_1)\sg( \sigma_2)e_{i_1\dots i_s},\\
  y*x=y\,\llcorner\,\tilde x =y\,\llcorner\, \sg( \sigma_3)e^{k_1\dots k_q}=\sg( \sigma_3)\sg( \sigma_4)e_{i_1\dots i_s},
  \end{array}
  $$
 for
 $$ \begin{array}{l}
 \sigma_1=(i_1,\dots ,i_s,k_1,\dots ,k_q,j_1,\dots, j_p),\quad \sigma_2=(j_1,\dots,j_p,i_1,\dots,i_s),\\
 \sigma_3=(i_1,\dots ,i_s,j_1,\dots, j_p,k_1,\dots ,k_q),\quad \sigma_4=(k_1,\dots,k_q,i_1,\dots,i_s).
 \end{array}
$$
Now $\sg( \sigma_1)\sg( \sigma_3)=(-1)^{pq}$ and $\sg( \sigma_2)\sg( \sigma_4)=(-1)^{ps+qs}$. Taking in mind that $i=p+s$, $j=q+s$ and $i+j=n+s$, then
$pq+(p+q)s=ij-(i+j-n)^2\equiv _2 ij+i+j-n$.
  \end{proof}

 \begin{remark}\label{re_Z3}
 With the  above introduced notations, there is a concise but complete description of the model of $\e_8$ based on the $\ZZ_3$-grading with fixed part $\slf(9)$.
 Take $V$ a vector space over  $\FF$ of dimension 9, and  take $\mcL=\mcL_{\bar 0}\oplus\mcL_{\bar 1}\oplus\mcL_{\bar 2}$ as 
$$
\begin{array}{l}
\mcL_{\bar 0}=\slf(V), \\
\mcL_{\bar i}=\bigwedge^{[3i]_9}V ,
\end{array}
$$
for $i=1,2 $. Fix $a_{11},a_{22},b_1\in\FF $ and define on $\mcL$ the product such that $\mcL_{\bar 0}$ is subalgebra, the action of $\mcL_{\bar 0}$ on each $\mcL_{\bar i}$ is the usual one, and, if $i,j\in\{1,2\}$, $x\in \mcL_{\bar i}$, $y\in \mcL_{\bar j}$,
\begin{equation}\label{model_Z3}
[x,y]_{\mcL}=\left\{
\begin{array}{ll}
a_{ij}\,x*y\quad\  &i+j\ne3,\\
b_i[x,y]\ &i+j=3.\end{array}\right.
\end{equation}
Then $\mcL$ is a Lie algebra if and only if $a_{11}a_{22}+b_1=0$,  following Remark~\ref{remark:slv} 
for $\mathfrak{h}=\mcL_{\bar 0} $ and $\mathfrak{m}=\mcL_{\bar 1}\oplus\mcL_{\bar 2}$.    If besides the three scalars are nonzero, then the obtained $\ZZ_3$-graded Lie algebra  is isomorphic to $\e_8$, because there is  only one 
 simple Lie algebra of dimension $248$ over $\FF$, up to isomorphism. For instance, this is the case for the choice of scalars
$a_{11}=a_{22}=-b_1=1$.

This description of $\e_8$ for $(G,\mathfrak h)=(\mathbb Z_3,A_8)$ is essentially the same as   the ones mentioned in Introduction based on $A_8$, but the comparison is worthwhile. First, the model in \cite{EMSIII} and \cite{trivectoresZ3} gives, for $S, S_1,S_2\in\bigwedge^3V$ and $T,T^1,T^2\in\bigwedge^3V^*$,
the products
$$\begin{array}{c}
[S,T]_j^i=\frac12S^{ijk}T_{jkl}-\frac1{18}S^{klm}T_{klm}\delta_j^i,\quad
[S_1,S_2]_{ijk}=\frac1{36}\varepsilon_{i_1j_1k_1i_2j_2k_2ijk}S_1^{i_1j_1k_1}S_2^{i_2j_2k_2},\vspace{5pt}
\\
{[}T^1,T^2]^{ijk}=-\frac1{36}\varepsilon^{i_1j_1k_1i_2j_2k_2ijk}T^1_{i_1j_1k_1}T^2_{i_2j_2k_2}.
\end{array}
$$
Here the products are expressed in coordinates, which would be more difficult to write when increasing the number of pieces of the considered grading. 
Second, the description in \cite[\S22.4]{FultonHarris} uses two fixed trilinear maps
$T\colon \bigwedge^3(\bigwedge^3V)\to \mathbb C$ and $T\colon \bigwedge^3(\bigwedge^3V)^*\to \mathbb C$ to give a pair of identifications $\wedge\colon \bigwedge^2(\bigwedge^3V)\to(\bigwedge^3V)^*$ and $\wedge\colon \bigwedge^2(\bigwedge^3V)^*\to\bigwedge^3V$ to define thereafter (for $a$, $b$ and $c$ fixed  scalars)
$$
[v,w]=a\,v\wedge w,\quad [\varphi,\psi]=b\,\varphi\wedge \psi,\quad [v,\varphi]=c\,v\star\varphi,
$$
where $v\star\varphi$ is determined by $B(v\star\varphi,Z)=\varphi(Z\cdot v)$, for $B$ the Killing form, $v\in \bigwedge^3V$, $\varphi\in \bigwedge^3V^*$ and $Z\in\slf(V)$. With this notation, the Jacobi identity is equivalent to $18ab+c=0$ (\cite[Exercise~22.21]{FultonHarris}). Our version is quite inspired in this last one, but we have needed to establish a notation that would allow the procedure to be easily generalised  (observe the similarity between Eq.~\eqref{model_Z3} and   Proposition~\ref{teo_Z5model}).
 \end{remark}
 
 \begin{remark}\label{re_dim2}
 In the specific case of $U$ a vector space of dimension 2, we   can give more  concrete descriptions of the invariant   maps $(\ ,\ )\colon U\times U\to\FF$  and
 $[\ ,\ ]\colon U\times U\to\slf(U)$ considered in Eqs.~\eqref{eq_laforma} and \eqref{eq_dualizando}. If $\varphi\colon U\times U\to \FF$ is a nonzero alternating map, and $ \spf(U,\varphi)=\{f\in\mathfrak{gl}(U):\varphi(f(u),v)+\varphi(u,f(v))=0\}$ is the corresponding symplectic Lie algebra, then   the map $\varphi_{u,v}:=\varphi(u,-)v+\varphi(v,-)u$ belongs to $\mathfrak{sp}(U,\varphi)$. Moreover 
 $\mathfrak{sp}(U,\varphi)=\{\sum_i\varphi_{u_i,v_i}:u_i,v_i\in U\}$ coincides with the algebra $\slf(U)$. 
 If we choose $(e_1,e_2)$ a basis of $U$ with $\varphi(e_1,e_2)=1$, and the   isomorphism $\phi$  in \eqref{eq_fi}    is taken as $\phi(e_1\wedge e_2)=1  $, then     straightforward computations give:
   $$
(u,v)= u*v=u\,\llcorner\,\tilde v=-\varphi(u,v), \qquad [u,v]=\frac12\varphi_{u,v}.     
 $$ 
  \end{remark}

  
 \subsection{ $\ZZ_m$-gradings on simple Lie algebras}\label{se_gradings}
 
 Any $\ZZ_m$-grading  on a simple Lie algebra $\mcL=\oplus_{\bar j\in\ZZ_m}\mcL_{\bar j} $ is the eigenspace decomposition of an order $m$ automorphism $\theta$ of $\mcL$, namely, the automorphism defined as $\theta\vert_{\mcL_{\bar j} }=\xi^j\id$, for $\xi\in\FF$  a fixed primitive $m$th root of the unit. The converse is also true and any order $m$ automorphism of $\mcL$ produces  a $\ZZ_m$-grading on $\mcL$. In particular, the neutral component $\mcL_{\bar 0}$ is the subalgebra fixed by $\theta$. Moreover, each $\mcL_{\bar j}$ is an irreducible $\mcL_{\bar 0}$-module.
  
 The finite order automorphisms of $\e_8$ are classified, up to conjugation, in \cite[\S 8.6]{Kac} in terms of affine Dynkin diagrams and sequences of relatively prime nonnegative integers $(s_0,\ldots,s_8)$. If the fixed subalgebra of certain   nontrivial finite order automorphism $\theta$  is semisimple, then \cite[Proposition 8.6]{Kac} shows that the related sequence is of the form $(0,\ldots,1,\ldots,0)$ with only one $s_i=1$ and $i>0$. In such case,  the order of $\theta$  is just $m_i$, if 
 $\tilde \alpha=2 \alpha_1+3\alpha_2+4\alpha_3+6\alpha_4+5\alpha_5+4\alpha_6+3\alpha_7+2\alpha_8 =\sum_{j=1}^8 m_j\alpha_j $ denotes the maximal root of  a root system $\Phi$ of $\e_8$ relative to some Cartan subalgebra $H$
    and to some set of simple roots $\{\alpha_1,\dots,\alpha_8\}$   of $\Phi$.  Moreover, the   subalgebra fixed by $\theta$ has as Dynkin diagram the obtained one when removing the $i$th node of the 
extended affine Dynkin diagram $E_8^{(1)}$,

 
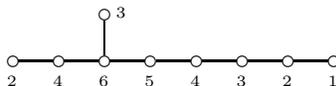
\vspace{-3pt}\begin{center}{
\begin{picture}(30,5)(3,0.5)  
\put(5,0){\circle{1}} \put(9,0){\circle{1}} \put(13,0){\circle{1}}
\put(17,0){\circle{1}} \put(21,0){\circle{1}}\put(25,0){\circle{1}}\put(29,0){\circle{1}}\put(33,0){\circle{1}}
\put(20.5,-2){$\scriptstyle _{4}$} \put(24.5,-2){$\scriptstyle _{3}$} \put(28.5,-2){$\scriptstyle _{2}$} \put(32.5,-2){$\scriptstyle _{1}$}  
\put(16.5,-2){$\scriptstyle _{5}$} \put(12.5,-2){$\scriptstyle _{6}$} \put(8.5,-2){$\scriptstyle _{4}$} \put(4.5,-2){$\scriptstyle _{2}$} \put(14,4){$\scriptstyle _{3}$} 
\put(13,4){\circle{1}}
\put(5.5,0){\line(1,0){3}}
\put(9.5,0){\line(1,0){3}} \put(13.5,0){\line(1,0){3}}  
\put(17.5,0){\line(1,0){3}} 
\put(21.5,0){\line(1,0){3}}
\put(25.5,0){\line(1,0){3}}
\put(29.5,0){\line(1,0){3}}   
\put(13,0.5){\line(0,1){3}}
\end{picture}
}\end{center}\vskip0.4cm

\noindent since $\{-\tilde\alpha,\alpha_j:j\ne i\}$ is a set of simple roots of $\mcL_{\bar0}$.
 This entails that, up to conjugation, there are 8  finite order automorphisms of $\e_8$  whose fixed subalgebra is semisimple: two of order 2, two of order 3, two of order 4,   one of order 5 and another one of order 6.  The choice $i=2$ provides the order 3 automorphism of $\mcL=\e_8$ with fixed part $  \slf(9) $, 
corresponding to the model detailed in Remark~\ref{re_Z3}.
The choices
$i=1,6,7,8$, provide gradings with neutral components with summands of type $D$ and $E$. There are only 3 cases left, producing gradings on $\e_8$ over cyclic groups:
\begin{itemize}
\item If $i=5$, we have a $\ZZ_5$-grading on $ \e_8$ with $\mcL_{\bar0}\cong \slf(5)\oplus\slf(5)$, since the removed diagram is

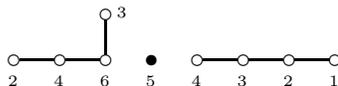
\vspace{-10pt} \begin{center}{
\begin{picture}(30,5)(3,0.5)  
\put(5,0){\circle{1}} \put(9,0){\circle{1}} \put(13,0){\circle{1}}
\put(17,0){\circle*{1}} \put(21,0){\circle{1}}\put(25,0){\circle{1}}\put(29,0){\circle{1}}\put(33,0){\circle{1}}
\put(20.5,-2){$\scriptstyle _{4}$} \put(24.5,-2){$\scriptstyle _{3}$} \put(28.5,-2){$\scriptstyle _{2}$} \put(32.5,-2){$\scriptstyle _{1}$}  
\put(16.5,-2){$\scriptstyle _{5}$} \put(12.5,-2){$\scriptstyle _{6}$} \put(8.5,-2){$\scriptstyle _{4}$} \put(4.5,-2){$\scriptstyle _{2}$} \put(14,4){$\scriptstyle _{3}$} 
\put(13,4){\circle{1}}
\put(5.5,0){\line(1,0){3}}
\put(9.5,0){\line(1,0){3}} 
\put(21.5,0){\line(1,0){3}}
\put(25.5,0){\line(1,0){3}}
\put(29.5,0){\line(1,0){3}}   
\put(13,0.5){\line(0,1){3}}
\end{picture}
}  \end{center}\vspace{8pt}

\item If $i=3$, we have a $\ZZ_4$-grading on $\mcL=\e_8$ with $\mcL_{\bar0}\cong \slf(8)\oplus\slf(2)$, since the removed diagram is

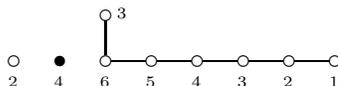
\vspace{-10pt}\begin{center}{
\begin{picture}(30,5)(3,0.5)  
\put(5,0){\circle{1}} \put(9,0){\circle*{1}} \put(13,0){\circle{1}}
\put(17,0){\circle{1}} \put(21,0){\circle{1}}\put(25,0){\circle{1}}\put(29,0){\circle{1}}\put(33,0){\circle{1}}
\put(20.5,-2){$\scriptstyle _{4}$} \put(24.5,-2){$\scriptstyle _{3}$} \put(28.5,-2){$\scriptstyle _{2}$} \put(32.5,-2){$\scriptstyle _{1}$}  
\put(16.5,-2){$\scriptstyle _{5}$} \put(12.5,-2){$\scriptstyle _{6}$} \put(8.5,-2){$\scriptstyle _{4}$} \put(4.5,-2){$\scriptstyle _{2}$} \put(14,4){$\scriptstyle _{3}$} 
\put(13,4){\circle{1}}
\put(13.5,0){\line(1,0){3}}  
\put(17.5,0){\line(1,0){3}} 
\put(21.5,0){\line(1,0){3}}
\put(25.5,0){\line(1,0){3}}
\put(29.5,0){\line(1,0){3}}   
\put(13,0.5){\line(0,1){3}}
\end{picture}
}    \end{center} \vspace{8pt}

\item If $i=4$, we have a $\ZZ_6$-grading on $ \e_8$ with $\mcL_{\bar0}\cong \slf(6)\oplus\slf(2)\oplus\slf(3)$, since the removed diagram is

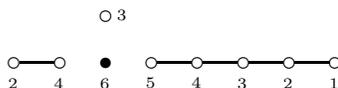
\vspace{-10pt}\begin{center}{
\begin{picture}(30,5)(3,0.5)  
\put(5,0){\circle{1}} \put(9,0){\circle{1}} \put(13,0){\circle*{1}}
\put(17,0){\circle{1}} \put(21,0){\circle{1}}\put(25,0){\circle{1}}\put(29,0){\circle{1}}\put(33,0){\circle{1}}
\put(20.5,-2){$\scriptstyle _{4}$} \put(24.5,-2){$\scriptstyle _{3}$} \put(28.5,-2){$\scriptstyle _{2}$} \put(32.5,-2){$\scriptstyle _{1}$}  
\put(16.5,-2){$\scriptstyle _{5}$} \put(12.5,-2){$\scriptstyle _{6}$} \put(8.5,-2){$\scriptstyle _{4}$} \put(4.5,-2){$\scriptstyle _{2}$} \put(14,4){$\scriptstyle _{3}$} 
\put(13,4){\circle{1}}
\put(5.5,0){\line(1,0){3}}
\put(17.5,0){\line(1,0){3}} 
\put(21.5,0){\line(1,0){3}}
\put(25.5,0){\line(1,0){3}}
\put(29.5,0){\line(1,0){3}}   
\end{picture}
}  \end{center}\vspace{5pt}

 \end{itemize}

In order to be even more precise, the  homogeneous components of the $\ZZ_{m_i}$-grading related to an index $i$ can be written as a sum of root spaces: 
$$
\begin{array}{l}
\mcL_{\bar 0}=H\oplus\big(\oplus\{\mcL_\alpha:\alpha=\sum_{k=1}^8 t_k\alpha_k\in\Phi,\, t_i\equiv_{m_i}0\}\big),\vspace{2pt}\\
\mcL_{\bar j}=\oplus\{\mcL_\alpha:\alpha=\sum_{k=1}^8 t_k\alpha_k\in\Phi,\, t_i\equiv_{m_i} j\}.
\end{array}
$$
Although this description of the grading is very concrete, it is not always easy to work with root spaces, because it is not easy to know  the specific scalars when multiplying two root vectors. 
 In Sections \ref{se_2A4}, \ref{se_A1A7} and \ref{seZ6} we will deal with the above $\ZZ_5$, $\ZZ_4$ and $\ZZ_6$-gradings, respectively, giving models which will not make use of the root spaces. 
 In Sections \ref{se_4A2} and \ref{se_2A12A3}, we will deal with gradings over noncyclic groups, namely, $\ZZ_3^2$ and $\ZZ_2\times \ZZ_4$. They have been  obtained by considering two  finite order commuting automorphisms (related to the nodes $7$th, $8$th and $6$th, respectively), whose homogeneous components can again be described in terms of root spaces.

 
 \subsection{Model based on the subalgebra of type $2A_4$}\label{se_2A4}
 
As mentioned, a $\ZZ_5$-grading on $\mcL=\e_8$ is obtained when removing the node of the extended Dynkin diagram marked with $5$, so that the neutral component $\mcL_{{\bar 0}}$ is of type $2A_4$ and the homogeneous components are $\mcL_{{\bar 0}}$-irreducible modules, each of them necessarily tensor product of two $\slf_5(\FF)$-irreducible modules. 
The precise description of   the involved irreducible modules can be concluded from their dimensions and from the dimensions of the tensor products among them. Alternatively, this   decomposition appears  in \cite[Table~5]{Minchenko}, which is a work on the classification of  semisimple subalgebras, as well as  in \cite{Alb_Jordan gradings}, where an interesting refinement of such $\ZZ_5$-grading is studied. 
According to these references, there is  a vector space $V$ of $\FF$-dimension 5, such that, up to isomorphism, the $\mcL_{\bar 0}$-modules $\mcL_{\bar i}$ can be identified to:
\begin{equation}\label{model_2A4}
\begin{array}{ll}
\mcL_{\bar 0}&=\slf(V)\oplus\slf(V)\\
\mcL_{\bar i}&=\bigwedge^{[i]_5}V\otimes\bigwedge^{[2i]_5}V\qquad
i=1,2,3,4.
\end{array}
\end{equation}
Thus $\dim\mcL_{\bar 0}=48$ and $\dim\mcL_{\bar i}=50$ for any $i=1,2,3,4$.
Keep in mind that the action of $\mcL_{\bar 0}$ on $\mcL_{\bar i}$ is given by $((f , g),   x\otimes y) \mapsto f\cdot x \otimes y + x \otimes g\cdot y$ for any $x\in \bigwedge^{i}V$ and $y\in \bigwedge^{[2i]_5}V$, that is, each simple component of $\mcL_{\bar 0}$ acts on the corresponding module.
We abuse of the notation  using $f+g$ for $(f,g)\in \mcL_{\bar 0}$ when there is no confusion with the copy of $\slf(V)$ to which $f$ and $g$ belong. 
When this is not clear, we will use $f_1+g_2$.
Next, let us  change the viewpoint and begin with a graded vector space $\mcL$ with homogeneous components as in Eq.~\eqref{model_2A4}.
We will  endow this graded vector space  $\mcL$ with a    Lie algebra structure, to recover our exceptional Lie algebra.

\begin{proposition} \label{teo_Z5model}

Let $\mcL$ be the $\ZZ_5$-graded vector space given in Eq.~\eqref{model_2A4} for $V$ a vector space of $\FF$-dimension 5. Consider the product of $\mcL_{\bar 0}$ with $\mcL_{\bar i}$ given by the action of the $j$th copy of $\slf(V)$ on the $j$th slot ($j=1,2$). 
Fix some nonzero scalars $a_{ij},b_k^{(1)},b_k^{(2)}\in\FF^\times$ for any  $1\le i\le j\le 4$, $i+j\ne5$, $k=1,2$. Define
 the   bracket of $\mcL_{\bar i}$ with $\mcL_{\bar j}$, for any  $1\le i\le j\le 4$, by   
\begin{equation}\label{eq_pr2A4}
[x_1\otimes x_2,y_1\otimes y_2]:=\left\{
\begin{array}{ll}
a_{ij}\ x_1*y_1\otimes x_2*y_2&i+j\ne5,\\
b_i^{(1)}\,(x_2,y_2)[x_1,y_1]_1\,+\, b_i^{(2)}\,(x_1,y_1)[x_2,y_2]_2\qquad&i+j=5,\end{array}\right.
\end{equation}
for    
  any  $x_1\in\bigwedge^i V$, $x_2\in\bigwedge^{[2i]_5}V$, $y_1\in\bigwedge^j V$ and $y_2\in\bigwedge^{[2j]_5}V$.
Extend the bracket to the whole $\mcL$ such that it is skew-symmetric and the restriction to $\mcL_{\bar 0}$ is the usual bracket in $2\slf(V)$.
Then, $\mcL$ endowed with the  product $[\ ,\ ]$   is a ($\ZZ_5$-graded)   Lie algebra if and only if 
\begin{equation}\label{eq_SOL}
\begin{array}{l}
b_1^{(1)}=b_1^{(2)}=-a_{11}a_{24}=-a_{12}a_{34}=-a_{13}a_{44}, \\
b_2^{(1)}=b_2^{(2)}=-a_{12}a_{33}=-a_{22}a_{34}=-a_{13}a_{24}.\end{array}
\end{equation}
Moreover, such Lie algebra $(\mcL, [\ ,\ ])$ is   simple and isomorphic to $\e_8$. 
A solution of this system of   equations is, for instance,   
\begin{equation}\label{eq_unasol}
\begin{array}{rl}
1&=a_{11}=a_{22}=a_{33}=a_{34}=a_{44}=a_{12}=a_{13}=a_{24},\\
-1&=b_{1}^{(1)}=b_{1}^{(2)}=b_{2}^{(1)}=b_{2}^{(2)}.\end{array}
\end{equation}
In particular, $(\mcL, [\ ,\ ]  )$ with the bracket extending the one in \eqref{eq_pr2A4}, 
for the scalars in \eqref{eq_unasol}, gives a model of the split Lie algebra of type $E_8$.  
\end{proposition}

\begin{remark}\label{re_haysol}
As in the paragraph above the theorem, the split Lie algebra of type $E_8$ has a decomposition like  \eqref{model_2A4} given by its $\ZZ_5$-grading. Now, taking into account 
item d) in  Theorem~\ref{teo_main}, the product in such Lie algebra has to be the one considered in   \eqref{eq_pr2A4} for some choice of scalars.
 This implies that there exist  scalars $a_{ij},b_i^{(1)},b_i^{(2)}\in\FF^\times$ such that the graded vector space $\mcL$, endowed with the bracket in \eqref{eq_pr2A4}, is a Lie algebra.
 So,
 the system of equations with variables $a_{ij},b_i^{(1)},b_i^{(2)}\in\FF^\times$ making  $\mcL$ a Lie algebra always possesses  a solution. 
 This a priori knowledge  that solutions exist will save many computations in the proof of Proposition~\ref{teo_Z5model} (and throughout the manuscript). 
\end{remark}

\begin{proof}
First, note that it makes sense 
to deal with the skew-symmetric extension of the bracket because, 
for $i=j$, it is true that $[x_1\otimes x_2,y_1\otimes y_2]=-[y_1\otimes y_2,x_1\otimes x_2]$. In fact, for $x,y\in\bigwedge^i V$, Lemma~\ref{le_notacionesyproductos} says that $x*y=(-1)^{i^2}y*x$ if $2i\le5$ and $x*y=(-1)^{i^2-5}y*x$ if $2i>5$. So $x*y=-y*x$ if $i=1,4$ and $x*y=y*x$ if $i=2,3$, but if $i\in\{1,4\}$ then $2i\in\{2,3\}$ and if $i\in\{2,3\}$ then $[2i]_5\in\{1,4\}$.

   The constructed skew-symmetric   algebra $(\mcL, [\ ,\ ]  )$ is a Lie algebra if and only if $J(\mcL_{\bar i},\mcL_{\bar j},\mcL_{\bar k})=0$ for $0\le i,j,k\le 4$, because the Jacobian operator is trilinear. 
   Taking into consideration Remark \ref{remark:slv}, that condition reduces to 
   $J(\mcL_{\bar i},\mcL_{\bar j},\mcL_{\bar k})=0$ for $1\le i\le j\le k\le 4$. 
    
 A relevant observation is that
  $J(\mcL_{\bar i},\mcL_{\bar i},\mcL_{\bar i})=0$ always holds for all $0 \leq i \leq 4$, independently of the choice of scalars. Although that fact could be checked directly for each $i$, an indirect argument immediately gives the result:  by Remark~\ref{re_haysol}, the system of equations with variables $a_{ij},b_i^{(1)},b_i^{(2)}\in\FF^\times$ always has at least one solution.  Now note that, for any $x,y,z\in\mcL_{\bar i}$,  $1 \le i \leq 4$, the expression $\frac{J(x,y,z)}{a_{ii}a_{[2i]_5,i}}$ does not involve any of such   variables, but  it is only written in terms of  $x$, $y$, $z$, $*$, $\sim$ and $\llcorner$.
  (Here we use the notation $a_{ij}=a_{ji}$ if $j>i$; for avoiding to distinguish the value of $i$.) This means that either it vanishes for any $x,y,z\in\mcL_{\bar i}$ or it does not vanish for some choice of homogeneous elements $x,y,z$. The mentioned existence of solutions of the system of equations implies that the situation can not be the second one.

Denote by $(J_{ijk})$ the identity $J(\mcL_{\bar i},\mcL_{\bar j},\mcL_{\bar k})=0$. We   proceed by checking  that the identity $(J_{ijk})$ holds case by case.  It is enough to check it by taking    elements in $\mcL_{\bar l}=\bigwedge^{l}V\otimes\bigwedge^{[2l]_5}V$ which are tensor product of elements in $\bigwedge^{l}V$ and $\bigwedge^{[2l]_5}V$, since the Jacobian operator $J$ is trilinear. The same multilinearity says that we can assume that these elements are wedge products of elements in $V$. So, take  $\{u_i,v_i:i\in\mathbb N\}$ arbitrary elements in $V$, and use  the notation $u_{i_1\dots i_k}$ for $u_{i_{1}}\wedge\dots\wedge u_{i_{k}}$ as in Lemma~\ref{le_notacionesyproductos} (and similarly for $v_i$'s). Sometimes we will restrict to basic elements of $\bigwedge^{l}V$, obtained from a fixed basis $\{e_i:i=1,\dots,5\}$ of $V$ such that $\phi(e_{12345})=1$.
\medskip

\noindent    $ \boxed{ (J_{1,1,2})}$
Take $x=u_1 \otimes v_{12} \in \mcL_{\bar 1}$, $y=u_2 \otimes v_{34} \in \mcL_{\bar 1}$, $z=u_{34} \otimes v_{5678} \in \mcL_{\bar 2}$. Then 
$$
\begin{array}{ll}
J(x,y,z) =& a_{11}a_{22}\, u_{1234} \otimes (v_{1234} \,\llcorner\, \alpha )\\ 
& -a_{12}a_{13}\, (u_{1234} \otimes (v_{12} \wedge (\alpha(v_3) v_4-\alpha(v_4) v_3 )) \\
& - a_{12} a_{13}\, (u_{1234} \otimes (v_{34} \wedge (\alpha(v_1) v_2 - \alpha(v_2) v_1 )),
\end{array}
$$
for $\alpha= \widetilde{v_{5678}}$. But $(v_{1234} ) \,\llcorner\, \alpha = \alpha (v_1)  v_2 \wedge v_3 \wedge v_4 - \alpha(v_2) v_1   \wedge v_3 \wedge v_4 + \alpha(v_3) v_1 \wedge v_2  \wedge v_4- \alpha(v_4) v_1 \wedge v_2 \wedge v_3  $, so that  
$$
J(x,y,z)  = (a_{11}a_{22}-a_{12} a_{13}) (u_{1234} \otimes (v_{1234} \,\llcorner\, \alpha)) .
$$ 
This vanishes (for any $x$, $y$ and $z$) if and only if  $a_{11}a_{22}=a_{12} a_{13}$.  

 \medskip
\noindent    $ \boxed{ (J_{1,1,3})}$
Take $x=u_1 \otimes v_{12} \in \mcL_{ \bar 1}$ , $y=u_2 \otimes v_{34} \in \mcL_{\bar 1}$, $z=u_{345} \otimes v_{5} \in \mcL_{ \bar 3}$. 
We compute 
$$ 
    \begin{array}{rl}
J(x,y,z) =&   \phi (v_{12345})\big( a_{11} b_2^{(1)}  [u_{12},u_{345}] -a_{13}  b_1^{(1)} ([u_1,u_{2345}]-[u_2,u_{1345}] )\big)_1\\
&+ \ \phi (u_{12345})\big( a_{11} b_2^{(2)}  [v_{1234},v_5]-a_{13}  b_1^{(2)} ( [v_{34},v_{125}]+[v_{12},v_{345}])\big)_2\in \mcL_{ \bar 0}.
\end{array}
$$
This expression is zero if and only its projections on each copy of $\slf(V)$ are zero. But, these projections are zero if and only if \begin{equation}\label{ecu:jacobiano1}
    a_{11}b_2^{(1)}\tr(f \circ[u_{12},u_{345}]) - a_{13}b_1^{(1)}\tr(f \circ ([u_1,u_{2345}]- [u_2,u_{1345}]))=0,
\end{equation} 
and
\begin{equation}\label{ecu:jacobiano2}
    a_{11}b_2^{(2)}\tr(f \circ[v_{1234},v_{5}]) - a_{13}b_1^{(2)}\tr(f \circ ([v_{12},v_{345}] + [v_{34},v_{125}]))=0,
\end{equation}  for every $f \in \slf (V)$, since it is  always possible to choose $u$'s  such that $\phi (u_{12345})\ne0$ (similarly for   $v$'s). 
According to the definition of ${[ \ , \ ]} $ in  Eq.~\eqref{eq_dualizando}, we can rewrite \eqref{ecu:jacobiano1} as
$$
a_{11}b_2^{(1)}( f \cdot u_{12}, u_{345})- a_{13}b_1^{(1)}\big((f \cdot u_1, u_{2345}) - (f \cdot u_2, u_{1345})\big)=0.
$$
Taking in mind that $f \cdot u_{12}\wedge u_{345}=(f \cdot u_1)\wedge u_{2345}-(f \cdot u_2)\wedge u_{1345}$, the above expression can be written as 
$$
(a_{11}b_2^{(1)}- a_{13}b_1^{(1)})( f \cdot u_{12}, u_{345})=0,
$$
which holds for any choice of $u$'s if and only if $a_{11}b_2^{(1)}=a_{13}b_1^{(1)}$.
Also $f\cdot v_{1234}\wedge v_5=(f\cdot v_{12})\wedge v_{345}+(f \cdot v_{34})\wedge v_{125}$, so that 
 \eqref{ecu:jacobiano2} is equivalent to
 $$
(a_{11}b_2^{(2)}-a_{13}b_1^{(2)})(f\cdot v_{1234}, v_5) =0, 
$$
which is always true when $a_{11}b_2^{(2)}=a_{13}b_1^{(2)}$.
  \medskip
 
\noindent    $ \boxed{ (J_{1,1,4})}$
If  $x=u_1\otimes v_{12} ,\ y=u_2 \otimes v_{34} \in \mcL_{\bar 1}$ and $z=u_{3456} \otimes v_{567} \in \mcL_{\bar 4}$, we have that
\begin{equation}\label{eq114}
\begin{array}{ll}
J(x,y,z)=&  a_{11}a_{24}(u_{12}\,\llcorner\, \widetilde{u_{3456}} \otimes v_{1234} \,\llcorner\, \widetilde{v_{567}}) \\
& + b_1^{(1)}\phi(v_{56734 }) [u_2,u_{3456}](u_1) \otimes v_{12} + b_1^{(2)}\phi(u_{34562}) u_1 \otimes [v_{34},v_{567}](v_{12})\\
& - b_1^{(1)}\phi(v_{56712})[u_1,u_{3456}](u_2) \otimes v_{34}- b_1^{(2)}\phi(u_{34561})u_{2} \otimes [v_{12},v_{567}](v_{34}).\\
\end{array}
\end{equation}
We want necessary and sufficient conditions on the scalars $a_{11}$, $a_{24}$, $b_1^{(1)}$ and $b_1^{(2)}$ which guarantee that 
$J(\mcL_{\bar1},\mcL_{\bar1},\mcL_{\bar4})=0$. Necessary conditions are easy to obtain by using concrete elements. In fact, 
$$
0=J(e_1 \otimes e_{12},e_2 \otimes e_{34} ,e_{1345}\otimes e_{125} )=\left(a_{11}a_{24} +\frac{1}{5}b_1^{(1)}+\frac{4}{5}b_1^{(2)}\right)e_1\otimes e_{12},
$$
and
$$
0=J(e_1 \otimes e_{12},e_2 \otimes e_{12},e_{1345}\otimes e_{345})=\frac{6}{5}\left(b_1^{(1)}- b_1^{(2)}\right)e_1\otimes e_{12},
$$
which imply that some necessary conditions are
\begin{equation}\label{equ_114}
b_1^{(1)}= b_1^{(2)}=-a_{11}a_{24}.
\end{equation}
Now let us check that these conditions are enough too. Assume we have taken the scalars satisfying Eq.~\eqref{equ_114}. Then Eq.~\eqref{eq114} becomes,
$$
\begin{array}{ll}
 J(x,y,z )=& {b_1^{(1)}} \big( -u_{12}\,\llcorner\, \widetilde{u_{3456}} \otimes v_{1234} \,\llcorner\, \widetilde{v_{567}}  \\
 &+  \phi(v_{56734 }) [u_2,u_{3456}](u_1) \otimes v_{12} +\phi(u_{34562}) u_1 \otimes [v_{34},v_{567}](v_{12})\\
& -  \phi(v_{56712})[u_1,u_{3456}](u_2) \otimes v_{34}-\phi(u_{34561})u_{2} \otimes [v_{12},v_{567}](v_{34})\big).\\
\end{array}
$$
Now we are sure that this expression is zero without the need to  compute the value of the right side in full  detail: simply note that the vanishing of the above expression does not depend on the choice of the nonzero scalar  $ {b_1^{(1)}} $, and recall that the system of equations with variables $a_{ij},b_i^{(1)},b_i^{(2)}\in\FF^\times$ always has at least one solution as in Remark~\ref{re_haysol}. So we can conclude that $J(\mcL_{\bar1},\mcL_{\bar1},\mcL_{\bar4})$ is identically zero.

 From now on, we can assume that the scalars satisfy 
\begin{equation}\label{recopil21}
 a_{11}a_{22}=a_{12} a_{13},\quad
 a_{11}b_2^{(1)}=a_{13}b_1^{(1)},\quad
a_{11}b_2^{(2)}=a_{13}b_1^{(2)},\quad
b_1^{(1)}= b_1^{(2)}=-a_{11}a_{24},
\end{equation}
 all of them necessary conditions for $\mcL$ to be a Lie algebra. 
 Since $a_{11}b_2^{(1)}=a_{13}b_1^{(1)}=a_{13}b_1^{(2)}=a_{11}b_2^{(2)}$ and $a_{11}$ is non zero, so  $ b_2^{(1)}=b_2^{(2)}$. 
 Then, to simplify notation, we denote by 
 $$
 b_1:=b_1^{(1)}=b_1^{(2)},\qquad  b_2:=b_2^{(1)}=b_2^{(2)}.
 $$ \vspace{1pt}
 
\noindent    $ \boxed{ (J_{1,2,2})}$
Now take $x=u_1\otimes v_{12} \in \mcL_{\bar 1}$ and $y=u_{23}\otimes v_{3456}, z=u_{45}\otimes v_{789\,10} \in \mcL_{\bar 2}$.  
Note that Eq.~\eqref{recopil21} tells $(a_{12}a_{13})(b_2a_{11})=(a_{11}a_{22})(b_1a_{13})$, so that  $$a_{12}b_2=a_{22}b_1.$$ This allows to write
$$
\begin{array}{ll}
J(x,y,z) =& a_{12}b_{2}\big(-[u_{45},u_{123}]_1(v_{789\,10},v_{12}\,\llcorner\, \widetilde{v_{3456}}) - (u_{45}, u_{123})[v_{789\,10},v_{12}\,\llcorner\, \widetilde{v_{3456}}]_2\\
& + [u_{23},u_{145}]_1(v_{3456},v_{12}\,\llcorner\, \widetilde{v_{789\,10}}) + (u_{23}, u_{145})[v_{3456},v_{12}\,\llcorner\, \widetilde{v_{789\,10}}]_2\\
& - [u_{1},u_{2345}]_1(v_{12},v_{3456}\,\llcorner\, \widetilde{v_{789\,10}}) - (u_{1},u_{2345})[v_{12},v_{3456}\,\llcorner\, \widetilde{v_{789\,10}}]_2\big).
\end{array} 
$$
We conclude that this expression is always zero due to the fact that being zero or not does not depend on the choice of scalars. That is, we use the same arguments as those ones when checking  the identity $ (J_{1,1,4})$.
\medskip

\noindent     $ \boxed{ (J_{1,3,3})}$
Let $x= u_1\otimes v_{12} \in \mcL_{\bar{1}}$ and  $y=u_{234}\otimes v_3  , \, z=u_{567}\otimes v_4 \in \mcL_{\bar{3}}$. We get that
$$
J(x,y,z) = 
 \Big(a_{13}a_{34}(u_{567}\,\llcorner\, \widetilde{u_{1234}}+u_{234} \,\llcorner\, \widetilde{u_{1567}}) 
 + a_{33}a_{11}(u_{234}\,\llcorner\, \widetilde{u_{567}})\wedge u_1
 \Big) \otimes v_{1234}.
$$
For instance, for $x=e_1 \otimes e_{12} \in \mcL_{\bar{1}}$, $y=e_{234}\otimes e_3$ and $z=e_{125}\otimes e_4 \in \mcL_{\bar{3}}$, it is easy to check that
$$
J(x,y,z) =( a_{13}a_{34} 
 - a_{33}a_{11} )e_{12} \otimes e_{1234},
$$
so the fact $J(x,y,z)=0$ gives 
\begin{equation}\label{recopil22}
a_{13} a_{34}=a_{33}a_{11}.
\end{equation}
This necessary condition on the scalars is of course sufficient for getting $(J_{1,3,3})$, since the above general expression becomes
$$
J(x,y,z) = 
a_{33}a_{11}  \big( u_{567}\,\llcorner\, \widetilde{u_{1234}}+u_{234} \,\llcorner\, \widetilde{u_{1567}}
 + (u_{234}\,\llcorner\, \widetilde{u_{567}})\wedge u_1
 \big) \otimes v_{1234}.
$$
As it has to vanish for some choice of scalars, and $a_{33}a_{11} \ne0$,  the only way is that it vanishes for any choice of the scalars.
\medskip

\noindent
 $ \boxed{ (J_{2,2,4})}$
 A first computation,
 $$
 J(e_{12}\otimes e_{1234},e_{34}\otimes e_{1235},e_{1235}\otimes e_{345} ) =(a_{22}a_{44}-a_{24}a_{12})e_{123}\otimes e_3,
 $$
 gives, as a necessary condition, 
\begin{equation}\label{recopil23}
a_{22}a_{44}=a_{24}a_{12}.
\end{equation}
With this requisite on the scalars, $(J_{2,2,4})$ holds, since for  $x= u_{12}\otimes v_{1234} ,y=u_{34}\otimes v_{5678} \in \mcL_{\bar{2}}$ and  $ z=u_{5678}\otimes v_{9\,10\,11} \in \mcL_{\bar{4}}$,
$$
\begin{array}{ll}
J(x,y,z)=&a_{22}a_{44}\Big(u_{1234} \,\llcorner\, \widetilde{u_{5678}}\otimes (v_{1234} \,\llcorner\, \widetilde{v_{5678}})\,\llcorner\,  \widetilde{v_{9\,10\,11}}\\
&+(u_{34} \,\llcorner\, \widetilde{u_{5678}})\wedge u_{12}\otimes (v_{5678} \,\llcorner\, \widetilde{v_{9\,10\,11}}) \,\llcorner\, \widetilde{v_{1234}} \\
&-(u_{12} \,\llcorner\, \widetilde{u_{5678}})\wedge u_{34}\otimes (v_{1234} \,\llcorner\, \widetilde{v_{9\,10\,11}}) \,\llcorner\, \widetilde{v_{5678}}\Big),  
\end{array}
$$
which is zero independently of the values of $a_{22}$ and $a_{44}$.
\medskip

\noindent
 $ \boxed{ (J_{2,2,3})}$
Let $x=e_{12}\otimes e_{1234} \in \mcL_{\bar{2}}$, $y=e_{34}\otimes e_{1235} \in \mcL_{\bar{2}}$ and $z=e_{125}\otimes e_{4} \in \mcL_{\bar{3}}$. From
$$ 
J(x,y,z) 
 = (a_{22}a_{34}+b_2)e_{12}\otimes e_{1234},
$$
we get 
\begin{equation}\label{recopil24}
a_{22}a_{34}=-b_2.
\end{equation} 
As in the above cases, this condition is sufficient to guarantee $(J_{2,2,3})$.\medskip

Compiling the information of \eqref{recopil21}, \eqref{recopil22}, \eqref{recopil23} and \eqref{recopil24},  we have that all the restrictions on the scalars in \eqref{eq_SOL} are necessary for $\mcL$ to be a Lie algebra, since
$$
\begin{array}{lll}
a_{13}a_{24}a_{12}=a_{24}a_{11}a_{22}=-b_1a_{22}=-a_{12}b_2&\Rightarrow &a_{13}a_{24} =-b_2,\\
-b_1a_{13}=a_{11}a_{24}a_{13}=-b_2a_{11}=a_{22}a_{34} a_{11}=a_{13}a_{34}a_{12}  &\Rightarrow &a_{34}a_{12} =-b_1,\\
-b_2a_{11}=a_{22}a_{34}a_{11}=a_{13}a_{34}a_{12} =a_{33}a_{11}a_{12}  &\Rightarrow & a_{33}a_{12}=-b_2, \\
    a_{22}a_{44}a_{13}= a_{24}a_{12}a_{13}=  a_{22}a_{11}a_{24}=-a_{22}b_1 &\Rightarrow &   a_{44}a_{13}=-b_1.
\end{array}
$$
From now on, we asume that our scalars satisfy all the conditions in \eqref{eq_SOL}. In particular they satisfy all the relations in  \eqref{recopil21}, \eqref{recopil22}, \eqref{recopil23} and \eqref{recopil24}. Then we only have to test that   the remaining identities $(J_{i,j,k})$ hold too.
\medskip

\noindent
 $ \boxed{ (J_{1,2,3})}$
For elements $x=u_1\otimes v_{12} \in \mcL_{\bar 1}$, $y=u_{23}\otimes v_{3456}\in \mcL_{\bar 2}$ and $ z=u_{456}\otimes v_{7}\in \mcL_{\bar 3} $, and taking into account that $a_{12}a_{33}=-b_2=a_{13}a_{24}$, we obtain that
$$
\begin{array}{ll}
J(x,y,z) =&  b_2\Big(
[u_{23},u_{456}] (u_1) \otimes (v_{3456},v_7) v_{12} + (u_{23},u_{456}) u_1 \otimes [v_{3456},v_7](v_{12}) \\
&\quad - u_{123}\,\llcorner\, \widetilde{u_{456}} \otimes (v_{12} \,\llcorner\, \widetilde{v_{3456}}) \wedge v_7 
-  u_{23 } \,\llcorner\, \widetilde{u_{1456}} \otimes v_{3456} \,\llcorner\, \widetilde{v_{127}} \Big);
\end{array}
$$
which is necessarily  zero as in the above cases.
\medskip
 
\noindent
 $ \boxed{ (J_{1,2,4})}$
Consider now $x=u_1\otimes v_{12} \in \mcL_{\bar 1}$, $y=u_{23}\otimes v_{3456}\in \mcL_{\bar 2}$ and $ z=u_{4567}\otimes v_{789}\in \mcL_{\bar 4} $ and take into account $ a_{12}a_{34}= a_{24}a_{11}=-b_1 $ to get
$$
\begin{array}{ll}
J(x,y,z) =
&  -b_1\Big(u_{123}\,\llcorner\, \widetilde{u_{4567}} \otimes (v_{12}\,\llcorner\, \widetilde{v_{3456}})\wedge v_{789}
 +(u_{23}\ll{u_{4567}})\wedge u_1 \otimes (v_{3456} \ll{v_{789}})\wedge v_{12} \\
&\qquad+ [u_1,u_{4567}](u_{23})\otimes  (v_{12}, v_{789})v_{3456 } + (u_1,u_{4567})u_{23}\otimes [v_{12}, v_{789}] (v_{3456})\Big).
\end{array}$$
 Arguments as in the previous cases permit us to conclude that it vanishes.

\medskip
\noindent
 $ \boxed{ (J_{1,3,4})}$ Recall that    $a_{13}a_{44}=a_{34}a_{12}=-b_1$, 
 so that, for
  $x=u_1\otimes v_{12} \in \mcL_{\bar{1}}$, $y=u_{234}\otimes v_3 \in \mcL_{\bar{3}}$ and $z=u_{5678}\otimes v_{456} \in \mcL_{\bar{4}}$, we have
$$
\begin{array}{ll}
J(x,y,z) =&b_1\Big( u_1 \wedge(u_{234}\,\llcorner\,\widetilde{u_{5678}}) \otimes v_{12}\,\llcorner\, \widetilde{v_{3456}}- u_{1234}\,\llcorner\, \widetilde{u_{5678}} \otimes v_{123} \,\llcorner\, \widetilde{v_{456}} \\
&\quad- [u_1,u_{5678}](u_{234})\otimes \phi(v_{12 456}) v_3 - \phi(u_{15678})u_{234}\otimes [v_{12},v_{456}](v_3)\Big),
\end{array}$$
which is identically zero.

\medskip

Clearly, we do not have  to compute all the expressions to be completely sure that the rest of the identities   hold too:
\begin{itemize}
\item $(J_{1,4,4})$ holds since,  for any $x\in\mcL_{\bar 1}$, $y,z\in\mcL_{\bar 4}$,
we can write $\frac{J(x,y,z)}{b_1}$ without using any scalar, taking into account that $b_1=-a_{44}a_{13}$.
\item Similarly $\frac{J(x,y,z)}{b_2}$ is independent of the choice of scalars for any $x\in\mcL_{\bar 2}$, $y,z\in\mcL_{\bar 3}$, since $b_2=-a_{33}a_{12}$.  That is, we get that $(J_{2,3,3})$ is true.
\item $(J_{2,3,4})$ holds, since $b_2=-a_{24}a_{13}=-a_{22}a_{34}$.
\item $   (J_{2,4,4})$ follows from  $a_{24}b_1=a_{44}b_2$.
 \item $   (J_{3,3,4})$ follows from  $a_{33}b_1=a_{34}b_2$.
\item The identity  $   (J_{3,4,4})$ is consequence from  $a_{34}a_{24}=a_{44}a_{33}$.
\end{itemize}
All this implies that $\mcL$ is a Lie algebra under our assumptions in \eqref{eq_SOL}.   
 In fact, it is very easy to find   concrete solutions of the system \eqref{eq_SOL}, for instance, the one provided in \eqref{eq_unasol}.

The following step is to check that   the obtained Lie algebras are mutually isomorphic. To do that, take $\mcL$ and $\mcL'$  the $\ZZ_5$-graded vector spaces as in \eqref{model_2A4} endowed with the Lie bracket as in \eqref{eq_pr2A4} with the coefficients given by  \eqref{eq_SOL} and 
 \eqref{eq_unasol} respectively. Let $ f\colon\mcL\to\mcL'$ be the map  given by $f(x)=\alpha_i x$, for any $x\in\mcL_{\bar i}$, for any choice of  scalars $\{\alpha_i:i=0,\dots,4\}\subset\FF^\times$ such that
\begin{equation}\label{eq_simplicidad}
\alpha_0=1,\quad
(\alpha_1)^5=  -b_1a_{11}a_{12}a_{13},\quad
\alpha_2=-\frac{b_1a_{12}a_{13} }{\alpha_1^3},\quad
\alpha_3=  \frac{-b_1a_{13}}{\alpha_1^2} ,\quad
\alpha_4=-\frac{b_1}{\alpha_1}.
 \end{equation}
An straightforward computation shows that
 $$
 \frac{\alpha_i\alpha_j}{\alpha_{[i+j]_5}}=a_{ij},\quad   {\alpha_1\alpha_4}=-b_1, \quad   {\alpha_2\alpha_3}=-b_2, 
 $$
 for any $0\le i\le j\le 4$ such that $(i,j)\in\{ (1,1),(2,2),(3,3),(4,4),(1,2),(1,3),(2,4),(3,4) \}$, where    $b_i$ is used for both $ b_i^{(1)}$ and $b_i^{(2)}$.
 Thus $f$ is an isomorphism between $\mcL$ and $\mcL'$. 
 Recall from Remark~\ref{re_haysol} that  there is some solution of the system of quadratic equations  \eqref{eq_SOL} such that the corresponding Lie algebra $\mcL$ is the split algebra      $\e_8$.
Hence all the obtained Lie algebras  are  isomorphic to $\e_8$ and so, in particular, they are simple. 
\end{proof}

\begin{remark}\label{re_simplicidad}
The   construction in Proposition~\ref{teo_Z5model} provides simple Lie algebras of type $E_8$ even when the field is not algebraically closed, although in such case  the obtained Lie algebras are not necessarily isomorphic. More precisely, we cannot take the scalars as in \eqref{eq_simplicidad} 
if the field does not contain fifth roots of any element. We can take fifth roots in  the real field, so that only one of the three real forms of the complex Lie algebra of type $E_8$ can be obtained by Proposition~\ref{teo_Z5model}  (the split one, of signature of its Killing form equal to 8, taking into account Section~\ref{se_gradings}). For an arbitrary field, the simplicity can be argued as follows. If   an ideal $I$  of $\mcL$ possesses a nonzero homogeneous element in $ \mcL_{\bar i}$, for some $i\ne0$, since $\mcL_{\bar i}$ is an irreducible $\mcL_{\bar 0}$-module, then $\mcL_{\bar i}\subset I$ and $I=\mcL$. 
 If  $\mcL$ possesses a nonzero   element in $ \mcL_{\bar 0}$, at least one of the copies of $\slf(V)$ is contained in $I$, and multiplying for instance with 
 $ \mcL_{\bar 1}$ we get $ \mcL_{\bar 1}\subset I$ and again $I$ is the whole algebra $\mcL$. Finally, aiming for a contradiction, suppose that  certain nonzero ideal $I$  contains  no  nonzero homogeneous element. Let us check first that any nonzero element  $x=\sum  x_i\in I$ ($x_i\in\mcL_{\bar i}$), gives another nonzero element $x'=\sum_{i\ne{  0}} x'_i\in I$ such that $x_i'=0$ for all the indices $i$ such that $x_i=0$. In fact,  as each copy of $\slf(V)$ acts faithfully in $ \mcL_{\bar i}$ if $i\ne{  0}$ but centralizes the other copy of  $\slf(V)$; then we multiply $x$ first with one of the copies of  $\slf(V)$ and afterwards with the other copy to get the required nonzero element $x'$ without term in $ \mcL_{\bar 0}$. Next, we multiply $x'$ with $ \mcL_{\bar 1}$ to get a nonzero element in $[\sum_{i\ne{\bar 0}}\mcL_{\bar i},\mcL_{\bar 1}]\subset \mcL_{\bar 2}+\mcL_{\bar 3}+\mcL_{\bar 4}+\mcL_{\bar 0}$ and at the same time in $I$. By repeating the above procedure to eliminate the ${\bar 0}$-part, we obtain a nonzero element in $(\mcL_{\bar 2}+\mcL_{\bar 3}+\mcL_{\bar 4})\cap I$. 
 Again we multiply such element with $\mcL_{\bar 1}$ to get a nonzero element in $ (\mcL_{\bar 3}+\mcL_{\bar 4}+\mcL_{\bar 0})\cap I$, and afterwards we eliminate its part in $\mcL_{\bar 0}$. At some point we will obtain  a nonzero homogeneous element in $I$, getting the desired contradiction.
\end{remark}

\begin{remark}\label{re_Z5cubo}
It is well known the existence of a $\ZZ_p^3$-graded Lie algebra with zero neutral homogeneous component and all the remaining $p^3-1$ homogeneous components of dimension 2 (pieces of Cartan subalgebras), for $p=2,3,5$ and the obtained Lie algebra $\mathfrak{g}_2$, $\mathfrak{f}_4$ and $\mathfrak{e}_8$ respectively (dimensions $2(p^3-1)$). 
These are examples of the Jordan gradings described in \cite{Alb_Jordan gradings} produced by the Jordan subgroups in \cite{Alekseevski-hno} (see also \cite[3.8]{EMSIII}).
The related highly symmetric model of $\mathfrak{g}_2$ as a twisted ring group based on this $\ZZ_2^3$-grading has been very recently obtained in \cite{twistedg2}.
The challenge of getting a  similar construction of $\e_8$ by doubling and twisting the group $\ZZ_5^3$ could have as a convenient starting point the concrete products in Proposition~\ref{teo_Z5model}, as explained in \cite[Theorem~2.1]{Alb_Jordan gradings}.
\end{remark}
 

 \subsection{Model based on the subalgebra of type $A_1\oplus A_7$}\label{se_A1A7}

 Recall from Section~\ref{se_gradings} the existence of $U$ and $W$ vector spaces over $\FF$ with $\dim U=2$ and $\dim W=8$ such that the split Lie algebra 
 $\mcL=\e_8=\sum_{\bar i\in\ZZ_4}\mcL_i$ is
   $\ZZ_4$-graded  
 with $\mcL_{\bar 0}\cong\slf(U)\oplus\slf(W)$ 
and   
$$
\textstyle
\mcL_{\bar 1}\cong U\otimes \bigwedge^2W,\quad
\mcL_{\bar 2}\cong \FF\otimes \bigwedge^4W,\quad
\mcL_{\bar 3}\cong U\otimes \bigwedge^6W,
$$
isomorphisms of  $\mcL_{\bar 0}$-modules.
According to Theorem~\ref{teo_main}, we can recover the products of $\e_8$ from these pieces since the brackets 
$[\mcL_{\bar i},\mcL_{\bar j}]\subset \mcL_{\bar i+\bar j}$ are determined up to scalars. 
We use again the invariant products described in Section~\ref{se_invariantproducts}, to be precise, 
$*$ in Eq.~\eqref{eq_starproduct},
$(\ ,\ )$ in Eq.~\eqref{eq_laforma} and $[\ ,\ ]$ in Eq.~\eqref{eq_dualizando}.
In Remark~\ref{re_dim2}, these three products are expressed for the vector space $U$ in   terms of an alternating form $\varphi\colon U\times U\to \FF$.

\begin{proposition} \label{teo_Z4model}

Let $\mcL$ be the  $\ZZ_4$-graded  vector space given, for any $i=1,2,3$,  by
\begin{equation}\label{ecumodel_Z4}
\begin{array}{ll}
\mcL_{\bar 0}&=\slf(U)\oplus\slf(W)\\
\mcL_{\bar i}&=\bigwedge^{[i]_2}U\otimes\bigwedge^{[2i]_8}W\qquad
\end{array}
\end{equation}
 where $U$ and $W$  denote vector spaces over $\FF$ of dimensions $2$ and $8$, respectively. Consider the product of $\mcL_{\bar 0}$ with $\mcL_{\bar i}$ given by the action of the $j$th ideal of $\mcL_{\bar0}$ on the $j$th slot ($j=1,2$). Consider the   product of $\mcL_{\bar i}$ with $\mcL_{\bar j}$, for   $1\le i\le j\le 3$, given by
\begin{equation}\label{eq_prA7}
[x_1\otimes x_2,y_1\otimes y_2]=\left\{
\begin{array}{ll}
a_{ij}\,x_1*y_1\otimes x_2*y_2& i+j\ne4, \\
b_1^{(1)}[x_1,y_1](x_2,y_2)+b_1^{(2)}(x_1,y_1)[x_2,y_2]\qquad&(i,j)=(1,3),\\
b_2^{(2)}(x_1,y_1)[x_2,y_2]\qquad& (i,j)=(2,2),
\end{array}\right.
\end{equation}
for    
  any  $x_1\in\bigwedge^{[i]_2} U$, $x_2\in\bigwedge^{[2i]_8}W$, $y_1\in\bigwedge^{[j]_2} U$ and $y_2\in\bigwedge^{[2j]_8}W$.
 
 \noindent Extend the bracket to $\mcL$ making that the product of $\mcL_{\bar i}$ with $\mcL_{\bar j}$ if $i>j$ is skew-symmetric, and making $(\mcL_{\bar 0},[\ ,\ ])$ a subalgebra.
Assume that all the scalars are nonzero.
Then, $\mcL$ endowed with the  product $[\ ,\ ]$   is a Lie algebra if and only if 
\begin{equation}\label{caso2_0}
\begin{array}{l}
b_1^{(1)}=b_1^{(2)}=-a_{11}a_{23}=-a_{12}a_{33}, \\
b_2^{(2)}=-a_{12}a_{23}.
\end{array}
\end{equation}

Moreover, $(\mcL, [\ ,\ ])$ is a simple exceptional Lie algebra isomorphic to $\e_8$. 
A solution of this system of quadratic equations that, in particular, gives a model of $\e_8$, is
\begin{equation}\label{solcaso2}
\begin{array}{rl}
1&=a_{11}=a_{12}=a_{23}=a_{33},\\
-1&=b_{2}^{(2)}=b_{1}^{(1)}=b_{1}^{(2)}.
\end{array}
\end{equation}
\end{proposition}

 \begin{proof}
 Again it makes sense to deal with the skew-symmetric extension of the bracket because, for $i=j$ it is true that $[x_1\otimes x_2,y_1\otimes y_2]=-[y_1\otimes y_2,x_1\otimes x_2]$. In fact,  if $i=1,3$, $x_1,y_1\in  U$, $x_2,y_2\in\bigwedge^{[2i]_8}W$, and Lemma~\ref{le_notacionesyproductos} says that $x_1*y_1=-y_1*x_1$ while $x_2*y_2=y_2*x_2$. For $i=2$, $x_1,y_1\in  \FF$ so that $(x_1,y_1)=x_1y_1=(y_1,x_1)$. In this case $x_2,y_2\in\bigwedge^{4}W$, and $[x_2,y_2]=-[y_2,x_2]$   since $(\ ,\ )\colon \bigwedge^{4}W\times \bigwedge^{4}W\to\FF$ is symmetric.
 
 Similarly to the proof of Proposition~\ref{teo_Z5model}, we have to prove the identity $ (J_{i,j,k})$: 
 $$
 J(\mcL_{\bar i},\mcL_{\bar j},\mcL_{\bar k})=0 \textrm{ for } 0\le i,j,k\le 3.
 $$
  Again we can assume that $1\le i,j,k\le 3$
    by Remark~\ref{remark:slv}, we can also assume $i\le j\le k$ by trilinearity, and the cases $i=j=k$ are consequence of  
 Remark~\ref{re_haysol} adapted to our setting: the veracity of  $ (J_{i,i,i})$ does not depend on the scalars and  the system of equations with variables $a_{ij},b_i^{(1)},b_i^{(2)}\in\FF^\times$ always has at least one solution.
 
 For checking the identities, we can restrict the considered elements in $\bigwedge^{[2l]_8}W$ to those which are 
   wedge products of elements in $W$. Take  $\{ v_i:i\in\mathbb N\}$ arbitrary elements in $W$, and use  the notation $v_{i_1\dots i_k}$ for $v_{i_{1}}\wedge\dots\wedge v_{i_{k}}$ as in Lemma~\ref{le_notacionesyproductos} (and similarly for $u_i$'s, elements in $U$). Also fix  $\{e_i:i=1,\dots,8\}$ a basis of $W$ and $\{f_1,f_2\}$ a basis of $U$. There is no ambiguity with the notation $f+ g$ in $\mcL_{\bar0}$,  for $f\in\mathfrak{sl}(U)$ and in $g\in\mathfrak{sl}(W)$. (We will not use subindices $f_1+g_2$ in this case.)
\medskip

 \noindent
 $ \boxed{ (J_{1,1,2})}$ Take $x=u_1\otimes v_{12},y=u_2\otimes v_{34}\in\mcL_{\bar1}$ and $z=1\otimes v_{5678}\in\mcL_{\bar2}$. Compute
$$
 \begin{array}{ll}
J(x,y,z) =&   -a_{11}b_2^{(2)}\phi(u_{12})[v_{1234},v_{5678}]\\
& -a_{12}b_1^{(1)}\phi(v_{1\dots8})[u_1,u_2]+a_{12}b_1^{(2)}\phi(u_{12})[v_{12},v_{345678}]\\
& +a_{12} b_1^{(1)}\phi(v_{1\dots8})[u_2,u_1]+a_{12} b_1^{(2)}\phi(u_{12})[v_{34},v_{125678}].
\end{array} 
$$
 
 As $[u_1,u_2]=[u_2,u_1]$; then
 \begin{equation}\label{eq_caso112}
J(x,y,z) = \phi(u_{12})\big(  -a_{11}b_2^{(2)}[v_{1234},v_{5678}] 
   +a_{12}b_1^{(2)} ( [v_{12},v_{345678}]
+[v_{34},v_{125678}])\big).
\end{equation}

 If we choose concrete elements, for instance,  $x= f_1 \otimes e_{12}$, $y=f_2 \otimes e_{34}$ and $z=1\otimes e_{5678}$,  
  we use Lemma~\ref{le_notacionesyproductos} (b) to get
 $$
 [e_{12},e_{345678}]+[e_{34},e_{125678}]=[e_{1234},e_{5678}],
 $$
  so that
 Eq.~\eqref{eq_caso112} becomes
$$
J(x,y,z) =\big(  -a_{11}b_2^{(2)}+a_{12}b_1^{(2)}\big)[e_{1234},e_{5678}] ,
$$
which gives a necessary condition to have  the identity $(J_{1,1,2})$:
\begin{equation}\label{caso2_1}
a_{11}b_2^{(2)}=a_{12}b_1^{(2)}.
\end{equation} 
Note that Eq.~\eqref{caso2_1} is also sufficient to get $(J_{1,1,2})$, since this condition placed on  Eq.~\eqref{eq_caso112} gives
$$
J(x,y,z) =  a_{11}b_2^{(2)}\phi(u_{12})\big( 
    [v_{12},v_{345678}]
+[v_{34},v_{125678}]  -[v_{1234},v_{5678}]  \big), 
$$
so that to be zero (for all $x,y \in\mcL_{\bar1}$ and all $z\in\mcL_{\bar2}$) or not does not depend on the choice of the scalars. That is, as so many times in the proof of Proposition~\ref{teo_Z5model},  knowing the existence of solutions  has made it possible for us  to avoid the direct verification of 
$$
 [v_{12},v_{345678}]+[v_{34},v_{125678}]=[v_{1234},v_{5678}],
 $$
 having   had to check
   it only for one concrete basic element. (The direct verification is not difficult at all, it is enough to use \eqref{eq_dualizando} and the concrete action of $\slf(W)$ on $\bigwedge^rW$, but it is very convenient   for us to mechanize and shorten  the arguments, realizing that 
 we can often omit the  computation with arbitrary elements.)
 \medskip

 \noindent
 $ \boxed{ (J_{1,2,2})}$
 For $x=u\otimes x_2\in\mcL_{\bar1}$ and $y=1\otimes y_2,z=1\otimes z_2\in\mcL_{\bar2}$,  
   \begin{equation}\label{eq_caso122}
 J(x,y,z) =u\otimes\big(b_2^{(2)}[y_2,z_2](x_2)+a_{12}a_{23}\big(y_2\,\llcorner\, \widetilde{x_2\wedge z_2}-z_2\,\llcorner\, \widetilde{x_2\wedge y_2}
 \big)
 \big).
\end{equation}
 This is easy to compute using concrete elements, taking into account Lemma~\ref{le_notacionesyproductos}. 
For instance,
$$
J(u\otimes e_{12},1\otimes e_{3456},1\otimes e_{1278}) = (-a_{12}a_{23}-b_2^{(2)})u \otimes e_{12}.
 $$

 \noindent 
 Hence, we get another necessary condition
\begin{equation}\label{caso2_3}
a_{12}a_{23}=-b_2^{(2)}.
\end{equation} 
This is sufficient too, since under this restriction on the scalars, Eq.~\eqref{eq_caso122} becomes 
$$
J(x,y,z) =b_2^{(2)} u\otimes\big([y_2,z_2](x_2)+  y_2\,\llcorner\, \widetilde{x_2\wedge z_2}-z_2\,\llcorner\, \widetilde{x_2\wedge y_2}
 \big),
$$
and now the identity $(J_{1,2,2})$ is independent of the choice of the scalars, and, as above,  it has to be  always true. \smallskip

\noindent
 $ \boxed{(J_{1,2,3}), (J_{1,1,3})}$ As the identity $ (J_{1,2,3})$ a priori involves $a_{12}a_{33}$, $a_{23}a_{11}$, $b_{1}^{(1)}$ and $b_{1}^{(2)}$,
we focus first on concrete elements. For instance, 
$$
J(f_1\otimes e_{12},1\otimes e_{3456},f_2\otimes e_{345678}) = (-a_{12}a_{33}-b_1^{(2)})1 \otimes e_{3456}.
 $$

 \noindent 
 This implies that a necessary condition to get $ (J_{1,2,3})$ is
\begin{equation}\label{caso2_4}
a_{12}a_{33}=-b_1^{(2)}.
\end{equation} 
Also,
$
J(f_1\otimes e_{12},f_{2}\otimes e_{34},f_1 \otimes e_{125678}) = \big(-a_{23}a_{11}-\frac{1}{2}(b_1^{(1)}+b_1^{(2)})\big)f_1 \otimes e_{12};
 $
so that 
\begin{equation}\label{caso2_2}
a_{11}a_{23}=-\frac{1}{2}(b_1^{(1)}+b_1^{(2)})
\end{equation} 
is necessary to get $ (J_{1,1,3})$. In particular, Eqs.~\eqref{caso2_1}, \eqref{caso2_3}, \eqref{caso2_4} and \eqref{caso2_2} are all necessary for $\mcL$ to be a Lie algebra. Let us check that these four equations imply Eq.~\eqref{caso2_0}.
First, from Eq.~\eqref{caso2_2} we have that $a_{11}a_{23}a_{12}=-\frac{1}{2}(b_1^{(1)}+b_1^{(2)})a_{12}$ and from Eq.~\eqref{caso2_3} we get that $-a_{11}b_2^{(2)}=-\frac{1}{2}(b_1^{(1)}+b_1^{(2)})a_{12}$. Now, applying Eq.~\eqref{caso2_1} we obtain $b_1^{(1)}=b_1^{(2)}$. We rewrite \eqref{caso2_2} as 
$a_{11}a_{23}=-b_1^{(1)}$ and then we have all the conditions required in Eq.~\eqref{caso2_0}. 

For the converse, assume that the scalars satisfy Eq.~\eqref{caso2_0} and let us check that $\mcL$ satisfies the Jacobi identity. This gives immediately \eqref{caso2_3}, \eqref{caso2_4} and \eqref{caso2_2},
and, as $a_{11}a_{23} b_2^{(2)}=-b_1^{(2)}b_2^{(2)}=b_1^{(2)}a_{12}a_{23}$, we have \eqref{caso2_1} too. This guarantees $(J_{1,1,2})$ and $(J_{1,2,2})$, at the moment. But also $(J_{1,2,3})$ holds, since $\frac{J(x,y,z)}{b_1^{(2)}}$ can be written without involving any of the variables, for any $x\in\mcL_{\bar1}$, $y\in\mcL_{\bar2}$ and $z\in\mcL_{\bar3}$. In a similar way, $ (J_{1,1,3})$ is true, since it only involves $a_{11}a_{23}$,   $b_{1}^{(1)}$ and $b_{1}^{(2)}$, which are all related.
The rest of the identities hold:
\begin{itemize}
    \item $(J_{1,3,3})$ follows from $a_{33}a_{12}=-b_1^{(1)}=-b_1^{(2)}$;
    \item $(J_{2,2,3})$ follows from $a_{23}a_{12}=-b_2^{(2)}$;
    \item $(J_{2,3,3})$ holds, since $a_{23}b_1^{(2)}=a_{23}b_1^{(1)}=b_2^{(2)}a_{33}$;
\end{itemize}

\noindent which finish the proof that $\mcL$ is a Lie algebra. 
Moreover, it is   evident that  the choice of scalars  in \eqref{solcaso2},
$$
1=a_{11}=a_{12}=a_{23}=a_{33},\qquad
-1=b_{2}^{(2)}=b_{1}^{(1)}=b_{1}^{(2)},
$$
provides a concrete solution of the system   of equations \eqref{caso2_0}.

Finally we prove that all the   Lie algebras obtained for different solutions of \eqref{caso2_0} are isomorphic. To do that, take $\mcL$ and $\mcL'$  the $\ZZ_4$-graded vector spaces as in \eqref{ecumodel_Z4} endowed with the Lie bracket as in \eqref{eq_prA7} with the coefficients given by  \eqref{caso2_0} and 
 \eqref{solcaso2} respectively. Let $ f\colon\mcL\to\mcL'$ be the map  given by $f(x)=\alpha_i x$, for any $x\in\mcL_{\bar i}$, where $\{\alpha_i:i=0,\dots,3\}\subset\FF^\times$ are chosen such that
\begin{equation}\label{eq_simplicidad1}
\alpha_0=1,\quad
\alpha _1^4=-b_2 a_{11}^2\, ,\quad
\alpha_2=\frac{-b_1a_{12}}{\alpha_1^2}\,,\quad
\alpha_3=\frac{-b_1 }{\alpha_1}.
 \end{equation}
An straightforward computation shows that
 $$
 \frac{\alpha_i\alpha_j}{\alpha_{[i+j]_4}}=a_{ij},\quad   {\alpha_1\alpha_3}=-b_1, \quad   {\alpha_2\alpha_2}=-b_2, 
 $$
 for any $0\le i\le j\le 3$ such that $(i,j)\in\{ (1,1),(1,2),(2,3), (3,3) \}$, where    $b_i$ is used for both $ b_i^{(1)}$ and $b_i^{(2)}$.
  Thus $f$ is an isomorphism between $\mcL$ and $\mcL'$. 
  Since one of the solutions of the system  \eqref{caso2_0}  makes the corresponding algebra Lie algebra $\mcL$ isomorphic to  $\e_8$, then
 all the obtained Lie algebras are simple and all of them are  isomorphic to $\e_8$. 
\end{proof}
 
 \begin{remark}
 The simplicity for arbitrary fields could have been argued  as  in Remark~\ref{re_simplicidad}, but again the uniqueness is no longer true. 
 This time, not even in the real field: 
 the arguments in \eqref{eq_simplicidad1} do not work because of the nonexistence of fourth roots in $\mathbb{R}$, which seems to mean that the Lie algebras obtained for different 
 values in  \eqref{caso2_0} may not be isomorphic  and that we could obtain both the real forms related to signatures $8$ and $-24$ for different solutions of \eqref{caso2_0}. Further calculations could be   done to be a bit more precise here.
 \end{remark}


 \subsection{Model based on the subalgebra of type $A_1\oplus A_2\oplus A_5$}\label{seZ6}

   In this case,   Section~\ref{se_gradings} implies the existence of  $U$, $V$ and $W$ vector spaces over $\FF$ with $\dim U=2$, $\dim V=3$  and $\dim W=6$ such that the Lie algebra 
 $\mcL=\e_8=\sum_{\bar i\in\ZZ_6}\mcL_{\bar i}$ is
   $\ZZ_6$-graded  
 with $\mcL_{\bar 0}\cong\slf(U)\oplus\slf(V)\oplus\slf(W)$ 
and   
$$
\begin{array}{c}
\textstyle 
\mcL_{\bar 1}\cong U\otimes V \otimes  W,\quad
\mcL_{\bar 2}\cong \FF\otimes \bigwedge^2V\otimes \bigwedge^2W,\quad
\mcL_{\bar 3}\cong U\otimes \FF \otimes \bigwedge^3W,\vspace{4pt}
\\
\mcL_{\bar 4} \cong \FF\otimes V \otimes  \bigwedge^4W,\quad
\mcL_{\bar 5}\cong U\otimes \bigwedge^2V\otimes \bigwedge^5W,
\end{array}
$$
where $\cong$ means isomorphisms as  $\mcL_{\bar 0}$-modules.
According to Theorem~\ref{teo_main},   the Lie algebra $\e_8$ can be recovered from these pieces since the brackets 
$[\mcL_{\bar i},\mcL_{\bar j}]\subset \mcL_{\bar i+\bar j}$ are determined up to scalars. 
We again make use of   
$*$ in Eq.~\eqref{eq_starproduct},
$(\ ,\ )$ in Eq.~\eqref{eq_laforma} and $[\ ,\ ]$ in Eq.~\eqref{eq_dualizando}.

\begin{proposition} \label{teo_Z6model}
Let $U$, $V$ and $W$ be vector spaces over $\FF$ of dimension $2$, $3$ and $6$ respectively. 
Let $\mcL$ be the  $\ZZ_6$-graded  vector space given by 
\begin{equation}\label{model_Z6}
\begin{array}{ll}
\mcL_{\bar 0}&=\slf(U)\oplus\slf(V)\oplus\slf(W),\\
\mcL_{\bar i}&=\bigwedge^{[i]_2}U\otimes\bigwedge^{[i]_3}V \otimes \bigwedge^{[i]_6}W,\qquad
\end{array}
\end{equation}
 for any $i=1,\dots,5$. Consider the product of $\mcL_{\bar 0}$ with $\mcL_{\bar i}$ given by the action of the $j$th ideal of $\mcL_{\bar0}$ on the $j$th slot ($j=1,2,3$). 
 Fix some nonzero scalars $a_{ij},b_k^{(l)} \in\FF^\times$ for any  $1\le i\le j\le 5$, $i+j\ne6$, $k,l=1,2,3$.\footnote{Actually, 
 there is no need to consider $b_2^{(1)} $ and $b_3^{(2)}$, since the following brackets vanish:   $[\ ,\ ]\colon\bigwedge^i U\times\bigwedge^{n-i} U\to\slf(U)$ if $i=2$ and $[\ ,\ ]\colon\bigwedge^i V\times\bigwedge^{n-i} V\to
\slf(V)$ if $i=3$. But this makes easier to write   \eqref{casoYol}.}
 Consider the   product of $x=x_1\otimes x_2 \otimes x_3\in\mcL_{\bar i}$ with $y=y_1\otimes y_2 \otimes y_3\in\mcL_{\bar j}$, for   $1\le i\le j\le 5$, given by
\begin{equation}\label{casoYol}
[x,y]=\left\{
\begin{array}{ll}
a_{ij}\,x_1*y_1\otimes x_2*y_2\otimes x_3*y_3&i+j\ne6,\vspace{2pt} \\
\sum_{k=1}^3 b_i^{(k)}[x_k,y_k]\Pi_{r=1,r\ne k}^3(x_r,y_r) \ \qquad & i+j=6.
\end{array}\right.
\end{equation}

\noindent Extend the bracket to $\mcL$ making that the product of $\mcL_{\bar i}$ with $\mcL_{\bar j}$ if $i>j$ is skew-symmetric, and making $(\mcL_{\bar 0},[\ ,\ ])$ a subalgebra.
Then, $\mcL$ endowed with the  product $[\ ,\ ]$   is a Lie algebra if and only if 
\begin{equation}\label{eq_ecuacionestercero}
\begin{array}{l}
b_1^{(1)}=b_1^{(2)}=b_1^{(3)}=-a_{11}a_{25}=-a_{12}a_{35}=-a_{13}a_{45}=-a_{14}a_{55}, \vspace{2pt}\\
b_2^{(2)}=b_2^{(3)}=a_{12}a_{34}=a_{14}a_{25}=-a_{22}a_{44}=-a_{23}a_{45}, \vspace{2pt}\\
b_3^{(1)}=b_3^{(3)}=a_{13}a_{34}, \vspace{2pt}\\
a_{11}a_{22}=-a_{12}a_{13}.
\end{array}
\end{equation}
Moreover, $(\mcL, [\ ,\ ])$ is a simple exceptional Lie algebra isomorphic to $\e_8$. 
A solution of this system of   equations that, in particular, gives a model of $\e_8$, is
\begin{equation}\label{solYol}
\begin{array}{rl}
1&=a_{11}=a_{13}=a_{22}=a_{23}=a_{25}=a_{34}=a_{44}=a_{45}=b_{3}^{(1)}=b_{3}^{(3)},\\
-1&=a_{12}=a_{14}=a_{35}=a_{55}=b_{1}^{(1)}=b_{1}^{(2)}=b_{1}^{(3)}=b_{2}^{(2)}=b_{2}^{(3)}.
\end{array}
\end{equation}
\end{proposition}

\begin{proof}

First of all, we will check that all the conditions in \eqref{eq_ecuacionestercero} are necessary for $\mcL$ to be a Lie algebra, simply by applying the Jacobi identity to several triples of elements to achieve the required conditions. 
We abuse a little bit of the notations, by using $\{e_1,e_2\}$, $\{e_1,e_2,e_3\}$ and $\{e_1,\dots,e_6\}$ bases of $U$, $V$ and $W$, respectively. When we do a tensor product, there is no confusion in the position of the elements. Also, in case there is some ambiguity, we denote a linear map $F\colon U\to U$ with a subindex, $F_{_U}$, and similarly for endomorphisms of $V$ and $W$. 
Hence assume now that $\mcL$ is a Lie algebra.

$\star$ We compute, for    $x=e_1 \otimes e_1 \otimes e_1$, $y=e_2 \otimes e_2 \otimes e_2$ and 
$$
\begin{array}{ll}
z=e_2 \otimes e_{13} \otimes e_{13456}&\Rightarrow\ J(x,y,z) = (a_{11}a_{25}+b_1^{(1)})e_2 \otimes e_{1}\otimes e_{1},\\
z=e_1 \otimes e_{23} \otimes e_{13456}&\Rightarrow\ J(x,y,z) = (a_{11}a_{25}+b_1^{(2)})e_1 \otimes e_{2}\otimes e_{1},\\
z=e_2 \otimes e_{23} \otimes e_{13456}&\Rightarrow\ J(x,y,z) = (a_{11}a_{25}+b_1^{(3)})e_2 \otimes e_{2}\otimes e_{1},\\
z=1 \otimes e_{13} \otimes e_{34}&\Rightarrow\ J(x,y,z) = -(a_{11}a_{22}+a_{12}a_{13})1 \otimes e_1 \otimes e_{1234},\\
z=e_1 \otimes 1 \otimes e_{345}&\Rightarrow\  J(x,y,z) = -(a_{11}a_{23}+a_{13}a_{14})e_1 \otimes e_{12}\otimes e_{12345}.
\end{array}
$$
The Jacobi identity so implies that $b_1^{(1)}=b_1^{(2)}=b_1^{(3)}$, which we will denote as $b_1$, as well as
\begin{align}
  a_{11}a_{25}=-b_1, \label{caso3_5}\\
   a_{11}a_{22}=-a_{12}a_{13}, \label{caso3_1}\\
  a_{11}a_{23}=-a_{13}a_{14}. \label{caso3_2}
\end{align} 
 \noindent 
Use Lemma~\ref{le_notacionesyproductos} (b) to check that  $[e_{1},e_{23}]_{_V}-[e_{2},e_{13}]_{_V}=[e_{12},e_3]_{_V}$ and  $[e_1,e_{23456}]_{_W}-[e_2,e_{13456}]_{_W}=-[e_{12},e_{3456}]_{_W}$, and recall that $[\ ,\ ]_{_U}$ is symmetric, to get
$$
J(x,y,1 \otimes e_{3} \otimes e_{3456})=-( a_{11}b_2^{(2)}+a_{14}b_1 )[e_{12},e_{3}]_{_V}-(a_{11}b_2^{(3)}+a_{14}b_{1} )[e_{12},e_{3456}]_{_W}.
$$
Again its projections on $\slf(V)$ and $\slf(W)$ are both zero, so that $b_2^{(2)}=b_2^{(3)}=:b_2$ and
\begin{equation}\label{caso3_3_1}
    a_{11}b_2 =-a_{14}b_1.
\end{equation}

$\star$ Now take       $x=e_1 \otimes e_1 \otimes e_1$, $y=1 \otimes e_{23} \otimes e_{23}$, 
and compute 
$$
\begin{array}{ll}
J(x,y,e_2 \otimes 1 \otimes e_{456})=&(a_{23}b_1-a_{12}b_3^{(1)})[e_1,e_2]_{_U}+(a_{13}b_2-a_{23}b_1)[e_1,e_{23}]_{_V}
\\
&+(a_{23}b_1-a_{12}b_3^{(3)} )[e_1,e_{23456}]_{_W}-(a_{13}b_2-a_{12}b_3^{(3)} )[e_{23},e_{1456}]_{_W};
\end{array}
$$
taking into account that 
 $[e_1,e_{23}]_{_V}=-[e_{23},e_1]_{_V}$,   and $[e_1,e_{23456}]_{_W}-[e_{23},e_{1456}]_{_W}=[e_{123},e_{456}]_{_W}$. 
 The three projections   are zero, so that $b_3^{(1)}=b_3^{(3)}=:b_3$ and
  \begin{equation} \label{caso3_7}
 a_{23}b_1=a_{12}b_3=a_{13}b_2. 
 \end{equation}
Choosing a different  third element,
$$
\begin{array}{ll}
z=1 \otimes e_1 \otimes e_{1456}&\Rightarrow  J(x,y,z)=(a_{12}a_{34}-b_2)e_1 \otimes e_1 \otimes e_1,\\
z=e_2 \otimes e_{12} \otimes e_{12456}&\Rightarrow  J(x,y,z)=(-a_{12}a_{35}+a_{11}a_{25})1 \otimes e_{12} \otimes e_{12},
\end{array}
$$
   we obtain 
 \begin{equation}\label{caso3_7_0}
 a_{12}a_{34}=b_2,\qquad  a_{12}a_{35}=a_{11}a_{25}\stackrel{\eqref{caso3_5}}=-b_1.      
 \end{equation}
 
  \noindent
 $\star$  More necessary conditions appear when considering $x=e_1 \otimes e_1 \otimes e_1$, $y=e_2 \otimes 1 \otimes e_{234}$. We compute
 $$
\begin{array}{ll}
z=1 \otimes e_{2} \otimes e_{2356}&\Rightarrow  J(x,y,z)=-(a_{13}a_{44}-a_{14}a_{35} )1\otimes e_{12}\otimes e_{23}, \\
z= e_1 \otimes e_{23} \otimes e_{12356}&\Rightarrow  J(x,y,z)=(a_{13}a_{45}-a_{35}a_{12})e_1\otimes 1\otimes e_{123},
\end{array}
$$
which yields  
 \begin{equation} \label{caso3_8}
 a_{13}a_{44}=a_{14}a_{35},\qquad a_{13}a_{45}=a_{35}a_{12}.     
 \end{equation} 

  $\star$   Finally  $J(e_1 \otimes e_1 \otimes e_1,1 \otimes e_2 \otimes e_{2345},e_2 \otimes e_{13} \otimes e_{12346})=(a_{14}a_{55}-a_{45}a_{13})1\otimes e_1\otimes e_{1234}$, so that 
  \begin{equation}\label{old47}
  a_{14}a_{55}=a_{45}a_{13}.
  \end{equation}  
   Keeping in mind   this information, next we prove that all the equations in \eqref{eq_ecuacionestercero} are necessary. 
From
$
a_{12}b_3\stackrel{\eqref{caso3_7}}=a_{13}b_2\stackrel{\eqref{caso3_7_0}}=a_{13}a_{12}a_{34},
$
we get
\begin{equation}\label{new48}
b_{3}=a_{13}a_{34}.
\end{equation}
Also, $a_{11}b_2\stackrel{\eqref{caso3_3_1}}=-a_{14}b_1 \stackrel{\eqref{caso3_5}}=a_{14}a_{11}a_{25}$, so that
\begin{equation}\label{new49}
b_{2}=a_{14}a_{25}.
\end{equation}
Note 
$a_{11}a_{23}a_{45}\stackrel{\eqref{caso3_2}}=-a_{14}a_{13}a_{45} \stackrel{\eqref{caso3_8}}=-a_{14}a_{35}a_{12}\stackrel{\eqref{caso3_7_0}}=-a_{14}a_{11}a_{25}$,
getting
\begin{equation}\label{new50}
 a_{23}a_{45}=-a_{14} a_{25}\stackrel{\eqref{new49}}=-b_2.
\end{equation}
Take also into account 
$a_{35}a_{12}a_{23}\stackrel{\eqref{caso3_7_0}}=-b_1a_{23}\stackrel{\eqref{caso3_7}}=-a_{12}b_3  $, so
\begin{equation}\label{new52}
b_{3}=-a_{23}a_{35};
\end{equation}
and $a_{22}a_{13}a_{14}\stackrel{\eqref{caso3_2}}=-a_{11}a_{22}a_{23}\stackrel{\eqref{caso3_1}}=a_{12}a_{13}a_{23}$, which gives
\begin{equation}\label{new53}
a_{22} a_{14}=a_{12} a_{23}.
\end{equation}
And we also  need $a_{13}b_2\stackrel{\eqref{caso3_7}}=a_{12}b_3\stackrel{\eqref{new52}}=-a_{35}a_{12}a_{23}
\stackrel{\eqref{new53}}=-a_{22}a_{14}a_{35}\stackrel{\eqref{caso3_8}}=-a_{22}a_{13}a_{44}   $, which yields
\begin{equation}\label{new51}
b_{2}=-a_{22}a_{44}.
\end{equation}
Finally observe that Eqs.~\eqref{caso3_5}, \eqref{caso3_1}, \eqref{caso3_7_0}, \eqref{caso3_8} ,\eqref{old47}, \eqref{new48}, \eqref{new49}, \eqref{new50}, \eqref{new51} give all the conditions in   \eqref{eq_ecuacionestercero}.
 That is, the conditions in  \eqref{eq_ecuacionestercero} are necessary for $\mcL$ to be a Lie algebra.\smallskip


Conversely, assume that we have nonzero scalars satisfying  \eqref{eq_ecuacionestercero}  and let us prove that $\mcL$ is a Lie algebra. We can talk about the  skew-symmetric extension because, for $i=j$, the defined bracket already satisfies $[x,y]=-[y,x]$. In fact,
\begin{itemize}
\item If $i=j\in\{1,5\}$, then $x_1*y_1=-y_1*x_1$, $x_2*y_2=-y_2*x_2$ and $x_3*y_3=-y_3*x_3$;  
\item If $i=j\in\{2,4\}$, then $x_1*y_1=y_1*x_1$, $x_2*y_2=-y_2*x_2$ and $x_3*y_3=y_3*x_3$; 
\item If $i=j=3$, then $[x_1,y_1]=[y_1,x_1]$, $(x_1,y_1)=-(y_1,x_1)$, $(x_2,y_2)=(y_2,x_2)$, and $[x_3,y_3]=[y_3,x_3]$, $(x_3,y_3)=-(y_3,x_3)$.
\end{itemize}
 
The constructed skew-symmetric   algebra $(\mcL, [\ ,\ ]  )$ is a Lie algebra if and only if $J(\mcL_{\bar i},\mcL_{\bar j},\mcL_{\bar k})=0$ for $0\le i,j,k\le 5$. We will again denote this identity by $(J_{i,j,k})$. By Remark \ref{remark:slv} applied to $\mathfrak m=\oplus_{i=1}^5\mcL_{\bar i}$, we have only to check
   $(J_{i,j,k})$  for $1\le i\le j\le k\le 5$. Furthermore,  if one establishes some relations  between $a_{ij}a_{[i+j]_6,k}$, $a_{jk}a_{[j+k]_6,i}$, and $a_{ki}a_{[k+i]_6,j}$, (where we abuse of the notation and understand  $a_{i,6-i}=b_i$ if $i=1,2,3$, and $a_{ij}=a_{ji}$ if $i>j$), then the identity  $(J_{i,j,k})$ holds. In fact,  $\frac{J(x,y,z)}{a_{ij}a_{[i+j]_6,k}}$ would be written without any appearance of the scalars, only   in terms of  $x$, $y$, $z$, $*$, $\sim$ and $\llcorner$; so that it would be always zero due to the existence of scalars making $\mcL$ a Lie algebra, as in   Remark~\ref{re_haysol}. For instance, $(J_{i,i,i})$ is always true since the three quantities are obviously equal, $a_{ii}a_{[2i]_6,i}$. In case two of the indices $i,j,k$ are repeated, we will only have one equation relating the scalars, instead of two.  
      All this means that, when we     check that  \eqref{eq_ecuacionestercero}   implies all the identities on the left of Table~\ref{tabla:caso3},  this will   guarantee the corresponding identities of type $(J_{i,j,k})$ and then the Jacobi identity.
\begin{table} 
\centering
\begin{tabular}{|c|c|}\hline
   Identity&Gives\\  \hline    $a_{11}a_{22}=-a_{12}a_{13}$&$(J_{1,1,2})$\\
    $a_{11}a_{23}=-a_{13}a_{14}$&$(J_{1,1,3})$\\
    $ a_{11}b_2 =-a_{14}b_1 $&$(J_{1,1,4})$\\
    $a_{11}a_{25} =-b_1$&$(J_{1,1,5})$\\
    $a_{12}a_{23}=a_{14}a_{22}$&$  (J_{1,2,2})$\\
     $a_{13}a_{34}=b_3$&$  (J_{1,3,3})$ \\
      $a_{14}a_{45}=a_{44}a_{12}$&$  (J_{1,4,4})$ \\
       $a_{14}a_{55}=-b_1$&$  (J_{1,5,5})$ \\
 $a_{22}a_{34}=a_{23}a_{25}$&$  (J_{2,2,3})$ \\ 
  $a_{22}a_{44}=-b_2$&$ (J_{2,2,4})$ \\
  $a_{22}a_{45}=-a_{25}a_{12}$& $  (J_{2,2,5})$ \\ 
   $a_{23}a_{35}=-b_3$&$  (J_{2,3,3})$\\
 $a_{22}a_{44}=-b_2$&$  (J_{2,4,4})$\\
 $a_{55}b_2=-a_{25}b_1$&$ (J_{2,5,5})$\\
 $a_{13}a_{34}=b_3$&$  (J_{3,3,4})$\\
 $ a_{35}a_{23}=-b_3 $ &$  (J_{3,3,5})$\\
 $a_{23}a_{44}=-a_{34}a_{14}$&$  (J_{3,4,4})$\\
 $a_{35}a_{25}=a_{55}a_{34}$&  $  (J_{3,5,5})$\\
 $a_{44}a_{25}=a_{45}a_{34}$& $  (J_{4,4,5})$ \\
 $a_{45}a_{35}=a_{55}a_{44}$& $  (J_{4,5,5})$ \vspace{2pt}\\
  \hline
   \end{tabular}\vspace{2pt}
   \qquad \begin{tabular}{|c|c|}\hline
   Identity&Gives \\    \hline 
   $a_{23}b_1=a_{12}b_3=a_{13}b_2$&  $  (J_{1,2,3})$\\
    $b_2=a_{12}a_{34}=a_{14}a_{25}$&  $  (J_{1,2,4})$\\
    $   a_{12}a_{35}=a_{11}a_{25}=-b_1  $&  $  (J_{1,2,5})$\\
    $a_{13}a_{44}=a_{14}a_{35}=a_{11}a_{34}$&$  (J_{1,3,4})$ \\
    $a_{13}a_{45}=a_{35}a_{12}=-b_1$&$  (J_{1,3,5})$ \\
     $a_{14}a_{55}=a_{45}a_{13}=-b_1$&$  (J_{1,4,5})$ \\
     $a_{23}a_{45}=-a_{34}a_{12}=-b_2$&$  (J_{2,3,4})$ \\
 $a_{23}a_{55}=a_{35}a_{22}=-a_{25}a_{13}$&$  (J_{2,3,5})$\\
$a_{25}a_{14}=-a_{23}a_{45}=b_2$&$ (J_{2,4,5})$\\
$-a_{34}b_1=a_{35}b_2=a_{45}b_3$&$ (J_{3,4,5})$\\
    \hline
   \end{tabular}
   \caption{Equivalent conditions to $J(\mcL_{\bar i},\mcL_{\bar j},\mcL_{\bar k})=0$.}
\label{tabla:caso3}
\end{table}

Proving all the  identities in Table~\ref{tabla:caso3}\, is a tedious but rather straightforward computation starting from \eqref{eq_ecuacionestercero}. First we obtain immediately the following useful equations:   
\begin{equation}\label{eqold}
 \dfrac{a_{44}}{a_{34}}=\dfrac{-a_{12}}{a_{22}}=\dfrac{a_{11}}{a_{13}}=\dfrac{a_{45}}{a_{25}}=\dfrac{-a_{14}}{a_{23}}.
\end{equation}
Next proceed case by case. The identity related to $(J_{1,1,2})$,  \eqref{caso3_1}, is one of the identities in the list \eqref{eq_ecuacionestercero}. 
The one related to  $(J_{1,1,3})$ follows from \eqref{eqold}.
Third, $a_{11}a_{25}b_2=-b_1b_2=-a_{14}a_{25}b_1$, so that we get $a_{11}b_2 =-a_{14}b_1$,
the identity necessary for $(J_{1,1,4})$. The restriction for $(J_{1,1,5})$ appears in \eqref{eq_ecuacionestercero}. 
For  $  (J_{1,2,2})$, we need  $a_{12}a_{23}=a_{14}a_{22}$, which follows from  \eqref{eqold}. 
 The two  identities necessary for $  (J_{1,2,3})$ are easy to recover:
 $a_{23}b_1=-a_{23}a_{13}a_{45}=a_{13}b_2$ and $a_{13}b_2=a_{13}a_{12}a_{34}=a_{12}b_3$. The identities for $  (J_{1,2,4})$,
 $  (J_{1,2,5})$, $  (J_{1,3,3})$, $  (J_{1,3,5})$, $  (J_{1,4,5})$ and $(J_{1,5,5})$  just appear    in the list. 
 Now,  $a_{13}a_{44}=a_{11}a_{34}$ comes from \eqref{eqold},
  but also
 $a_{11}a_{34}a_{12}=a_{11}a_{14}a_{25}=a_{14}a_{35}a_{12}$ gives
 $a_{11}a_{34}=a_{14}a_{35}$; so that 
 $  (J_{1,3,4})$ holds. 
 Also $(a_{14}a_{25})(a_{13}a_{45})=(-a_{44}a_{22})(a_{11}a_{25})=a_{44}a_{12}a_{13}a_{25}$ gives $  (J_{1,4,4})$.
 The identity for $(J_{2,2,3})$, $a_{34}a_{22} =a_{25} a_{23}$, follows from $a_{34}a_{22}a_{11}=-a_{34}a_{12}a_{13}=-a_{14}a_{25}a_{13}\stackrel{\eqref{eqold}}=a_{25}a_{11}a_{23}$. Next,
 $a_{22}a_{45}a_{14}\stackrel{\small{J_{144}}}=a_{22}a_{44}a_{12}=-a_{12}a_{25}a_{14}$, and, removing $a_{14}$, we get the identity in $(J_{2,2,5})$. Note $a_{12}b_3=b_1a_{23}=-a_{12}a_{35}a_{23}$, which gives $b_3=-a_{23}a_{35}$ and $(J_{2,3,3})$. The identities
 $(J_{2,2,4})$ and $(J_{2,3,4})$ belong to our original list. As regards the two equations in  $(J_{2,3,5})$,
 simply remove $a_{12}$ from $a_{12}a_{23}a_{55}\stackrel{\eqref{eqold}}=a_{22}a_{14}a_{55}=a_{22}a_{12}a_{35}$, as well as 
 remove $a_{14}$ from $a_{14}a_{55}a_{23}=a_{25}a_{11}a_{23}\stackrel{\eqref{eqold}}=-a_{25}a_{13}a_{14}$.
 Again $(J_{ 2,4,4})$, $(J_{2,4,5 })$ and $(J_{ 3,3,4})$ are direct from \eqref{eq_ecuacionestercero}.
 We  easily get $(J_{2,5,5 })$ as $a_{55}b_2=a_{55}a_{25}a_{14}=-a_{25}b_1$.
 Note $(J_{ 3,3,5})=(J_{2,3,3 })$. Both  $(J_{3,4,4 })$ and $(J_{4,4,5 })$  are direct  from \eqref{eqold}.
 From $a_{55}a_{12}a_{34}=a_{55}a_{25}a_{14}=a_{25}a_{12}a_{35}$, we derive $(J_{3,5,5 })$.
 Multiply the two identities above to have $(a_{35}a_{25})(a_{45}a_{34})=(a_{55}a_{34})(a_{44}a_{25})$ and simplify $a_{25}a_{34}$ to get $(J_{4,5,5 })$. Only $(J_{3,4,5 })$ is left, which is achieved from $a_{45}b_3=a_{45}a_{13}a_{34}=-b_1a_{34}$
 and
 $a_{35}b_2=-a_{35}a_{45}a_{23}\stackrel{J_{335}}=a_{45}b_3$. This finishes the proof that $\mcL$ is a Lie algebra.
 The fact that  \eqref{solYol} is a solution is plain.

Finally we show that all the   Lie algebras   obtained by this procedure are isomorphic. Let us give an isomorphism between $\mcL$ and $\mcL'$  the $\ZZ_6$-graded vector spaces as in \eqref{model_Z6} endowed with the Lie bracket as in \eqref{casoYol} with the coefficients given by   \eqref{eq_ecuacionestercero} and 
 \eqref{solYol} respectively. 
 An isomorphism can be provided by  $ f\colon\mcL\to\mcL'$, $f\vert_{\mcL_{\bar i}}=\alpha_i \id_{\mcL_{\bar i}}$,   by choosing some scalars $\{\alpha_i:i=0,\dots,5\}\subset\FF^\times$    such that
  $\alpha_1^6={-a_{11}a_{12}a_{13}a_{14}b_1}$ and 
$$\alpha_0=1,\quad 
\alpha_2=\frac{\alpha_1^2}{a_{11}},\quad
\alpha_3=  \frac{a_{13}a_{14}b_1}{\alpha_1^3} ,\quad
\alpha_4=\frac{a_{14}b_1}{\alpha_1^2},\quad
\alpha_5=\frac{-b_1}{\alpha_1}.$$
If we denote by $a_{i,6-i}=b_i$, the condition for $f$ to be isomorphism follows from 
 $$
 \frac{\alpha_i\alpha_j}{\alpha_{[i+j]_6}}=
 \begin{cases}a_{ij},\quad   \textrm{ if } (i,j)=(1,1),(1,3),(2,2),(2,3),(2,5),(3,3),(3,4),(4,4),(4,5),\\
 -a_{ij},\quad   \textrm{ if }(i,j)=(1,2),(1,4),(3,5),(5,5),(1,5),(2,4);
 \end{cases}
 $$
 which is  a routine calculation.
  Since one of the solutions of  \eqref{eq_ecuacionestercero} makes the corresponding  Lie algebra $\mcL$ isomorphic to  $\e_8$, then
 all the obtained Lie algebras are simple and all of them are  isomorphic to $\e_8$. 
 \end{proof} 

  

 \subsection{Model based on the subalgebra of type $4A_2$}\label{se_4A2}  
 
 The maximal elementary abelian $p$-subgroups of $E_8$ (from the viewpoint of algebraic groups) have been
obtained in  \cite{Griess}, where a $\ZZ_3^5$ subgroup of the group of type $E_8$, automorphisms of $\e_8$, is described. This 3-group coincides with its
centralizer, so producing a fine $\ZZ_3^5$-grading on the Lie algebra $\e_8$. If we consider the coarsening produced by any two of the  order 3 automorphisms involved,  the $\ZZ_3^2$-grading so obtained is toral, and hence it satisfies the hypothesis in Theorem~\ref{teo_main}.

 A concrete description of this $\ZZ_3^5$-grading on $\e_8$ can be found on \cite[\S6.3]{EK}, which starts with two $\ZZ_3^2$-graded Okubo algebras and constructs $\e_8$ from these algebras. Thus  $\e_8$ is naturally endowed with the grading produced when combining the two pairs of related order three automorphisms with the triality automorphism. This nice construction, due to A.~Elduque, makes use of symmetric composition algebras. It is clear that we can describe our $\ZZ_3^2$-grading as a coarsening of the above $\ZZ_3^5$-grading, simply by taking the 
 $\ZZ_3^2$-grading induced on $\e_8$  when taking only the $\ZZ_3^2$-grading on one of the Okubo algebras involved. But we would like to describe a model similar to all the others in this work.  The first step is to describe the homogeneous components of the $\ZZ_3^2$-grading. Recall from Section~\ref{se_gradings} that there are two types of order three automorphisms, which fix an algebra of type $\mathfrak{a}_8$ and   an algebra of type $\e_6\oplus\mathfrak{a}_2$, respectively. So the dimensions of the neutral homogeneous component are respectively, 80 and 86. Both the order three automorphisms involved now are of the second type. The $\ZZ_3$-grading is given by $\mcL=\mcL_{\bar0}\oplus \mcL_{\bar1}\oplus \mcL_{\bar2}$, for
 $$
 \mcL_{\bar0}=\e_6\oplus\mathfrak{sl}(V),\qquad
 \mcL_{\bar1}=V(\varpi_1)\otimes V,\qquad
  \mcL_{\bar2}=V(\varpi_1)^*\otimes V^*,
 $$
 if $V$ is a 3-dimensional vector space and $V(\varpi_1)$ denotes the $\e_6$-irreducible module of highest
weight $\varpi_1$, the first fundamental weight.
 Now,   recall the nice $\ZZ_3$-grading on $\mathcal{M}=\e_6$,   described in a very symmetric way as
 $$
 \e_6= \mathfrak{sl}(V)\oplus\mathfrak{sl}(V)\oplus\mathfrak{sl}(V)\,\oplus\, (V\otimes V\otimes V)\,\oplus\,(V^*\otimes V^*\otimes V^*).
 $$
 According to \cite{modsgrad}, the module $\mathcal{V}= V(\varpi_1)$ can be $\ZZ_3$-graded in a way compatible with the $\ZZ_3$-grading on $\e_6$, 
 that is, $\mathcal{M}_{\bar i}\cdot \mathcal{V}_{\bar j}\subset \mathcal{V}_{\bar i+\bar j}$, for
 $$
  \mathcal{V}_{\bar0}=\FF\otimes V\otimes V^*,\qquad 
  \mathcal{V}_{\bar1}=V\otimes  V^*\otimes  \FF, \qquad
  \mathcal{V}_{\bar2}=V^*\otimes  \FF\otimes  V.
  $$
  
\noindent  All this together gives the homogeneous components of the $\ZZ_3^2$-grading on $\e_8$ as modules for the neutral component
 $
 \mcL_{(\bar0,\bar0)}= \mathfrak{sl}(V)\oplus\mathfrak{sl}(V)\oplus\mathfrak{sl}(V)\oplus\mathfrak{sl}(V)$:
 $$
 \begin{array}{ll}
  \mcL_{(\bar0,\bar1)}=\FF\otimes  V\otimes  V^*\otimes  V\,, \qquad \qquad &
   \mcL_{(\bar0,\bar2)}=\FF\otimes  V^*\otimes  V\otimes  V^*\,,\\
   \mcL_{(\bar1,\bar0)}=V\otimes V\otimes V\otimes  \FF\,,&
   \mcL_{(\bar2,\bar0)}=V^*\otimes V^*\otimes V^*\otimes  \FF\,,\\
     \mcL_{(\bar1,\bar1)}=V \otimes V^*  \otimes \FF  \otimes V\,,  &
   \mcL_{(\bar2,\bar1)}= V^*\otimes  \FF \otimes   V\otimes  V\,, \\
     \mcL_{(\bar1,\bar2)}=V \otimes  \FF \otimes  V^* \otimes V^* \,, &
   \mcL_{(\bar2,\bar2)}= V^*\otimes  V \otimes  \FF \otimes   V^*\,.
   \end{array}
 $$
 Thus we have the $\ZZ_3^2$-graded vector space to be endowed with a Lie algebra structure. 
 
 \begin{proposition} \label{teo_Z32model}

Let $\mcL$ be the graded vector space given, for any $i,j=0,1,2$,  by
\begin{equation}\label{ecumodel_Z32}
\begin{array}{ll}
\mcL_{(\bar 0,\bar 0)}&=\mathfrak{sl}(V)\oplus\mathfrak{sl}(V)\oplus\mathfrak{sl}(V)\oplus\mathfrak{sl}(V),\\
\mcL_{(\bar i,\bar j)}&=\bigwedge^{[i]_3}V\otimes\bigwedge^{[i+j ]_3}V\otimes\bigwedge^{[i+2j ]_3}V\otimes\bigwedge^{[j ]_3}V,\qquad
\end{array}
\end{equation}
 where $V$  denotes a vector space  over $\FF$ of dimension $3$. 
 
 Note that, for any   $(\bar 0,\bar 0)\ne\alpha=(\bar i,\bar j) \in\ZZ_3^2$, there is just one index $k_\alpha\in\{1,2,3,4\}$ such that $\alpha_{k_\alpha}=0$, where   $\alpha_1=[i]_3$, $\alpha_2=[i+j]_3$, $\alpha_3=[i+2j]_3$, $\alpha_4=[j]_3$, that is, $\mcL_\alpha=\bigwedge^{\alpha_1}V\otimes\bigwedge^{\alpha_2}V\otimes\bigwedge^{\alpha_3}V\otimes\bigwedge^{\alpha_4}V$.
 
 Consider the product of $\mcL_{(\bar 0,\bar 0)}$ with $\mcL_{(\bar i,\bar j)}$ given by the action of the $k$th ideal of $\mcL_{\bar0}$ on the $k$th slot ($k=1,2,3,4$). Consider the   product of $x=x_1\otimes x_2\otimes x_3\otimes x_4\in\mcL_{\alpha}$ and  $y=y_1\otimes y_2\otimes y_3\otimes y_4\in\mcL_{\beta}$, for   $(\bar 0,\bar 0)\ne\alpha,\beta\in\ZZ_3^2$, given by
\begin{equation}\label{eq_pr4A2}
[x,y]=\left\{
\begin{array}{ll}
a_{\alpha,\beta}\,x_1*y_1\otimes x_2*y_2\otimes x_3*y_3\otimes x_4*y_4&\textrm{if }\beta\ne2\alpha, \\
\sum_{i=1,i\ne k_\alpha}^4 b_\alpha^{(i)}[x_i,y_i]_i\Pi_{j=1,j\ne i}^4(x_{j},y_{j})  &\textrm{if }\beta=2\alpha,
\end{array}\right.
\end{equation}
for some nonzero scalars $a_{\alpha,\beta}$ and $b_\alpha^{(i)}$ (for any $i\ne k_\alpha$).

 \noindent  
Then, $\mcL$ endowed with the  product $[\ ,\ ]$   is a Lie algebra if and only if 
\begin{equation}\label{caso4_0}
\begin{split} 
a_{\alpha,\beta}&=a_{\beta,\alpha },\\
-b_\alpha^{(i)}&=a_{\alpha,\alpha}a_{2\alpha,2\alpha}   ,\\
-a_{\alpha,\alpha}a_{2\alpha, \beta}&=a_{\alpha, \beta}a_{\alpha, \alpha+\beta},\\
-a_{\alpha,\beta}a_{\alpha+\beta,\alpha+\beta}&=a_{ \beta,\alpha+\beta}a_{\alpha+2\beta,\alpha },\\ 
-b_\alpha^{(i)}&=a_{\alpha,   \beta}a_{\alpha+\beta,2\alpha }   ,  
\end{split}
\end{equation}
for all $\beta\ne\alpha,2\alpha$ and for all $i\ne k_\alpha$.

Moreover, $(\mcL, [\ ,\ ])$ is a simple exceptional Lie algebra isomorphic to $\e_8$. 
A solution of this system of  equations that, in particular, gives a model of $\e_8$, is
\begin{equation}\label{solcaso4}
\begin{array}{rll}
1&=a_{\alpha,\beta},\qquad&\forall\alpha\ne\beta,\\
-1&=b_\alpha^{(i)}=a_{\alpha,\alpha}, \quad\  &\forall i\ne k_\alpha.
\end{array}
\end{equation}
\end{proposition}

Here we consider more scalars and more equations than necessary, unlike the above models based on cyclic groups. This permits us to take advantage of a greater symmetry,  and so to handle   a big number of equations simultaneously.

\begin{proof}
Recall that,  for all $u,v\in V$ and all $\mu, \eta\in  \bigwedge^2V$,
\begin{equation}\label{comoda}
\begin{array}{cc}
u*v=-v*u,\qquad  &\mu*\eta=-\eta*\mu,\vspace{2pt} \\
\quad(u,\eta)=u*\eta=\eta*u=(\eta,u),\qquad\ & [u,\eta]=-[\eta,u].
\end{array}
\end{equation}
Take also into account that $1*u=u=u*1$ and that $1*\eta=\eta=\eta*1$. Then, for any $x=\otimes_{i=1}^4 x_i\in\mcL_\alpha$ and
 $y=\otimes_{i=1}^4 y_i\in\mcL_\beta$, we have:
 \begin{itemize}
 \item if $\beta=\alpha$, then $x_i*y_i=-y_i*x_i$ for just three indices $i\in\{1,2,3,4\}$ (precisely, if and only if $i\ne k_\alpha$), so that $[x,y]=-[y,x]$;
 \item  if $\beta\ne \alpha,2\alpha$, then $x_i*y_i=-y_i*x_i$ for just one  index $i\in\{1,2,3,4\}$, so that $\otimes_{i=1}^4 x_i*y_i =- \otimes_{i=1}^4 y_i*x_i $. This implies that $[x,y]=-[y,x]$ if and only if $a_{\alpha,\beta}=a_{\beta,\alpha }$;
 \item  if $\beta=2\alpha$, then $\alpha=2\beta$. For any $i\ne k_\alpha$ ($=k_\beta$), one of the elements $x_i,y_i$ belongs   to $V$ and the other one to $\bigwedge^2V$, so that 
$ [x_i,y_i]\Pi_{j=1,j\ne i}^4(x_{j},y_{j}) =-[y_i,x_i]\Pi_{j=1,j\ne i}^4(y_{j},x_{j}) $. This gives that $[x,y]=-[y,x]$ if and only if $b_\alpha^{(i)}=b_{2\alpha}^{(i)}$ for all $i\ne k_\alpha$. 
 \end{itemize}

 Consequently, the bracket defined in $\mcL$ is skew-symmetric if and only if $a_{\alpha,\beta}=a_{\beta,\alpha }$ (for any $\beta\ne 2\alpha$) and 
 $b_\alpha^{(i)}=b_{2\alpha}^{(i)}$ for all $i\ne k_\alpha$.  Assume these conditions for the rest of the proof. Thus we have to determine 32 scalars type $a_{\alpha,\beta}$ (8 with $\alpha=\beta$ and 24 with $ \beta\notin\langle\alpha\rangle$) and 12 type $b_{\alpha}^{(i)}$.
 
 Now, the condition for $\mcL$ to be a Lie algebra reduces to check that the Jacobi identities   $J(\mcL_{\alpha},\mcL_{\beta},\mcL_{\gamma})=0$ hold for any choice of $\alpha,\beta,\gamma\in\ZZ_3^2$. We denote such identities as $(J_{\alpha,\beta,\gamma})$. 
 Again they are always satisfied whenever $(\bar0,\bar0)\in\{ \alpha,\beta,\gamma \}$, and also when $\alpha=\beta=\gamma$. Of course, the order is not important in the choice of the triple $\{\alpha,\beta,\gamma\}$. Note that we have just five different kind of identities to study:
 \begin{enumerate} 
 \item[\rm (i)] $(J_{\alpha,\alpha,2\alpha})$;
 \item[\rm (ii)] $(J_{\alpha,\alpha,\beta})$; 
 \item[\rm (iii)] $(J_{\alpha,\beta,\alpha+\beta})$;
 \item[\rm (iv)] $(J_{\alpha,2\alpha,\beta})$;
 \item[\rm (v)] $(J_{\alpha,\beta,2(\alpha+\beta)})$;
 \end{enumerate}
 for any $(\bar0,\bar0)\ne \alpha,\beta\in\ZZ_3^2$ with $\beta\ne \alpha,2\alpha$. Indeed, if there are two repeated indices in $\{\alpha,\beta,\gamma\}$, of course the situations (i) and (ii) appear, otherwise we can assume the three $\alpha$, $\beta$   and $\gamma$ different. If the sum of two of them is  $(\bar0,\bar0)$; then we have (iv)  with $\beta\ne \alpha,2\alpha$.  If this is not the case but one element is the sum of the  other two, then we have (iii), again with $\beta\ne \alpha,2\alpha$. If the situation is not any of these ones, $\gamma\in\{2\alpha+2\beta,2\alpha+\beta,\alpha+2\beta\}$, but in the two last cases the elements could also be labelled to be in the situation (iii); so that $\gamma=2\alpha+2\beta$. To summarize, we have 8 identities of type (i), 48 of type (ii), 24 of type (iii) ($\alpha$ and $\beta$ interchangeable), 24 of type (iv) and 8 of type (v) (each of the three elements is twice the sum of the other ones), in total, 112 identities. 
 Some of the resultant equations are redundant, but this does not make it more difficult to find a solution. First we focus on writing the equations related to the identities   $J(\mcL_{\alpha},\mcL_{\beta},\mcL_{\gamma})=0$ according to this distribution of cases (i) $-\dots-$(v). \smallskip
 
 \noindent
 $ \boxed{ (J_{\alpha,\alpha,2\alpha})}$
  Consider first the case $\alpha=(\bar1,\bar0)$, so that  $\mcL_{\alpha}=V\otimes V\otimes V\otimes  \FF$. Let 
 $x=x_1\otimes x_2\otimes x_3\otimes1,z=z_1\otimes z_2\otimes z_3\otimes1\in\mcL_\alpha$ and take any element
 $y=\otimes_{i=1}^3 y_i\otimes1\in\mcL_{2\alpha}$ such that $(x_1,y_1)=0$. Then 
 $$
 [[x,y],z]=b_{\alpha}^{(1)}[x_1,y_1]_1(z_1)\otimes (x_2,y_2)z_2\otimes (x_3,y_3)z_3\otimes 1.
 $$
 (Recall that, for $f\in\mathfrak{sl}(V)$, $f_1$ denotes $(f,0,0,0)\in 4\mathfrak{sl}(V)$.)
 If besides the elements have been chosen such that $(y_2,z_2)=0=(y_3,z_3)$, then $[[y,z],x]=0$.  As usual, let us fix a basis $\{e_1,e_2,e_3\}$ such that $\phi(e_{123})=1$ in order to make concrete computations with the help of Lemma~\ref{le_notacionesyproductos}. For instance, if we take
 $$
 x_1=z_2=z_3=e_1,\quad x_2=x_3=z_1=e_2,\quad y_1=y_2=y_3=e_{31},
 $$
  thus $ [[x,y],z]=b_{\alpha}^{(1)} [e_1,e_{31}]_1(e_2)\otimes (e_2,e_{31})e_1\otimes  (e_2,e_{31})e_1\otimes 1=b_{\alpha}^{(1)} e_1\otimes   e_1 \otimes  e_1  \otimes 1$ and
 $$\begin{array}{ll}
 [[z,x],y]&=a_{\alpha,\alpha}a_{2\alpha,2\alpha} (z_1*x_1)*y_1\otimes (z_2*x_2)*y_2\otimes (z_3*x_3)*y_3\otimes 1\\
 &=a_{\alpha,\alpha}a_{2\alpha,2\alpha} e_1\otimes   e_1 \otimes  e_1  \otimes 1.
 \end{array}
 $$
 That is, for our choice $x=e_1\otimes e_2\otimes e_2\otimes 1$, $z=e_2\otimes e_1\otimes e_1\otimes 1$, $y=e_{31}\otimes e_{31}\otimes  e_{31}\otimes 1$, we have
 $$
 J(x,y,z)=  \big( b_{\alpha}^{(1)}+a_{\alpha,\alpha}a_{2\alpha,2\alpha}\big)e_1\otimes   e_1 \otimes  e_1  \otimes 1.
 $$
 If $\mcL$ is a Lie algebra, this implies that  $b_{\alpha}^{(1)}=-a_{\alpha,\alpha}a_{2\alpha,2\alpha}$. 
 The same computation proves $b_{\alpha}^{(i)}=-a_{\alpha,\alpha}a_{2\alpha,2\alpha}$ for all $i=1,2,3$, 
 when we \emph{permute the indices}, i.e., for $x_i=z_j=e_1$ and $x_j=z_i=e_2$ for any $j\ne i$, $j\in\{1,2,3\}$ and the same $y$;
  in particular $b_{\alpha}^{(i)}$ does not depend on the superindex $i$.
 
 Furthermore, these conditions $b_{\alpha}^{(i)}=-a_{\alpha,\alpha}a_{2\alpha,2\alpha}$ for any $i\ne 4$ are not only necessary but sufficient to guarantee the identity $ (J_{\alpha,\alpha,2\alpha})$, due to the preliminary knowledge of the $\ZZ_3^2$-grading on $\e_8$, with similar arguments to those ones used in  the above models (see Remark~\ref{re_haysol}).

 We have checked the identity only for   $\alpha=(\bar1,\bar0)$. For the remaining values of $\alpha\in\ZZ_3^2$ ($\alpha\ne(\bar0,\bar0)$), 
   some concrete choices of elements will be particularly useful.  
 \begin{itemize}
  \item[(a)] There are $x,z\in V$, $y\in\bigwedge^2V$ such that
  $$
  [x,y](z)=(z*x)*y=x,\qquad (x,y)=0.
  $$
  In fact, $x=e_1$, $z=e_2$, $y=e_{31}$ provides a solution.  (Always in mind Lemma~\ref{le_notacionesyproductos}.)
   \item[(b)] There are $x,z\in V$, $y\in\bigwedge^2V$ such that
  $$
  (x,y)z=-(z*x)*y=z,\qquad (z,y)=0.
  $$
  In fact, $x=e_2$, $z=e_1$, $y=e_{31}$ provides a solution.  
    \item[(c)] There are $x,z\in \bigwedge^2V$, $y\in V$ such that
  $$
  [x,y](z)=(z*x)*y=x,\qquad (x,y)=0.
  $$
  In fact, $x=e_{12}$, $z=e_{23}$, $y=e_1$ gives a solution. 
  
       \item[(d)] There are $x,z\in \bigwedge^2V$, $y\in V$ such that
  $$
   (x,y)z=-(z*x)*y=z,\qquad (z,y)=0.
  $$
  In fact, $x=e_{23}$, $z=e_{12}$, $y=e_1$ is a solution. 
   \end{itemize}

With these elements, it is not difficult to prove  that necessarily $b_{\alpha}^{(i)}+a_{\alpha,\alpha}a_{2\alpha,2\alpha}=0$ for any $i\ne k_\alpha$. 
 In fact, let $j$ and $k$ be   the two indices in $\{1,2,3,4\}$ different from $i$ and $k_\alpha$. Take $x_{k_\alpha}=y_{k_\alpha}=z_{k_\alpha}=1$. 
Take $\{x_i,y_i,z_i\}$ as in item (a) if $\alpha_i=1$ and as in item (c) if   $\alpha_i=2$.
Take $\{x_j,y_j,z_j\}$  as in item (b) if $\alpha_j=1$ and as in item (d) if   $\alpha_j=2$.
  Similarly, take $\{x_k,y_k,z_k\}$  as in item (b) if $\alpha_k=1$ and as in item (d) if   $\alpha_k=2$. 
Now $[[y,z],x]=0$, and $\otimes_{t=1}^4 (z_t*x_t)*y_t$ coincides with $1  \otimes x_i  \otimes (-z_j)  \otimes (-z_k)$ but suitably ordered,  which in particular is nonzero and coincides with $1  \otimes [x_i,y_i]_i(z_i)  \otimes (x_j,y_j)z_j  \otimes (x_k,y_k)z_k$, also suitably ordered. Hence   
 $$
 J(x,y,z)=\big (b_{\alpha}^{(i)}  +a_{\alpha,\alpha}a_{2\alpha,2\alpha}\big)\big(\otimes_{t=1}^4 (z_t*x_t)*y_t \big),
 $$
 getting  the required equations.
 Of course, the equations $-a_{\alpha,\alpha}a_{2\alpha,2\alpha}=b_{\alpha}^{(i)}=b_{\alpha}^{(j)}=b_{\alpha}^{(k)}$ are sufficient to get that $J(x,y,z)=0$ for all $x,z\in\mcL_{\alpha}$, $y\in\mcL_{2\alpha}$.
 Moreover, these equations evidently imply $b_{\alpha}^{(i)}=b_{2\alpha}^{(i)}$.
 \medskip
 
 \noindent
 $ \boxed{ (J_{\alpha,\alpha,\beta})}$ with $\beta\ne \alpha,2\alpha$. Take for instance $\alpha=(\bar1,\bar0)$ and $\beta=(\bar0,\bar1)$.
 So we can assume that the elements $x,y\in\mcL_\alpha $ and $z\in  \mcL_{\beta}$
 are $x=u_1\otimes u_2\otimes u_3\otimes 1$, $y=v_1\otimes v_2\otimes v_3\otimes 1$ and $z=1\otimes w_2\otimes \eta_3\otimes w_4$, for some $u_i,v_i,w_i\in V$ and $\eta_i\in \bigwedge^2V$.
 Trivially we have
 $$
 \begin{array}{l}
 \,[[x,y],z]=a_{\alpha,\alpha}a_{2\alpha, \beta} \,u_1*v_1\otimes (u_2*v_2)*w_2\otimes  (u_3*v_3)*\eta_3\otimes w_4,\\
  \,[[y,z],x]=a_{\alpha, \beta}a_{\alpha, \alpha+\beta}\, v_1*u_1\otimes (v_2*w_2)*u_2\otimes  (v_3*\eta_3 )*u_3\otimes w_4,\\
   \,[[z,x],y]=a_{\alpha, \beta}a_{\alpha, \alpha+\beta}\, u_1*v_1\otimes (w_2*u_2)*v_2\otimes  (\eta_3*u_3)*v_3\otimes w_4.
 \end{array}
 $$
 First take into account   that $ (u*v)*w=\phi(u\wedge v\wedge w)$, so that 
\begin{equation*}\label{eq_asociativo3}
 (u*v)*w= (v*w)*u=(w*u)*v, 
\end{equation*} 
 and
 \begin{multline*}
 J(x,y,z)=u_1*v_1\otimes (u_2*v_2)*w_2\otimes  \big(a_{\alpha,\alpha}a_{2\alpha, \beta}  (u_3*v_3)*\eta_3+\\
 +a_{\alpha, \beta}a_{\alpha, \alpha+\beta}  (-(v_3*\eta_3 )*u_3+(\eta_3*u_3)*v_3) \big) \otimes w_4.
\end{multline*}
This is identically zero if and only if
\begin{equation}\label{re_tipo2}
a_{\alpha,\alpha}a_{2\alpha, \beta}  (u_3*v_3)*\eta_3
 +a_{\alpha, \beta}a_{\alpha, \alpha+\beta}  \big(-(v_3*\eta_3 )*u_3+(\eta_3*u_3)*v_3 \big) =0.
\end{equation}
 Second, observe $(u\wedge v)\,\llcorner\,\tilde \eta=\tilde \eta(u)v-\tilde \eta(v)u $ for any $u,v\in V$ and $\eta\in\bigwedge^2V$, according to Lemma~\ref{le_notacionesyproductos} or to \cite[Exercise B.15]{FultonHarris}. So,
\begin{equation}\label{formula}
 (u*v)*\eta=-(v*\eta)*u+( \eta*u)*v,
 \end{equation}
 and Eq.~\eqref{re_tipo2} becomes
 $\big(a_{\alpha,\alpha}a_{2\alpha, \beta}+a_{\alpha, \beta}a_{\alpha, \alpha+\beta} \big)  (u_3*v_3)*\eta_3=0$. This always holds if and only if 
 $$
 a_{\alpha,\alpha}a_{2\alpha, \beta}+a_{\alpha, \beta}a_{\alpha, \alpha+\beta}=0.
 $$
 
In general, for any  choice of $\alpha,\beta\in\mathbb Z_3^2$, $\beta\ne\pm\alpha$, there is  just one index $i\in\{1,2,3,4\}$ such that $\alpha_i=\beta_i\in\{ 1, 2\}$, and there is just one index $j\ne k_\alpha,k_\beta$
 such that $\alpha_j\ne \beta_j$ (both different from $ 0$, so $\{\alpha_j, \beta_j\}=\{1,2\}$). If $\alpha_i= 1$, the situation is completely analogous to the above, replacing $j$ with the index $3$ and $i$ with the index $2$. In order to deal with the case $\alpha_i= 2$, we have to take into consideration  
 \begin{align}
 (\mu*\nu)*\eta= (\nu*\eta)*\mu=(\eta*\mu)*\nu, \label{eq_asociativo2}\\
 (\mu*\eta)*w=-(\eta*w)*\mu+(w*\mu)*\eta, \label{formula2}
 \end{align}
 for any $w\in V$ and $\mu,\nu,\eta\in \bigwedge^2V$. Equation \eqref{eq_asociativo2} follows from  $(e_{ij}*e_{jk})*e_{ki}=1$, and from $(e_{ji}*e_{jk})*e_{jl}=0$, for any $\{i,j,k\}$ a permutation of $\{1,2,3\}$ and $l\ne j$, again by Lemma~\ref{le_notacionesyproductos}. 
 Equation \eqref{formula2} is trivial if $\mu=\eta$, since both sides vanish.  Otherwise we can assume $\mu=e_{ij}$ and $\eta=e_{ik}$ for 
  $\sigma=\{i,j,k\}$ a permutation of $\{1,2,3\}$. Then $\mu*\eta=\sg(\sigma) e_i$. If $w=e_i$, both sides vanish; if $w=e_j$, both sides equal $\sg( \sigma) e_{ij}$; and if $w=e_k$, both sides are equal to $\sg(\sigma)e_{ik}.$
  \medskip
 
 \noindent
 $ \boxed{ (J_{\alpha,\beta,\alpha+\beta})}$
 Take for instance $\alpha=(\bar1,\bar0)$ and $\beta=(\bar0,\bar1)$.
 So we can assume that the elements $x\in\mcL_\alpha $, $y\in  \mcL_{\beta}$, $z\in \mcL_{(\bar1,\bar1)}$
 are $x=u_1\otimes u_2\otimes u_3\otimes 1$,   $y=1\otimes v_2\otimes \eta_3\otimes v_4$, and 
 $z=w_1\otimes \eta_2\otimes 1\otimes w_4$,
 for some $u_i,v_i,w_i\in V$ and $\eta_i\in \bigwedge^2V$.
 Trivially we have
 $$
 \begin{array}{l}
 \,[[x,y],z]=a_{\alpha,\beta}a_{\alpha+\beta, \alpha+\beta} \,u_1*w_1\otimes (u_2*v_2)*\eta_2\otimes  u_3*\eta_3\otimes v_4*w_4,\\
  \,[[y,z],x]=a_{\beta,\alpha+ \beta}a_{ \alpha+2\beta,\alpha}\, w_1*u_1\otimes (v_2*\eta_2)*u_2\otimes   \eta_3 *u_3\otimes v_4*w_4,\\
   \,[[z,x],y]=a_{\alpha, \alpha+\beta}a_{ \beta+2\alpha,\beta}\, w_1*u_1\otimes (\eta_2*u_2)*v_2\otimes  u_3*\eta_3\otimes w_4*v_4,
 \end{array}
$$
which gives, by \eqref{comoda},
\begin{multline*}
 J(x,y,z)= u_1*w_1\otimes \Big (a_{\alpha,\beta}a_{\alpha+\beta, \alpha+\beta}(u_2*v_2)*\eta_2\\
 -a_{\beta,\alpha+ \beta}a_{ \alpha+2\beta,\alpha}(v_2*\eta_2)*u_2
 +a_{\alpha, \alpha+\beta}a_{ \beta+2\alpha,\beta}(\eta_2*u_2)*v_2
 \Big)\otimes  u_3*\eta_3\otimes v_4*w_4.
\end{multline*}
This is identically zero if and only if, 
$$
-\big(a_{\alpha,\beta}a_{\alpha+\beta, \alpha+\beta}+a_{\beta,\alpha+ \beta}a_{ \alpha+2\beta,\alpha} \big)(v_2*\eta_2)*u_2+
\big(a_{\alpha,\beta}a_{\alpha+\beta, \alpha+\beta}+a_{\alpha, \alpha+\beta}a_{ \beta+2\alpha,\beta}\big)(\eta_2*u_2)*v_2=0,
$$
by \eqref{formula}. And of course this holds just when
$$
-a_{\alpha,\beta}a_{\alpha+\beta, \alpha+\beta}=a_{\beta,\alpha+ \beta}a_{ \alpha+2\beta,\alpha} =a_{\alpha, \alpha+\beta}a_{ \beta+2\alpha,\beta}.
$$
 The second identity does not give any extra information, if we swap the roles of $\alpha$ and $\beta$.
 
 What happens for another choice of $\alpha$ and $\beta$? As mentioned above,
  there is  just one index $i\in\{1,2,3,4\}$ such that $\alpha_i=\beta_i\in\{ 1, 2\}$, and there is just one index $j\ne k_\alpha,k_\beta$
 such that $\alpha_j\ne \beta_j$ (both different from $ 0$). Now $  k_\alpha $ and $k_\beta$ play the role of the indices 1 and 4 (it does not mind if its value is either $ 1$ or $ 2$, since in both cases $*$ is skew-symmetric (see Eq.~\eqref{comoda}). The index $j$ plays the role of the index $3$ in the above example, so that we use again $u*\eta=\eta*u$. And the index $i$ plays the role of the index $2$ in the   example. If $\alpha_i= 1$, the above computations work without any change; while if 
 $\alpha_i= 2$, the identity we need is not more
   \eqref{formula}  but  \eqref{formula2}.
  \smallskip

 \noindent
 $ \boxed{ (J_{\alpha,2\alpha,\beta})}$
 Again we begin with  $\alpha=(\bar1,\bar0)$ and $\beta=(\bar0,\bar1)$, to have an example to figure out how to deal with the remaining 23 cases.
 Take  $x=u_1\otimes u_2\otimes u_3\otimes 1\in\mcL_\alpha$,   $y=1\otimes v_2\otimes \eta_3\otimes v_4\in  \mcL_{\beta}$, and 
 $z=\mu_1\otimes \mu_2\otimes \mu_3\otimes 1\in\mcL_{2\alpha}$,
 for some $u_i,v_i \in V$ and $\mu_i,\eta_i\in \bigwedge^2V$.
 Recalling, from the above cases, that the scalars $b_{\alpha}^{(i)}$ do not depend on $i$, we write
 $$
 \begin{array}{ll}
 \,[[x,y],z]=&a_{\alpha,\beta}a_{\alpha+\beta, 2\alpha } \,u_1*\mu_1\otimes (u_2*v_2)*\mu_2\otimes  (u_3*\eta_3)*\mu_3\otimes v_4,\\
  \,[[y,z],x]=&a_{\beta,2\alpha }a_{ 2\alpha+\beta,\alpha}\, \mu_1*u_1\otimes (v_2*\mu_2)*u_2\otimes   (\eta_3 *\mu_3)*u_3\otimes v_4,\\
   \,[[z,x],y]=&b_{\alpha}^{(1)} \, \big(  (\mu_1,u_1)1\otimes [\mu_2,u_2]_2(v_2)\otimes  (\mu_3 , u_3)\eta_3\otimes  v_4 +\\
   &\qquad\qquad    + (\mu_1,u_1)1\otimes (\mu_2,u_2)v_2\otimes  [\mu_3 , u_3]_3\cdot\eta_3\otimes  v_4 \big).
 \end{array}
$$
Since $u_1*\mu_1=(u_1,\mu_1)=\mu_1*u_1$,   we focus on the elements in positions $2$ and $3$ of the tensor product. If we choose (as in item (b)),
$$
(u_2,v_2,\mu_2)=(e_2,e_1,e_{31})\Rightarrow (u_2*v_2)*\mu_2=e_1,\ (v_2*\mu_2)*u_2=0,\  (\mu_2,u_2)v_2=e_1,
$$
 and (as in item (d)),
 $$
 (\eta_3,\mu_3,u_3)=(e_{23},e_{12},e_1)\Rightarrow (u_3*\eta_3)*\mu_3=e_{12},\ (\mu_3 , u_3)=0,\ [\mu_3 , u_3]_3\cdot\eta_3=e_{12},
 $$
 then
 $$
 J(x,y,z)=\big(a_{\alpha,\beta}a_{\alpha+\beta, 2\alpha }+b_{\alpha}^{(1)}  \big)\,u_1*\mu_1\otimes e_1\otimes  e_{12}\otimes v_4,
 $$
 and the Jacobi identity implies 
 \begin{equation}\label{la4}
 a_{\alpha,\beta}a_{\alpha+\beta, 2\alpha }+b_{\alpha}^{(1)}  =0.
 \end{equation}
  The identity 
 $a_{\beta,2\alpha }a_{ 2\alpha+\beta,\alpha}+b_{\alpha}^{(1)}  =0$ is achieved as
 a direct consequence of \eqref{la4}
 by replacing $\alpha$ with $2\alpha$ (recall that $a_{\alpha,\beta}=a_{\beta,\alpha}$ and $b_{\alpha}^{(1)}=b_{2\alpha}^{(1)}$). Of course, the two equations are sufficient to get $ (J_{\alpha,2\alpha,\beta})$. For other values of $\alpha$ and $\beta$, we proceed as in the above cases, exploiting the symmetry.
 
  \medskip
 \noindent
 $ \boxed{ (J_{\alpha,\beta,2(\alpha+\beta)})}$
 Until now we have proved that all the conditions in Eq.~\eqref{caso4_0} are necessary. Let us check that they are sufficient since this new identity holds too.
  For any $x\in\mcL_\alpha$, $y\in \mcL_{\beta}$, and $z\in\mcL_{2(\alpha+\beta)}$, the three expressions
$$
 \frac{[[x,y],z]}{a_{\alpha,\beta}b_{\alpha+\beta  }^{(1)} }\,,\quad
  \frac{ [[y,z],x]}{a_{\beta,2\alpha+2\beta}b_{2\alpha   }^{(1)}  }\,,\quad
  \frac{ [[z,x],y]}{a_{\alpha,2\alpha+2\beta}b_{2\beta  }^{(1)}}\,,
$$
 do not depend on the scalars.
 Observe that
 $$
a_{\alpha,  2(\alpha+ \beta)}b_{ \beta }^{(i)} \,\stackrel{\eqref{la4}}=\, -a_{\alpha,  2(\alpha+ \beta)}a_{\beta,\alpha}a_{\beta+\alpha,2\beta}
=-a_{\alpha,\beta}a_{\beta+\alpha,2\beta}a_{\alpha,  2(\alpha+ \beta)}
 \,\stackrel{\eqref{la4}}=\, a_{\alpha,   \beta}b_{\alpha+\beta}^{(i)}.
 $$
 Changing the role of $\alpha$ and $\beta$, and recalling  $b_{ \beta }^{(i)}=b_{ 2\beta }^{(i)}$  , we also have 
 $$
  {a_{\alpha,\beta}b_{\alpha+\beta  }^{(1)} }=
   {a_{\beta,2\alpha+2\beta}b_{2\alpha   }^{(1)}  }=
  {a_{\alpha,2\alpha+2\beta}b_{2\beta  }^{(1)}},
$$
and so $\frac{J(x,y,z)}{ {a_{\alpha,\beta}b_{\alpha+\beta  }^{(1)} }}$ has to annihilate and hence $\mcL$ is a Lie algebra. 
 \smallskip
 
 It is besides evident that the set of scalars  in  Eq.~\eqref{solcaso4} provides a solution. 
 
  The fact that the obtained Lie algebra is isomorphic to $\e_8$ is a direct consequence of its simplicity. 
  And this simplicity  comes from the fact that all the chosen scalars are nonzero. 
  We can provide a proof by adapting the arguments in Remark~\ref{re_simplicidad}.
  Suppose there were   an ideal $I$ of $\mcL$ without any homogeneous element.
  Denote by $\pi_\alpha\colon\mcL\to\mcL_{\alpha}$ the projection given by the decomposition \eqref{ecumodel_Z32}.
  First, for any $x  \in I$,   we find $x' \in I$ with $\pi_{(\bar0,\bar0)}(x')=0$ and such that  $\pi_{\alpha}(x')\ne0$ if and only if $\pi_{\alpha}(x)\ne0$ for any $\alpha$ different from the neutral element.  
 And second, for any $0\ne   x\in I$ with $\pi_{(\bar0,\bar0)}(x)=0$,  we find $x' \in I$ with $\pi_{(\bar0,\bar0)}(x')\ne0$  such that the number of $\alpha$'s (including $(\bar0,\bar0) $) with $\pi_{\alpha}(x')\ne0$ is at most  the number of $\alpha$'s  with $\pi_{\alpha}(x)\ne0$.
  The first step is achieved by multiplying alternatively with two convenient copies of $\mathfrak{sl}(V)$, while for the second step we only have    to multiply with one homogeneous component. This leads to contradiction. 
  Next, assume we have an  ideal $I$ with   $0\ne  \mcL_\alpha\cap I$ for some $\alpha\ne(\bar0,\bar0)$. As $\mcL_\alpha$ is an irreducible 
  $\mcL_{(\bar0,\bar0)}$-module, then $\mcL_\alpha\subset I$ 
  and we   deduce $I=\mcL$.  It is easy to get the same conclusion if $0\ne  \mcL_{(\bar0,\bar0)}\cap I$.
 \end{proof}

 
  \subsection{Model based on the subalgebra of type $2A_1\oplus 2A_3$}\label{se_2A12A3}
  
  Our last model will be based on the simultaneous diagonalization relative to   two commuting automorphisms $F$ and $G$, with $F$ of order two (relative to the 8th node) and $G$ of order four    (relative to the 6th node), as in Section~\ref{se_gradings}. The $\ZZ_2$-grading induced by $F$ is (up to isomorphism) 
  $$
    \mcL_{\bar0}=\slf(U)\oplus\e_7 , \qquad
   \mcL_{\bar1}=U\otimes V(\varpi_7);
  $$
  where $U$ is a 2-dimensional vector space and $ V(\varpi_7)$ is the $\e_7$-irreducible module of dimension 56. 
  Now apply Section~\ref{se_gradings} to $\mathcal{M}=\e_7$ to describe its $\ZZ_4$-grading as
  $$
  \begin{array}{ll}
  \mathcal{M}_{\bar0}=\slf(U)\oplus\slf(W)\oplus\slf(W),\qquad\quad&
  \mathcal{M}_{\bar2}=\FF\otimes\bigwedge^2W\otimes\bigwedge^2W,\\
  \mathcal{M}_{\bar1}=U\otimes W\otimes W,
  &\mathcal{M}_{\bar3}=U\otimes\bigwedge^3W\otimes\bigwedge^3W,
  \end{array}
  $$
  for $W$ a 4-dimensional vector space. (Here only the dimensions are not enough to distinguish  at a first glance whether the module $\mathcal{M}_{\bar1}$ -which determines the others- is either the above or  $U\otimes W\otimes \bigwedge^3W$. Note that this is not a problem: $W$ and its dual module $\bigwedge^3W$ are not isomorphic but there is an -outer- automorphism of $\slf(W)$ interchanging them. This means that we can recover the Lie algebra $\mathcal M$ in both ways.)
 According to \cite{modsgrad}, the module $\mathcal{V}= V(\varpi_7)$ is compatible with the $\ZZ_4$-grading on $\e_7$, 
 that is, $\mathcal{M}_{\bar i}\cdot \mathcal{V}_{\bar j}\subset \mathcal{V}_{\bar i+\bar j}$, for
 $$
   \begin{array}{ll}
  \mathcal{V}_{\bar0}=U\otimes \bigwedge^2W \otimes \FF,\qquad\quad&
  \mathcal{V}_{\bar1}=\FF\otimes\bigwedge^3W\otimes W,\\
  \mathcal{V}_{\bar2}=U\otimes \FF\otimes \bigwedge^2W,
  &\mathcal{V}_{\bar3}=\FF\otimes W\otimes\bigwedge^3W.
  \end{array}
  $$
  Again, shifts are possible (but, once $\mathcal{V}_{\bar0}$ is fixed, all is determined). 
  Gathering the information, we have the descriptions  (descriptions as modules for the neutral component) of all the homogeneous components of a $\ZZ_2\times\ZZ_4$-grading on the split algebra $\e_8$ with neutral homogeneous component isomorphic to $2\slf(U)\oplus2\slf(W)$. This is the starting point for the model given in the next proposition.  
  
 \begin{proposition} \label{teo_Z24model}
 
 Take $U$ and $W$     vector spaces  over $\FF$ of dimensions $2$ and $4$ respectively, and let $\mcL$ be the $\ZZ_2\times\ZZ_4$-graded vector space given, for any $i=0,1$ and $j=0,1,2,3$,  by \vspace{-5pt}
\begin{equation}\label{ecumodel_Z24}
\begin{array}{ll}
\mcL_{(\bar 0,\bar 0)}&=\mathfrak{sl}(U)\oplus\mathfrak{sl}(U)\oplus\mathfrak{sl}(W)\oplus\mathfrak{sl}(W),\\
\mcL_{(\bar i,\bar j)}&=\bigwedge^{[i]_2}U\otimes\bigwedge^{[i+j ]_2}U\otimes\bigwedge^{[2i+j ]_4}W\otimes\bigwedge^{[j ]_4}W.
\end{array}
\end{equation}
  That is,  for     $(\bar 0,\bar 0)\ne\alpha=(\bar i,\bar j)  $,  we write $\mcL_\alpha=\bigwedge^{\alpha_1}U\otimes\bigwedge^{\alpha_2}U\otimes\bigwedge^{\alpha_3}W\otimes\bigwedge^{\alpha_4}W$ for $\alpha_1=[i]_2$, $\alpha_2=[i+j]_2$, $\alpha_3=[2i+j]_4$, $\alpha_4=[j]_4$. Denote by $I_\alpha=\{k\in\{1,2,3,4\}:\alpha_k=0\}$.

Take the product of $\mcL_{(\bar 0,\bar 0)}$ with $\mcL_{(\bar i,\bar j)}$ given by the action of the $k$th ideal of $\mcL_{\bar0}$ on the $k$th slot ($k=1,2,3,4$). Order $(\ZZ_2 \times \ZZ_4,\prec)$ lexicographically, and consider, for   $(\bar 0,\bar 0)\ne\alpha,\beta\in\ZZ_2 \times \ZZ_4$, $\alpha\prec\beta$, the   product of $x=x_1\otimes x_2\otimes x_3\otimes x_4\in\mcL_{\alpha}$ and  $y=y_1\otimes y_2\otimes y_3\otimes y_4\in\mcL_{\beta}$ given by \vspace{-9pt}

\begin{equation}\label{eq_Z4Z2}
[x,y]=\left\{
\begin{array}{ll}
a_{\alpha,\beta}\,x_1*y_1\otimes x_2*y_2\otimes x_3*y_3\otimes x_4*y_4&\textrm{if }\beta+\alpha\ne (\bar 0,\bar 0), \\
\sum_{i=1,i\notin I_\alpha}^4 b_\alpha^{(i)}[x_i,y_i]_i\Pi_{j=1,j\ne i}^4(x_{j},y_{j})  &\textrm{if }\beta+\alpha= (\bar 0,\bar 0), 
\end{array}\right.
\end{equation}
for some nonzero scalars $a_{\alpha,\beta}$ and $b_\alpha^{(i)}$. 
Extend the bracket to $\mcL$ making that the product of $\mcL_{\alpha}$ with $\mcL_{\beta}$  is skew-symmetric if $\beta\prec\alpha$, and making $(\mcL_{(\bar0,\bar 0)},[\ ,\ ])$ a subalgebra.
Then, $\mcL$ endowed with the  product $[\ ,\ ]$   is a  Lie algebra if and only if 

\begin{equation}\label{paciencia} 
\begin{array}{l}
b_{1}^{(2)}=b_{1}^{(3)}=b_{1}^{(4)}=-a_{11}a_{23}=  a'_{10}a'_{31},\\
 b_{2}^{(3)}=b_{2}^{(4)}=-a_{12}a_{23},\\
   a_{11}a'_{20}=a'_{10}a'_{11}, \\
  a_{11}a'_{21}=-a'_{11}a'_{12}, \\
   a_{11}a'_{22}=a'_{12}a'_{13}, \\
   a_{11}a'_{23}=-a'_{10}a'_{13}, \\
    a_{12}a_{33}=a_{11}a_{23},\\
    a_{12}a'_{30}=a'_{12}a'_{20}, \\
    a_{12}a'_{31}=-a'_{13}a'_{21}, \\
  \end{array}\qquad
 \begin{array}{l}
 a_{12}a'_{32}=a'_{10}a'_{22}, \\
            a_{12}a'_{33}=-a'_{11}a'_{23}, \\
             a'_{10}a''_{11}= a'_{11}a''_{02}=-a_{11}a''_{01} ,\\
              a'_{10}a''_{12}=-a_{12}a''_{02}=-a'_{12}a''_{03}, \\
 b_{0}^{'(1)}=b_{0}^{'(2)}=b_{0}^{'(3)}=a'_{10}a''_{01}, \\
 b_{1}^{'(1)}=b_{1}^{'(3)}=b_{1}^{'(4)}=a'_{31}a''_{03}=a'_{11}a''_{23},\\
 b_{2}^{'(1)}=  b_{2}^{'(2)}= b_{2}^{'(4)}=-a''_{12}a'_{32},\\
        a'_{12}a''_{33}=-a''_{23}a_{11}.  
                \end{array}
\end{equation}
To abbreviate the notation a bit, we have used 
$$a_{ij}:=a_{(\bar0,\bar i),(\bar0.\bar j)},\quad a'_{ij}:=a_{(\bar 0,\bar i),(\bar1.\bar j)},\quad a''_{ij}:=a_{(\bar1,\bar i),(\bar1.\bar j)},
\quad 
b_i^{(s)}:=b_{(\bar0,\bar i)}^{(s)},\quad b_i^{'(s)}:=b_{(\bar1,\bar i)}^{(s)}.$$\vspace{-10pt}

Moreover, $(\mcL, [\ ,\ ])$ is a simple exceptional Lie algebra isomorphic to $\e_8$. 
A solution of this system of  equations that, in particular, gives a model of $\e_8$,  is, for instance,
\begin{equation}\label{dif_unasol}
\begin{array}{c}
1=a_{11}=a_{12}=a'_{10}=a'_{11}=a'_{12}=a'_{13}=a'_{20}=a'_{22 }=a'_{30}\,,\vspace{2pt}\\
1=a'_{31}=a'_{32}=a'_{33}=a''_{01 }=a''_{12}=a''_{33}=b_1^{(i)}= b_2^{(i)}=b_0^{'(i)} \, ,\vspace{1pt}\\
-1=a_{23}=a'_{21}=a'_{23}=a_{33}=a''_{02}=a''_{03}=a''_{11}=a''_{23}=b_1^{'(i)}= b_2^{'(i)}.
\end{array}
\end{equation}
 \end{proposition}

\begin{proof}
Choose
$\{e_1,e_2\}$ a basis of $U$ such that $\phi(e_{12})=1$ and $\{e_1,\dots,e_4\}$ a basis of $W$ such that $\phi(e_{1234})=1$. (The slight abuse of notation does not produce confusion.) Denote by $\textrm{pr}_i\colon\mcL_{(\bar 0,\bar 0)}\to\slf(U)$, $i=1,2$, and by $\textrm{pr}_i\colon\mcL_{(\bar 0,\bar 0)}\to\slf(W)$, $i=3,4$, the projections onto the four simple ideals of the neutral component, and use 
$\textrm{pr}_i(f)\equiv f_i$ if convenient.

First we will check that all the conditions in \eqref{paciencia} are necessary by imposing the Jacobi identity to concrete elements. 
\smallskip

$\star$ Take $x=1 \otimes e_1 \otimes e_1 \otimes e_1$, $\tilde x=1 \otimes e_1 \otimes e_1 \otimes e_2$,
$y=1 \otimes e_2 \otimes e_2 \otimes e_2$. We compute,
$$
\begin{array}{lll}
z=1 \otimes e_1 \otimes e_{134} \otimes e_{134}&\Rightarrow&  J(x,y,z) = -\big(a_{11}a_{23}+(\frac{1}{2}b_{1}^{(2)}+\frac{1}{4}b_{1}^{(3)}+\frac{1}{4}b_{1}^{(4)})\big)\, x,\\
 &&J(\tilde x,y,z)=\big(-(\frac{1}{2}b_{1}^{(2)}+\frac{1}{4}b_{1}^{(3)}+\frac{1}{4}b_{1}^{(4)})+b_{1}^{(4)}\big)  \,\tilde x,\\
z=1 \otimes e_2 \otimes e_{134} \otimes e_{134}&\Rightarrow& J(x,y,z)=(-a_{11}a_{23}-b_{1}^{(2)}) 1 \otimes e_2 \otimes e_1 \otimes e_1;
\end{array}
$$
so we have \vspace{-1pt}
\begin{align}
   -a_{11}a_{23}=b_{1}^{(2)}=b_1^{(3)}=b_1^{(4)}=:b_1.  \label{caso2.7_6}
\end{align}
Moreover,
 $$
 J(x,1 \otimes e_2 \otimes e_{123} \otimes e_{234},e_1 \otimes e_2 \otimes e_{34} \otimes 1 )=(-b_{1}+a'_{10}a'_{31}) e_1 \otimes e_2 \otimes e_{13} \otimes 1, $$
 which gives  \vspace{-1pt}
  \begin{equation}\label{caso2.7_17}
    b_{1}=a'_{10}a'_{31}.
    \end{equation}
    
We continue with a few more calculations,
$$
\begin{array}{ll}
z=1 \otimes 1 \otimes e_{13} \otimes e_{34}&\Rightarrow\  \textrm{pr}_3(J(x,y,z ))=(-a_{11}b_{ 2}^{(3)}+a_{12}b_{1} ) e^4_1  ,\\
z=1 \otimes 1 \otimes e_{34} \otimes e_{13}&\Rightarrow\  \textrm{pr}_4(J(x,y,z )) =(-a_{11}b_{ 2}^{(4)}+a_{12}b_{1} ) e^4_1. 
 \end{array}
$$
This implies $a_{11}b_{ 2}^{(3)}=a_{12}b_{1} =a_{11}b_{ 2}^{(4)}$, so that we get
\begin{align}
  b_{ 2}^{(3)}=b_{ 2}^{(4)}:=b_2, \qquad b_{ 2}=-a_{12} a_{23}  , 
  \label{caso2.7_1} 
 \end{align}
 from $a_{11}b_{ 2} =a_{12}b_{1}= -a_{12}a_{11}a_{23} $.
 Carrying on with the same choice of $x$ and $y$,
 $$
 \begin{array}{ll}
 z=e_1 \otimes e_1 \otimes e_{34} \otimes 1&\Rightarrow\ J(x,y,z)=(-a_{11}a'_{20}+a'_{10}a'_{11})\, e_1 \otimes e_1 \otimes 1 \otimes e_{12},\\
z= e_1 \otimes 1 \otimes e_{134} \otimes e_4&\Rightarrow\ J(x,y,z)=(a_{11}a'_{21}+a'_{11}a'_{12})\, e_1 \otimes 1 \otimes e_1 \otimes e_{124},\\
 z=e_1 \otimes e_1 \otimes 1 \otimes e_{34}&\Rightarrow\  J(x,y,z)=(-a_{11}a'_{22}+a'_{12}a'_{13})\, e_1 \otimes e_1  \otimes e_{12} \otimes 1,\\
z=e_1 \otimes 1 \otimes  e_4 \otimes e_{134} &\Rightarrow\ J(x,y,z)=(a_{11}a'_{23}+a'_{13}a'_{10})\, e_1 \otimes 1 \otimes e_{124} \otimes e_1;
\end{array}
$$
so that we achieve
 \begin{equation}\label{new73}
a_{11}a'_{20}=a'_{10}a'_{11},\quad
a_{11}a'_{21}=-a'_{11}a'_{12},\quad
a_{11}a'_{22}=a'_{12}a'_{13},\quad
a_{11}a'_{23}=-a'_{13}a'_{10}.
 \end{equation}

$\star$ Next, take $x=1 \otimes e_1 \otimes e_1 \otimes e_1$ and  $y=1 \otimes 1 \otimes e_{23} \otimes e_{23}$. From
$$
\begin{array}{ll}
z=1 \otimes e_2 \otimes e_{134} \otimes e_{134}&\Rightarrow\ 
J(x,y,z)=(-a_{12}a_{33}+a_{23}a_{11}) 1 \otimes 1 \otimes e_{13} \otimes e_{13},\\
z=e_2 \otimes e_2 \otimes e_{14} \otimes 1&
\Rightarrow\  J(x,y,z)=(-a_{12}a'_{30}+a'_{20}a'_{12}) e_2 \otimes 1 \otimes e_{1} \otimes e_{123},\\
z=e_2 \otimes 1 \otimes e_{134} \otimes e_{4} &
\Rightarrow\   J(x,y,z)=(a_{12}a'_{31}+a'_{21}a'_{13}) e_2 \otimes e_1 \otimes e_{13} \otimes 1,\\
z=e_2 \otimes e_2  \otimes 1\otimes e_{14}&
\Rightarrow\  J(x,y,z)=(-a_{12}a'_{32}+a'_{22}a'_{10}) e_2 \otimes 1  \otimes e_{123}\otimes e_{1},\\
z=e_2 \otimes 1  \otimes e_{4}\otimes e_{134} &
\Rightarrow\   J(x,y,z)=(a_{12}a'_{33}+a'_{23}a'_{11}) e_2 \otimes e_1  \otimes 1\otimes e_{13};
\end{array}
$$
we follow
  \vspace{-7pt} \begin{align}\label{caso2.7_12}
   a_{12}a_{33}=a_{23}a_{11} ,\quad a_{12}a'_{30}=a'_{20}a'_{12},\quad a_{12}a'_{31}=-a'_{21}a'_{13}, \nonumber\\\quad
   a_{12}a'_{32}=a'_{22}a'_{10},\quad  a_{12}a'_{33}=-a'_{23}a'_{11} .  
\end{align}  
In particular,
\begin{equation}\label{auxilio}
\dfrac{a'_{32}}{a'_{31}}=\dfrac{a_{11}a_{12}a'_{32}}{a_{11}a_{12}a'_{31}}
=\dfrac{a_{11}a'_{22}a'_{10}}{-a_{11}a'_{21}a'_{13}}
=\dfrac{a'_{12}a'_{13}a'_{10}}{a'_{11}a'_{12}a'_{13}}=\dfrac{ a'_{10}}{a'_{11} }\  \Rightarrow  \ 
    b_{1}=a'_{10}a'_{31}=a'_{11}a'_{32}. 
\end{equation}
\vspace{1pt}

$\star$  Take $x=1 \otimes e_1 \otimes e_1 \otimes e_1$, $y= e_2 \otimes e_2 \otimes e_{23} \otimes 1$, and $ \tilde y= e_2 \otimes e_2 \otimes e_{12} \otimes 1 $. Then
$$
\begin{array}{ll}
  z=e_1 \otimes 1 \otimes e_{124} \otimes e_{2}  &\Rightarrow\    J(x,y,z)=  (a'_{10}a''_{11}+a_{11}a''_{01}) 1 \otimes 1 \otimes e_{12} \otimes e_{12} ,\\
 z=e_1 \otimes 1 \otimes e_{234} \otimes e_{2}  &\Rightarrow\    J(x,\tilde y,z)=-(a'_{11}a''_{02}+a_{11}a''_{01}) 1 \otimes 1 \otimes e_{12} \otimes e_{12},\\
 z=e_1 \otimes e_1 \otimes 1 \otimes e_{23}   &\Rightarrow\     J(x,y,z)=-(a'_{10}a''_{12}+a_{12}a''_{02}) 1 \otimes e_1 \otimes e_{123} \otimes e_{123},\\
 z=e_1 \otimes e_2 \otimes 1 \otimes e_{23}  &\Rightarrow\     J(x,y,z)=-(a'_{10}a''_{12}+a'_{12}a''_{03}) 1 \otimes e_2 \otimes e_{123} \otimes e_{123}.
\end{array}
$$
So we have
 \begin{equation}\label{caso2.7_22}
    a'_{10}a''_{11}=    a'_{11}a''_{02}=-a_{11}a''_{01},\qquad a'_{10}a''_{12}=-a_{12}a''_{02}=-a'_{12}a''_{03}.
\end{equation}
Moreover, evaluating the Jacobian operator for more choices of $z$, we get
$$
\begin{array}{ll}
 z=e_1 \otimes e_1 \otimes e_{34} \otimes 1&\Rightarrow\   J(x,y,z)=(-b_{0}^{'(3)}+a'_{10}a''_{01}) 1 \otimes e_1 \otimes e_{3} \otimes e_1, \\
 z=e_1 \otimes e_2 \otimes e_{14} \otimes 1  &\Rightarrow\    J(x,y,z)=(-b_{0}^{'(2)}+a'_{10}a''_{01}) 1 \otimes e_2 \otimes e_{1} \otimes e_1;
\end{array}
$$
thus the identities
$
 b_{0}^{'(3)}=b_{0}^{'(2)}=a'_{10}a''_{01}  
$
follow. Also,
$$
J(e_1 \otimes e_1 \otimes e_{12} \otimes 1,e_1 \otimes e_2 \otimes e_{34} \otimes 1,e_2 \otimes 1 \otimes e_{123} \otimes e_{4} )=(-b_{0}^{'(1)}+a'_{10}a''_{01}) e_1 \otimes 1 \otimes e_{123} \otimes e_{4},
$$
and gathering the information,
\begin{equation}\label{new77}
 a'_{10}a''_{01}=b_{0}^{'(3)}=b_{0}^{'(2)}=b_{0}^{'(1)}=:b'_0. 
 \end{equation}

$\star$ Now, take arbitrary elements $x=\otimes x_i\in \mcL_{(0,1)}$, $y=\otimes y_i\in \mcL_{(1,0)}$, and $z=\otimes z_i\in \mcL_{(1,3)}$, (with $x_1=y_4=z_2=1$) and compute the four projections   of
$J(x,y,z)\in \mcL_{(\bar0,\bar0)}$:
$$ 
\begin{array}{l}
\textrm{pr}_1(J(x,y,z))= -\big(b_1^{'(1)}a'_{10}+a'_{13}b'_0\big)\phi(x_2\wedge y_2)\phi(x_3\wedge y_3\wedge z_3)\phi(x_4\wedge z_4)[y_1,z_1]_1,\\
\textrm{pr}_2(J(x,y,z))=\big(a'_{13}b'_0+a''_{03}b_1\big)\phi(y_1\wedge z_1) \phi(x_3\wedge y_3\wedge z_3)\phi(x_4\wedge z_4)[x_2,y_2]_2,\\
\textrm{pr}_3(J(x,y,z))=\phi(y_1\wedge z_1)\phi(x_2\wedge y_2)\phi(x_4\wedge z_4)
\Big( (a''_{03}b_1- a'_{10}b_1^{'(3)} ) [x_3,y_3\wedge z_3]_3 \\
\qquad\qquad\qquad\qquad\qquad\qquad\qquad\qquad\qquad\quad\qquad+
( a'_{10}b_1^{'(3)}+a'_{13}b'_0 )[y_3,x_3\wedge z_3]_3\Big),\\
\textrm{pr}_4(J(x,y,z))=\big(a''_{03}b_1-a'_{10}b_1^{'(4)}\big)\phi(y_1\wedge z_1)\phi(x_2\wedge y_2)\phi(x_3\wedge y_3\wedge z_3) [x_4,z_4]_4.\\
\end{array}
$$
(Here we have used $[x_3\wedge y_3,z_3]=[x_3,y_3\wedge z_3]-[y_3,x_3\wedge z_3]$ for any $x_3,z_3\in W,y_3\in\bigwedge^2W$.)
This gives $a'_{10}b_1^{'(1)}=-a'_{13}b'_0=a''_{03}b_1= a'_{10}b_1^{'(3)}=a'_{10}b_1^{'(4)}$, and hence
\begin{equation}\label{new78}
b_1^{' }:=b_1^{'(1)}=b_1^{'(3)}=b_1^{'(4)} =a''_{03}  a'_{31} ,
\end{equation}
since $a'_{10}b_1^{' }=a''_{03}b_1\stackrel{\eqref{caso2.7_17}}= 
  a''_{03}   a'_{10}a'_{31}.
$  
Now we calculate
$$
J(1 \otimes e_1 \otimes e_1 \otimes e_1,e_2 \otimes 1 \otimes e_{234} \otimes e_2,e_1 \otimes 1 \otimes e_1 \otimes e_{234} )
=(a'_{11}a''_{23}-b'_1) 1 \otimes e_1 \otimes e_{1} \otimes e_{2},
$$
which immediately gives  
\vspace{-2pt} \begin{equation}\label{new79}
   a'_{11}a''_{23}=b'_1.
\end{equation}  

$\star$ Similarly, for any $x=\otimes x_i\in \mcL_{(0,1)}$, $y=\otimes y_i\in \mcL_{(1,1)}$, and $z=\otimes z_i\in \mcL_{(1,2)}$ with $x_1=y_2=z_3=1$, we find the four projections  of  $J(x,y,z)\in \mcL_{(\bar0,\bar0)}$:
$$
\begin{array}{l}
\textrm{pr}_1(J(x,y,z))=\big(a'_{11}b_2^{'(1)} -a'_{12}b'_1 \big)  \phi(x_2\wedge z_2)\phi(x_3\wedge y_3 )\phi(x_4\wedge y_4\wedge z_4)[y_1,z_1]_1,\\
\textrm{pr}_2(J(x,y,z))=\big(a'_{11}b_2^{'(2)} +a''_{12}b_1 \big)  \phi(y_1\wedge z_1) \phi(x_3\wedge y_3 )\phi(x_4\wedge y_4\wedge z_4)[x_2,z_2]_2,\\
\textrm{pr}_3(J(x,y,z))=\big( a''_{12}b_1+ a'_{12}b'_1\big) \phi(y_1\wedge z_1)\phi(x_2\wedge z_2)\phi(x_4\wedge y_4\wedge z_4)[x_3,y_3]_3,
\\
\textrm{pr}_4(J(x,y,z))= \phi(y_1\wedge z_1)\phi(x_2\wedge z_2)\phi(x_3\wedge y_3 ) \Big(\big( a'_{12}b'_1-a'_{11}b_2^{'(4)} \big) [x_4\wedge y_4,z_4]_4\\
\qquad\qquad\qquad\qquad\qquad\qquad\qquad\qquad\qquad\qquad\qquad+\big(a'_{12}b'_1+a''_{12}b_1  \big)  [x_4, y_4\wedge z_4]_4\Big).
\end{array}
$$
This    yields
\begin{equation}\label{casob2}
b_2^{' }:=b_2^{'(1)}=b_2^{'(2)}=b_2^{'(4)}=-a''_{12}a'_{32},          
\end{equation}
taking into account  $a'_{11}b'_2=-a''_{12}b_1\stackrel{\eqref{auxilio}}=-a''_{12}a'_{11}a'_{32}.$

$\star$ Finally,   we need one more equation: 
for $x=1 \otimes e_1 \otimes e_1 \otimes e_1$, $y=e_2 \otimes e_2 \otimes 1 \otimes e_{23} $ and $z= e_1 \otimes 1 \otimes e_{2} \otimes e_{124} $,
 $$
 J(x,y,z)=
(a'_{12}a''_{33}+a_{11}a''_{23})1\otimes 1\otimes e_{12}\otimes e_{12};
 $$
 giving
  \begin{equation}\label{ultima}
        a'_{12}a''_{33}=-a_{11}a''_{23}.
    \end{equation}
    
    At last  we have achieved all the necessary conditions in \eqref{paciencia}, which  follow immediately from  
     \eqref{caso2.7_6}, \eqref{caso2.7_17}, \eqref{caso2.7_1}, \eqref{new73}, \eqref{caso2.7_12}, \eqref{caso2.7_22}, \eqref{new77}, \eqref{new78}, \eqref{new79}, \eqref{casob2}, and  \eqref{ultima}. That is, if $\mcL$ is a Lie algebra, the scalars have to satisfy \eqref{paciencia}.  \smallskip
     
     A very important fact is that, for any arbitrary choice of $a_{11},a_{12},a'_{10},a'_{11},a'_{12},a'_{13},a''_{01}\in\FF^\times$, \eqref{paciencia} forces the remaining 21 scalars to be
    \begin{equation}\label{todas}
     \begin{array}{c}
     b_1=\dfrac{a'_{10}a'_{11}a'_{12}a'_{13}}{a_{11}a_{12}},\quad
     b_2=\dfrac{a'_{10}a'_{11}a'_{12}a'_{13}}{a_{11}^2},\quad
     a_{23}=\dfrac{-a'_{10}a'_{11}a'_{12}a'_{13}}{a_{11}^2a_{12}},\vspace{4pt}\\
     a'_{20}=\dfrac{a'_{10}a'_{11}}{a_{11}},\quad
     a'_{21}=\dfrac{ -a'_{11}a'_{12}}{a_{11}},\quad
     a'_{22}=\dfrac{ a'_{12}a'_{13}}{a_{11}},\quad
      a'_{23}=\dfrac{ -a'_{10}a'_{13}}{a_{11}},\vspace{3pt}\\
      a_{33}=\dfrac{-a'_{10}a'_{11}a'_{12}a'_{13}}{a_{11}a_{12}^2},\quad
      a'_{30}=\dfrac{a'_{10}a'_{11}a'_{12} }{a_{11}a_{12}},\quad
      a'_{31}=\dfrac{a'_{11}a'_{12}a'_{13} }{a_{11}a_{12}},\quad
      a'_{32}=\dfrac{a'_{10}a'_{12}a'_{13} }{a_{11}a_{12}},\vspace{4pt}\\
      a'_{33}=\dfrac{a'_{10}a'_{11}a'_{13} }{a_{11}a_{12}},\quad 
      b'_0=a'_{10}a''_{01},\quad
      a''_{02}=\dfrac{-a_{11}a''_{01}}{a'_{11}},\quad
      a''_{03}=\dfrac{-a_{11}a_{12}a''_{01}}{a'_{11}a'_{12}},\vspace{4pt}\\
      a''_{11}=\dfrac{-a_{11}a''_{01}}{a'_{10}}, \quad
      a''_{12}=\dfrac{a_{11}a_{12}a''_{01}}{a'_{10}a'_{11}},\quad
      b'_1=-a'_{13}a''_{01},\vspace{4pt}\\
      b'_2=\dfrac{-a'_{12}a'_{13}a''_{01}}{a'_{11} }, \quad
      a''_{23}=\dfrac{- a'_{13}a''_{01}}{a'_{11} },\quad
      a''_{33}=\dfrac{a_{11} a'_{13}a''_{01}}{a'_{11}a'_{12} };
     \end{array}
     \end{equation}
     and conversely, this provides a solution  of \eqref{paciencia}, whose set of solutions is therefore   a 7-parametric family.  Making the 7 free parameters equal to 1, and substituting in \eqref{todas}, we just obtain the concrete solution provided in \eqref{dif_unasol}.\smallskip

  \begin{table} 
\centering
    \hspace{-9pt} \begin{tabular}{|ll|}\hline
      $  { (J_{(\bar0,\bar1),(\bar0,\bar1),(\bar0,\bar2)})}$: $a_{11}b_2=a_{12}b_1 $  
     & 
     ${ (J_{(\bar0,\bar1),(\bar0,\bar1),(\bar0,\bar3)})}$:   $ a_{11}a_{23}=-b_1 $
     \\
  ${ (J_{(\bar0,\bar1),(\bar0,\bar1),(\bar1,\bar0)})}$: $a_{11}a'_{20}=a'_{10}a'_{11} $
  &
   ${ (J_{(\bar0,\bar1),(\bar0,\bar1),(\bar1,\bar1)})}$: $a_{11}a'_{21}=-a'_{11}a'_{12} $
   \\
${ (J_{(\bar0,\bar1),(\bar0,\bar1),(\bar1,\bar2)})}$: $
    a_{11}a'_{22}=a'_{12}a'_{13} $
    &
  ${ (J_{(\bar0,\bar1),(\bar0,\bar1),(\bar1,\bar3)})}$: $
    a_{11}a'_{23}=-a'_{10}a'_{13} $
    \\
 ${ (J_{(\bar0,\bar1),(\bar0,\bar2),(\bar0,\bar2)})}$:
     $a_{12} a_{23}=-b_2 $
&
 ${ (J_{(\bar0,\bar1),(\bar0,\bar2),(\bar0,\bar3)})}$: $
    a_{12}a_{33}=a_{11}a_{23}
 =-b_1 
$
\\
 ${ (J_{(\bar0,\bar1),(\bar0,\bar2),(\bar1,\bar0)})}$: $
    a_{12}a'_{30}=a'_{12}a'_{20}=-a'_{10}a'_{21} 
 $
 &
  ${ (J_{(\bar0,\bar1),(\bar0,\bar2),(\bar1,\bar1)})}$: $
    a_{12}a'_{31}=-a'_{13}a'_{21}=a'_{11}a'_{22} 
$ \\
 ${ (J_{(\bar0,\bar1),(\bar0,\bar2),(\bar1,\bar2)})}$: $
    a_{12}a'_{32}=a'_{10}a'_{22}
 =-a'_{12}a'_{23} 
$
&
 ${ (J_{(\bar0,\bar1),(\bar0,\bar2),(\bar1,\bar3)})}$: 
 $
    a_{12}a'_{33}=-a'_{11}a'_{23}=a'_{13}a'_{20} 
$
\\
 ${ (J_{(\bar0,\bar1),(\bar0,\bar3),(\bar0,\bar3)})}$: $b_1=-a_{12}a_{33}$   
 &
  ${ (J_{(\bar0,\bar1),(\bar0,\bar3),(\bar1,\bar0)})}$: 
   $
   b_1= a'_{10}a'_{31}=a'_{13}a'_{30} 
$
\\
 ${ (J_{(\bar0,\bar1),(\bar0,\bar3),(\bar1,\bar1)})}$: 
  $
    b_1=a'_{10}a'_{31}=a'_{11}a'_{32} 
$
 &
  ${ (J_{(\bar0,\bar1),(\bar0,\bar3),(\bar1,\bar2)})}$: $  b_1= a'_{11}a'_{32} =  a'_{12}a'_{33} $ 
 \\
  ${ (J_{(\bar0,\bar1),(\bar0,\bar3),(\bar1,\bar3)})}$: $b_1=a'_{33}a'_{12}=a'_{13}a'_{30}$ 
 &
  ${ (J_{(\bar0,\bar1),(\bar1,\bar0),(\bar1,\bar0)})}$:  $b'_0 =a'_{10}a''_{01} 
$ 
\\
 ${ (J_{(\bar0,\bar1),(\bar1,\bar0),(\bar1,\bar1)})}$: $a'_{10}a''_{11}=-a_{11}a''_{01}=a'_{11}a''_{02}  $
 &
  ${ (J_{(\bar0,\bar1),(\bar1,\bar0),(\bar1,\bar2)})}$: $
    -a'_{10}a''_{12}=a_{12}a''_{02}
=a'_{12}a''_{03} 
$
 \\
  ${ (J_{(\bar0,\bar1),(\bar1,\bar0),(\bar1,\bar3)})}$:  
$
     a'_{10}b'_1=-a'_{13}b'_0=a''_{03}b_1     
$
&
 ${ (J_{(\bar0,\bar1),(\bar1,\bar1),(\bar1,\bar1)})}$: $a'_{11}a''_{12}=-a_{12}a''_{11}$ 
 \\
  ${ (J_{(\bar0,\bar1),(\bar1,\bar1),(\bar1,\bar2)})}$:  $a'_{11}b'_2=a'_{12}b'_1=-a''_{12}b_1$ 
&
 ${ (J_{(\bar0,\bar1),(\bar1,\bar1),(\bar1,\bar3)})}$: $
   b'_1=a'_{11}a''_{23}=-a'_{13}a''_{01}$ 
\\
 ${ (J_{(\bar0,\bar1),(\bar1,\bar2),(\bar1,\bar2)})}$: $b'_2=a'_{12}a''_{23} $
 &
  ${ (J_{(\bar0,\bar1),(\bar1,\bar2),(\bar1,\bar3)})}$:  
$- a'_{12}a''_{33}=a''_{23}a_{11}=a'_{13}a''_{02}$ 
\\
 ${ (J_{(\bar0,\bar1),(\bar1,\bar3),(\bar1,\bar3)})}$:  $a'_{13}a''_{03}=-a''_{33}a_{12}$ 
 &
  ${ (J_{(\bar0,\bar2),(\bar0,\bar2),(\bar0,\bar3)})}$:  $a_{12}a_{23}=-b_2$ 
\\
 ${ (J_{(\bar0,\bar2),(\bar0,\bar2),(\bar1,\bar0)})}$:  
 $  a'_{22}a'_{20}=b_{2} $
&
  ${ (J_{(\bar0,\bar2),(\bar0,\bar2),(\bar1,\bar1)})}$: $a'_{21}a'_{23}=b_2 $
 \\
  ${ (J_{(\bar0,\bar2),(\bar0,\bar2),(\bar1,\bar2)})}$:  $a'_{20}a'_{22}=b_2$ 
&
 ${ (J_{(\bar0,\bar2),(\bar0,\bar2),(\bar1,\bar3)})}$: $a'_{21}a'_{23}=b_2 $
 \\
  ${ (J_{(\bar0,\bar2),(\bar0,\bar3),(\bar0,\bar3)})}$:  $b_1a_{23}=b_2a_{33}$ 
 &
  ${ (J_{(\bar0,\bar2),(\bar0,\bar3),(\bar1,\bar0)})}$:  
 $
 a'_{10}a_{23}= a'_{23} a'_{30}=   -a'_{20}a'_{32}  
$
 \\
  ${ (J_{(\bar0,\bar2),(\bar0,\bar3),(\bar1,\bar1)})}$:   
$
 a_{23}a'_{11}=  -a'_{31}a'_{20}=a'_{21}a'_{33}  
$
&
 ${ (J_{(\bar0,\bar2),(\bar0,\bar3),(\bar1,\bar2)})}$:  $
 a_{23}a'_{12}=a'_{32}a'_{21}=  -a'_{22}a'_{30} $   
\\
 ${ (J_{(\bar0,\bar2),(\bar0,\bar3),(\bar1,\bar3)})}$:  $
  a_{23}a'_{13}= -a'_{33}a'_{22}=a'_{23}a'_{31} $
&
 ${ (J_{(\bar0,\bar2),(\bar1,\bar0),(\bar1,\bar0)})}$:  $
   a'_{20}a''_{02}=-b'_0 
$
 \\
  ${ (J_{(\bar0,\bar2),(\bar1,\bar0),(\bar1,\bar1)})}$: $a'_{20}a''_{12}
  =a_{12}a''_{01}= a'_{21}a''_{03}
   $
&
 ${ (J_{(\bar0,\bar2),(\bar1,\bar0),(\bar1,\bar2)})}$: 
$
   a'_{20}b'_2=a''_{02}b_2=-a'_{22}b'_0 
$  
\\
 ${ (J_{(\bar0,\bar2),(\bar1,\bar0),(\bar1,\bar3)})}$:  $a'_{20}a''_{23}=-a_{23}a''_{03}=a'_{23}a''_{01}$ 
 &
  ${ (J_{(\bar0,\bar2),(\bar1,\bar1),(\bar1,\bar1)})}$:  $a'_{21}b'_1=-a''_{11}b_2$ 
   \\
    ${ (J_{(\bar0,\bar2),(\bar1,\bar1),(\bar1,\bar2)})}$:  $a'_{21}a''_{23}=-a_{23}a''_{12}=a'_{22}a''_{01}$ \    
   &
    ${ (J_{(\bar0,\bar2),(\bar1,\bar1),(\bar1,\bar3)})}$: $a'_{21}a''_{33}=b'_1=-a'_{23}a''_{11}$ 
\\
 ${ (J_{(\bar0,\bar2),(\bar1,\bar2),(\bar1,\bar2)})}$: $a'_{22}a''_{02}=b'_2$ 
   &
    ${ (J_{(\bar0,\bar2),(\bar1,\bar2),(\bar1,\bar3)})}$: $a'_{22}a''_{03}=a_{12}a''_{23}=a'_{23}a''_{12}$ 
        \\
         ${ (J_{(\bar0,\bar2),(\bar1,\bar3),(\bar1,\bar3)})}$: $a'_{23}b'_1=a''_{33}b_2$ 
         &
            ${ (J_{(\bar0,\bar3),(\bar0,\bar3),(\bar1,\bar0)})}$: $a'_{20}a_{33}=-a'_{30}a'_{33}$ 
          \\
           ${ (J_{(\bar0,\bar3),(\bar0,\bar3),(\bar1,\bar1)})}$: $a'_{21}a_{33}=a'_{30}a'_{31}$ 
           &
            ${ (J_{(\bar0,\bar3),(\bar0,\bar3),(\bar1,\bar2)})}$: $a'_{22}a_{33}=-a'_{31}a'_{32}$ 
            \\
             ${ (J_{(\bar0,\bar3),(\bar0,\bar3),(\bar1,\bar3)})}$: $a'_{23}a_{33}=a'_{32}a'_{33}$ 
             &
              ${ (J_{(\bar0,\bar3),(\bar1,\bar0),(\bar1,\bar0)})}$: $a'_{30}a''_{03}=-b'_0$ 
              \\
               ${ (J_{(\bar0,\bar3),(\bar1,\bar0),(\bar1,\bar1)})}$:  $-a'_{30}b'_1=a''_{01}b_1=a'_{31}b'_0$ 
               &
                ${ (J_{(\bar0,\bar3),(\bar1,\bar0),(\bar1,\bar2)})}$:  $a'_{30}a''_{23}=a_{23}a''_{02}=-a'_{32}a''_{01}$ \ 
\\
 ${ (J_{(\bar0,\bar3),(\bar1,\bar0),(\bar1,\bar3)})}$:  $a'_{30}a''_{33}=a_{33}a''_{03}=-a'_{33}a''_{02}$ 
 &
  ${ (J_{(\bar0,\bar3),(\bar1,\bar1),(\bar1,\bar1)})}$: $a'_{31}a''_{01}=a_{23}a''_{11}$ 
  \\
   ${ (J_{(\bar0,\bar3),(\bar1,\bar1),(\bar1,\bar2)})}$:  $a'_{31}a''_{02}=a_{33}a''_{12}=a'_{32}a''_{11}$ 
 &
  ${ (J_{(\bar0,\bar3),(\bar1,\bar1),(\bar1,\bar3)})}$:  $a'_{31}a''_{03}=b'_1=-a'_{33}a''_{12}$ 
 \\
  ${ (J_{(\bar0,\bar3),(\bar1,\bar2),(\bar1,\bar2)})}$:  $a'_{32}a''_{12}=-b'_2$ 
 &
  ${ (J_{(\bar0,\bar3),(\bar1,\bar2),(\bar1,\bar3)})}$: $a'_{32}b'_1=a''_{23}b_1=a'_{33}b'_2$ 
  \\
   ${ (J_{(\bar0,\bar3),(\bar1,\bar3),(\bar1,\bar3)})}$: $a'_{33}a''_{23}=a_{23}a''_{33}$ 
   &
    ${ (J_{(\bar1,\bar0),(\bar1,\bar0),(\bar1,\bar1)})}$:  $b'_0=a''_{01}a'_{10}$ 
 \\
  ${ (J_{(\bar1,\bar0),(\bar1,\bar0),(\bar1,\bar2)})}$: $b'_0=-a'_{20}a''_{02}$ 
&
 ${ (J_{(\bar1,\bar0),(\bar1,\bar0),(\bar1,\bar3)})}$: $b'_0=-a'_{30}a''_{03}$ 
\\
 ${ (J_{(\bar1,\bar0),(\bar1,\bar1),(\bar1,\bar1)})}$:  $a'_{11}a''_{01}=-a'_{20}a''_{11}$ 
  &
   ${ (J_{(\bar1,\bar0),(\bar1,\bar1),(\bar1,\bar2)})}$: $a'_{12}a''_{01}=a'_{30}a''_{12}=a'_{21}a''_{02}$ 
 \\
  ${ (J_{(\bar1,\bar0),(\bar1,\bar1),(\bar1,\bar3)})}$: $-a'_{13}a''_{01}=b'_1=a'_{31}a''_{03}$ 
&
 ${ (J_{(\bar1,\bar0),(\bar1,\bar2),(\bar1,\bar2)})}$: $a'_{22}a''_{02}=b'_2$ 
\\
 ${ (J_{(\bar1,\bar0),(\bar1,\bar2),(\bar1,\bar3)})}$: $-a'_{23}a''_{02}=a'_{10}a''_{23}=a'_{32}a''_{03}$ 
 &
  ${ (J_{(\bar1,\bar0),(\bar1,\bar3),(\bar1,\bar3)})}$: $a'_{33}a''_{03}=-a'_{20}a''_{33}$ 
  \\
   ${ (J_{(\bar1,\bar1),(\bar1,\bar1),(\bar1,\bar2)})}$: $a'_{22}a''_{11}=-a'_{31}a''_{12}$ 
 &
  ${ (J_{(\bar1,\bar1),(\bar1,\bar1),(\bar1,\bar3)})}$: $a'_{23}a''_{11}=-b'_1$ 
  \\
   ${ (J_{(\bar1,\bar1),(\bar1,\bar2),(\bar1,\bar2)})}$: $a'_{32}a''_{12}=-b'_2$
   &
    ${ (J_{(\bar1,\bar1),(\bar1,\bar2),(\bar1,\bar3)})}$: $-a'_{33}a''_{12}=a'_{11}a''_{23}=b'_1$ 
    \\
     ${ (J_{(\bar1,\bar1),(\bar1,\bar3),(\bar1,\bar3)})}$: $b'_1=a'_{21}a''_{33}$ 
   &
    ${ (J_{(\bar1,\bar2),(\bar1,\bar2),(\bar1,\bar3)})}$: $b'_2=a'_{12}a''_{23}$ 
 \\
  ${ (J_{(\bar1,\bar2),(\bar1,\bar3),(\bar1,\bar3)})}$: $a'_{13}a''_{23}=-a'_{22}a''_{33}$& \vspace{2pt}\\
  \hline
     \end{tabular}
     \vspace{3pt}
      \caption{Equivalent conditions in the $\mathbb Z_2\times\mathbb Z_4$-model.}
\label{tabla2}
      \end{table}

    Conversely, we assume that the scalars $a_{\alpha,\beta},b_\alpha^{(i)}\in\mathbb F^\times$ satisfy all the equations in \eqref{paciencia} (and hence they satisfy \eqref{todas} too) and let us prove that then $\mcL$ is a Lie algebra. First, it is not difficult to check case by case that $[x,y]=-[y,x]$ if $x,y\in\mcL_{\alpha}$ for $\alpha\in\ZZ_2\times\ZZ_4$, as in the above four models, 
    so that the skew-symmetric extension is well defined. Alternatively, those computations can be skipped if we recall the existence of a solution making $\mcL$ a Lie algebra, since at least there must be a description of the exceptional Lie algebra $\e_8$ with the products as in \eqref{eq_Z4Z2}.    
    That is, we once again follow    the lines of Remark~\ref{re_haysol}. 
Besides, we only have to check Jacobi identities $J({\mcL_\alpha,\mcL_\beta,\mcL_\gamma})=0$ for all $\alpha\prec \beta\prec \gamma$, denoted again by $(J_{\alpha,\beta,\gamma})$. The strategy is much the same as for the model in Section~\ref{seZ6}:  it is enough to check that the equations in \eqref{paciencia} are equivalent to the     list of identities provided in Table~\ref{tabla2}: one/two for each of the possibilities
     $(J_{\alpha,\beta,\gamma})$  for $\alpha\prec \beta\prec \gamma$.

     All these identities are easily derived from \eqref{todas}, by direct substitution. Although this is fairly  straightforward, it is   tedious too: there are 77 cases to have into account (112 equations, most of them redundant, since they are eventually equivalent to the 21 equations  displayed in \eqref{paciencia}). It is therefore advisable to be aware   that not all these checks are   strictly necessary, considering   that  having a relationship between $a_{\alpha,\beta}a_{\alpha+\beta,\gamma}$ and $a_{ \beta,\gamma}a_{\beta+\gamma,\alpha}$  for any $\alpha,\beta,\gamma\in\ZZ_2\times\ZZ_4$ is sufficient  (it is not important if they are equal, opposite, double...), 
     again due to our prior knowledge of the existence of a solution. \smallskip

 The strategy for the uniqueness developed in Remark~\ref{re_simplicidad} is viable here. Alternatively, a suitable isomorphism between $\mcL$ and $\mcL'$, which is a scalar multiple of the identity in each homogeneous component, can be  explicitly given,   for  $\mcL$ and $\mcL'$ graded Lie algebras defined by \eqref{ecumodel_Z24}, \eqref{eq_Z4Z2} and \eqref{paciencia}, for different choices of the 28 variables (in fact, of the 7 free parameters). We think that  adding here the concrete isomorphism would not provide valuable new information. 
\end{proof} \smallskip


 \textbf{Acknowledgment:}  
 Cristina Draper thanks C\'andido Mart\'\i n for the years spent trying to understand Adams' book, which   seem to begin to be fruitful. Also the authors would like  to express him their gratitude   for his   model on $\e_8$ based on the 4 copies of $\slf(3)$, although this pretty version has not been enclosed here (for unification reasons).


\end{document}